\PassOptionsToPackage{shortlabels}{enumitem}
\documentclass[bj,authoryear,noshowframe]{imsart}
\usepackage{imsart}
\usepackage[utf8]{inputenc}
\usepackage[T1]{fontenc}
\usepackage{graphicx}
\usepackage{booktabs}
\usepackage{amsmath,amsthm}
\usepackage{amssymb}
\usepackage{xspace}
\usepackage{bbding}
\usepackage{caption}
\usepackage{placeins} 
\usepackage{pifont}  
\usepackage{bbding}
\usepackage{makecell}

\usepackage{mathtools,mathrsfs,bm}
\usepackage{tabularx,array}
\usepackage[capitalise]{cleveref}

\usepackage[mathscr]{euscript}   

\usepackage{subcaption}

\usepackage[ruled,vlined]{algorithm2e}

\usepackage{tikz}
\usetikzlibrary{positioning,calc,cd,decorations.pathmorphing,shapes,automata,arrows}
\tikzset{rv/.style={circle,inner sep=1pt,fill=color1,draw,font=\sffamily},
lv/.style={circle,inner sep=1pt,fill=gray!50,draw,font=\sffamily},
fv/.style={rectangle,inner sep=1.5pt,fill=gray!20,draw,font=\sffamily},
>=stealth}

\makeatletter
\def\NAT@sort{\z@}
\def\NAT@cmprs{\z@}
\makeatother

\startlocaldefs
\theoremstyle{plain}
\newtheorem{theorem}{Theorem}
\newtheorem{proposition}{Proposition}
\newtheorem{lemma}{Lemma}
\newtheorem{corollary}{Corollary}

\theoremstyle{definition}
\newtheorem{assumption}{Assumption}
\newtheorem{condition}{Condition}
\newtheorem{definition}{Definition}
\newtheorem{remark}{Remark}
\newtheorem{example}{Example}

\crefname{condition}{Condition}{Conditions}
\Crefname{condition}{Condition}{Conditions}

\crefname{algocf}{Algorithm}{Algorithms}
\Crefname{algocf}{Algorithm}{Algorithms}

\newtheorem{conditionalt}{Condition}[condition]

\crefname{conditionalt}{Condition}{Conditions}
\Crefname{conditionalt}{Condition}{Conditions}

\newenvironment{conditionp}[1]{
  \renewcommand\theconditionalt{#1}
  \conditionalt
}{\endconditionalt}

\def\T{{ \mathrm{\scriptscriptstyle T} }}

\newcommand{\Lp}{\mathcal{L}}

\newcommand{\M}{\mathcal{M}}
\newcommand{\Mo}{\M^{\mathrm{o}}}
\newcommand{\Mc}{\M^{\mathrm{c}}}
\newcommand{\Mug}{\Mo_{\textsc{ug}}}
\newcommand{\Mcpdag}{\Mo_{\textsc{cpdag}}}
\newcommand{\Mmpdag}{\Mc_{\textsc{mpdag}}}
\newcommand{\Mdag}{\Mc_{\textsc{dag}}}
\newcommand{\Madmg}{\Mc_{\textsc{admg}}}

\newcommand{\g}{\mathfrak{G}}

\DeclareMathOperator{\Ch}{Ch}

\DeclareMathOperator{\Pa}{Pa}

\DeclareMathOperator{\Sib}{Sib}
\DeclareMathOperator{\Nb}{N}
\DeclareMathOperator{\sk}{sk}
\DeclareMathOperator{\vstr}{v-str}
\DeclareMathOperator{\OP}{OP}

\newcommand{\poly}{\mathrm{poly}}

\newcommand{\du}{\downarrow\uparrow}
\newcommand{\ud}{\uparrow\downarrow}

\newcommand{\indep}{\mathrel{\text{\scalebox{1.07}{$\perp\mkern-10mu\perp$}}}}

\definecolor{applegreen}{rgb}{0.55, 0.71, 0.0}
\definecolor{color1}{RGB}{184, 198, 228}

\renewcommand{\checkmark}{\ding{51}}
\newcommand\crossmark{\ding{55}}

\newcommand{\sV}{\mathrm{V}}
\newcommand{\sA}{\mathrm{A}}
\newcommand{\sB}{\mathrm{B}}
\newcommand{\sC}{\mathrm{C}}
\newcommand{\sE}{\mathrm{E}}
\newcommand{\sF}{\mathrm{F}}
\newcommand{\sL}{\mathrm{L}}
\newcommand{\sI}{\mathrm{I}}
\newcommand{\sW}{\mathrm{W}}
\newcommand{\sU}{\mathrm{U}}
\newcommand{\sP}{\mathrm{P}}
\newcommand{\sS}{\mathrm{S}}
\newcommand{\sT}{\mathrm{T}}

\let \epsilon \varepsilon

\newcommand{\G}{\mathcal{G}}
\newcommand{\B}{\mathcal{B}}
\renewcommand{\H}{\mathcal{H}}
\newcommand{\D}{\mathcal{D}}

\newcommand{\N}{\mathcal{N}}

\newcommand{\Glst}{\G_{\hat{0}}}
\newcommand{\Ggr}{\G_{\hat{1}}}

\newcommand{\coverby}{\lessdot}
\newcommand{\cover}{\gtrdot}
\newcommand{\join}{\vee}
\newcommand{\meet}{\wedge}

\DeclareMathOperator{\E}{\mathbb{E}}

\newcommand*\dd{\mathop{}\!\mathrm{d}}
\DeclareMathOperator{\rank}{\mathrm{rank}}
\DeclareMathOperator{\pseudo}{\mathrm{pseudo}}
\DeclareMathOperator{\pseudorank}{\mathrm{pseudo-rank}}
\DeclareMathOperator{\SHD}{\mathrm{SHD}}

\newcommand{\appendixnumbering}{%
  \renewcommand{\thetheorem}{S\arabic{theorem}}%
  \renewcommand{\thealgocf}{S\arabic{algocf}}%
  \renewcommand{\thelemma}{S\arabic{lemma}}%
  \renewcommand{\theexample}{S\arabic{example}}%
  \renewcommand{\theproperty}{S\arabic{property}}%
  \renewcommand{\theassumption}{S\arabic{assumption}}%
  \renewcommand{\theproposition}{S\arabic{proposition}}%
  \renewcommand{\thecorollary}{S\arabic{corollary}}%
  \renewcommand{\thedefinition}{S\arabic{definition}}%
  \renewcommand{\thefigure}{S\arabic{figure}}%
  \renewcommand{\thetable}{S\arabic{table}}%
}

\endlocaldefs

\begin{document}

\begin{frontmatter}
\title{Model-oriented graph distances via partially ordered sets}
\runtitle{Model-oriented graph distances via posets}

\begin{aug}
\author[A]{\inits{A.}\fnms{Armeen}~\snm{Taeb}\thanksref{t1}\ead[label=e1]{ataeb@uw.edu}\orcid{0000-0002-5647-3160}}
\author[B]{\inits{F.~R.}\fnms{F.~Richard}~\snm{Guo}\thanksref{t1}\ead[label=e2]{ricguo@umich.edu}\orcid{0000-0002-2081-7398}}
\author[C]{\inits{L.}\fnms{Leonard}~\snm{Henckel}\thanksref{t1}\ead[label=e3]{leonard.henckel@ucd.ie}\orcid{0000-0002-1000-3622}}

\runauthor{A. Taeb, F.~R. Guo and L. Henckel}

\address[A]{Department of Statistics, University of Washington\printead[presep={,\ }]{e1}}
\address[B]{Department of Statistics, University of Michigan\printead[presep={,\ }]{e2}}
\address[C]{School of Mathematics and Statistics, University College Dublin\printead[presep={,\ }]{e3}}

\thankstext{t1}{All authors contributed equally to this work.}

\end{aug}

\begin{abstract}
A well-defined distance on the parameter space is key to evaluating estimators, ensuring consistency, and building confidence sets. While there are typically standard distances to adopt in a continuous space, this is not the case for combinatorial parameters such as graphs that represent statistical models. 
Defined on the graphs alone, existing proposals like the structural Hamming distance ignore the structure of the model space and can thus exhibit undesirable behaviors. 
We propose a model-oriented framework for defining the distance between graphs that is applicable across different graph classes. 
Our approach treats each graph as a statistical model and organizes the graphs in a partially ordered set based on model inclusion. 
This induces a neighborhood structure, from which we define the model-oriented distance as the length of a shortest path through neighbors, yielding a metric in the space of graphs. 
We apply this framework to probabilistic undirected graphs, causal directed acyclic graphs, {causal acyclic directed mixed graphs}, probabilistic completed partially directed acyclic graphs, and causal maximally oriented partially directed acyclic graphs.
We analyze theoretical and empirical behaviors of the model-oriented distance. By exploiting the underlying poset structures, we develop algorithms for computing and bounding the proposed distance that scale to moderate-sized graphs. 
Finally, we showcase its utility for quantifying the robustness of adjustment sets to errors in specifying the causal graph. 
\end{abstract}



\end{frontmatter}

\section{Introduction} \label{sec:intro}

Graphs provide an intuitive framework for encoding and reasoning about the dependency structure among a set of variables. 
In statistics, they have become popular in two closely related but different contexts: \emph{probabilistic} modeling and \emph{causal} modeling. 
For the former context, graphs typically encode conditional independence relations among a set of (factual) random variables \citep{lauritzen1996graphical}; for certain scenarios, they can also encode generalized versions of conditional independence and other constraints \citep{evans2016graphs,richardson2023nested} --- we refer to these as \emph{probabilistic graphical models}. 
Among other applications, these graphs lead to sound and efficient algorithms for probabilistic reasoning \citep{lauritzen1988local,lunn2000winbugs}.
For the latter context, a graph models a causal system by encoding the relations of its building blocks, namely counterfactual (as well as factual) random variables; as such, the graph not only describes the system's behavior as things naturally occur, but also its behavior under interventions \citep{wright1934method,spirtes2000causation,pearl2009causality} --- we use the term \emph{causal graphical models} to refer to graphs like this. 
These graphs are used to determine if a causal effect can be identified from observational data \citep{huang2006identifiability,shpitser2006identification}, and if so, to also inform the best way to estimate the effect \citep{rotnitzky2020efficient,henckel2022graphical,guo2023variable}. 
There is also a large literature on learning probabilistic or causal graphs from data, a problem known as structure learning or causal discovery; the reader is referred to  \citet{drton2017structure,glymour2019review} for two surveys.

In whichever context these graphs are used, a key unresolved issue is to define a suitable distance between them. 
Having a distance is essential for several tasks, including benchmarking causal discovery algorithms \citep{scutari2019learns,rios2023benchpressscalableversatileworkflow}, characterizing their statistical performance \citep{Jin2012OptimalityOG,Ke2012CovariateAS,han2016sparse} or finding a consensus graph among a collection of graphs through a Frech\'et mean \citep{Ferguson,wang2025confidence}. 
Other potential applications include constructing ball-like confidence regions for a graph and conducting sensitivity analyses with respect to potential mistakes in the causal graph.
However, defining a suitable distance is challenging for two reasons. 
First, there are many types of graphs, such as undirected graphs (UGs), directed acyclic graphs (DAGs), acyclic directed mixed graphs (ADMGs), completed partially directed acyclic graphs (CPDAGs), and maximally oriented partially directed acyclic graphs (MPDAGs), each with its own edge types and rules on how edges may be configured. Second, each type of graphs relies on a set of \emph{semantics} (sometimes called \emph{Markov properties}) to be interpreted as probabilistic or causal models, and a good distance between two graphs should reflect how close they are in terms of the models they represent. 

\subsection{Related work}
Several distances have been proposed in the literature \citep{tsamardinos2006max,peters2015structural,viinikka2018intersection,wahl2025metric}, which we divide into two categories --- \emph{model-agnostic} and \emph{task-oriented} distances. 

The most popular distance is the structural Hamming distance (SHD) \citep{tsamardinos2006max}. 
It counts the number of edges that differ between two graphs with various proposals for how to weigh different types of edge mismatches \citep{perrier2008finding}. 
Irrespective of the weighting, the SHD is computationally cheap, flexible with respect to the graph class, and guaranteed to define a metric.
It is, however, \emph{model-agnostic} --- it disregards the specific models the graphs represent and is solely defined on the graphs. 
This results in undesirable behaviors. 
First, for a pair of graphs, the SHD provides a fixed answer regardless of whether the graphs are interpreted probabilistically or causally. 
Second, the SHD does not depend on the set of graphs in consideration. It is agnostic to constraints posed by the graph type or domain specific considerations (e.g., one may wish to consider acyclic graphs). 
Instead, the SHD treats each edge individually, which ignores that not all edge configurations lead to an instance in the graph class under consideration. We illustrate these issues with an example.

\begin{example} \label{ex:SHD}
Consider the four CPDAGs (see \cref{sec:graph_background}) listed in \cref{fig:three_cpdags}.
Each graph represents a set of distributions over $n+1$ random variables that obey the probabilistic constraints it encodes, where a missing edge between two variables signifies the conditional independence between them given another set of variables. 
Based on the edges (ignoring the difference between directed and undirected edges for now), we can see $\G_1 \succ \G_2 \succ \G_3 \succ \G_4$ in terms of the \emph{models} they represent. 
\cref{tab:ex-SHD} shows the distance of graphs $\G_2, \G_3, \G_4$ relative to $\G_1$: 
the first two columns list the two versions of the SHD (see \cref{def:SHD}) and the third column is the distance we propose in this paper. 
Our distance reflects what we expect logically: relative to $\G_1$, the graph $\G_2$ is the closest, followed by $\G_3$, and lastly $\G_4$. 
In contrast, the SHD incorrectly deems $\G_3$ to be the closest (can be much closer than $\G_2$ and $\G_4$) and fails to capture the fact that $\G_2$ is strictly closer than $\G_4$. 
This is due to the fact that the SHD wrongly assumes that undirecting any edge in $\G_2$ leads to a valid CPDAG that lies in between $\G_1$ and $\G_2$, thereby overestimating the distance between $\G_1$ and $\G_2$. 

Moreover, our distance agrees with the Bayesian Information Criterion (BIC) in the last column: we generate data under $\G_1$ from a Gaussian structural equation model on 11 variables with unit edge weights and noise variances, draw 100 datasets of size 1000, then fit models and estimate $\E[\mathrm{BIC}(\hat{\G}_i)-\mathrm{BIC}(\hat{\G}_1)]$. Here, a lower score indicates a better model. The discordance between BIC and SHD suggests that using SHD to benchmark score-based causal discovery algorithms \citep[see][\S 4]{drton2017structure} based on BIC or similar criteria, such as greedy equivalence search \citep{chickering2002optimal}, can be misleading.

\begin{table}[b]
\centering
\caption{Distances and BIC ($\pm$ standard error) relative to $\G_1$ for graphs in \cref{fig:three_cpdags} (assuming $n$ is even)}
\label{tab:ex-SHD}
\scalebox{1.}{%
\begin{tabular}{ccccc} 
\toprule
Graph & $\SHD_1$ & $\SHD_2$ & Model-oriented distance & BIC under $n=11$ and sample size $1000$  \\
\midrule
$\G_2$ & $n+1$ & $n+2$ & $1$ & $690 \pm 5$  \\
$\G_3$ & $2$ & $4$ & $2$ & $1084 \pm 6$\\
$\G_4$ & $n/2 + 1$ & $n + 2$ & $n/2+1$ &  $10159 \pm 20$ \\
\bottomrule
\end{tabular}}
\end{table}
\end{example}

\begin{figure}[!t]
    \centering
    \begin{minipage}[t]{0.45\columnwidth}
        \centering
         {\includegraphics[width=.6\linewidth]{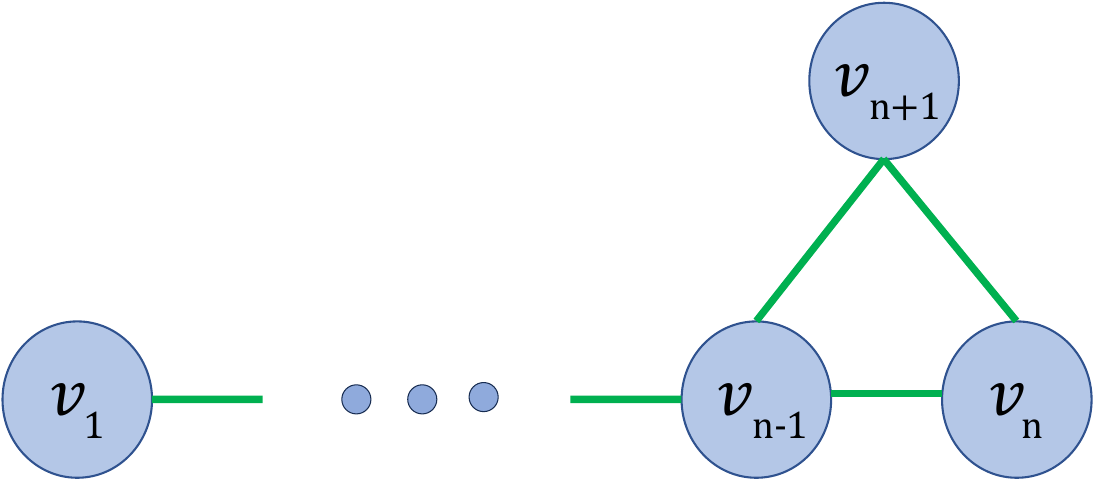}}\\
  (a) $\mathcal{G}_1$
  
\vspace{0.1in}         
       {\includegraphics[width=.6\linewidth]{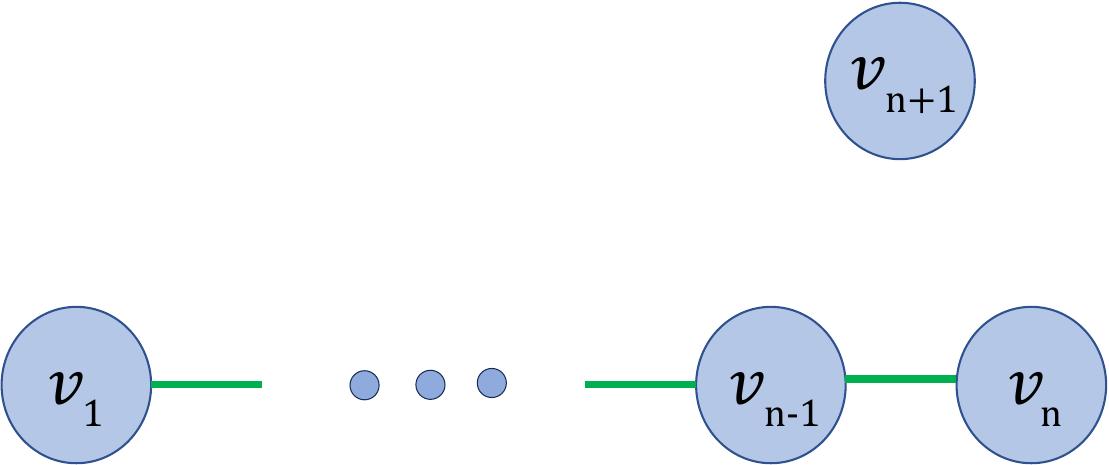}}\\
  (c) $\mathcal{G}_3$
    \end{minipage}
    \hspace{0.05\columnwidth}
    \begin{minipage}[t]{0.45\columnwidth}
        \centering
        {\includegraphics[width=.6\linewidth]{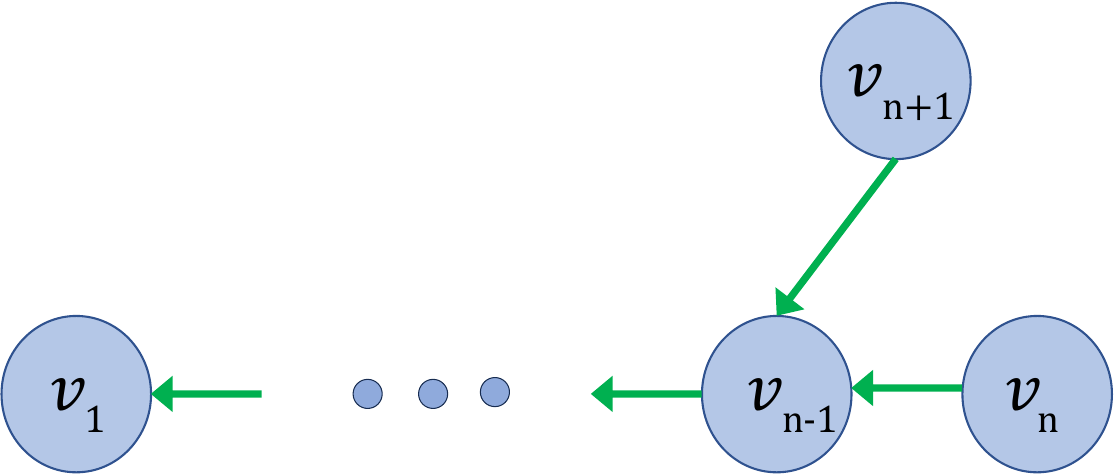}}\\
  (b) $\mathcal{G}_2$

  \vspace{0.1in}         
        {\includegraphics[width=0.9\linewidth]{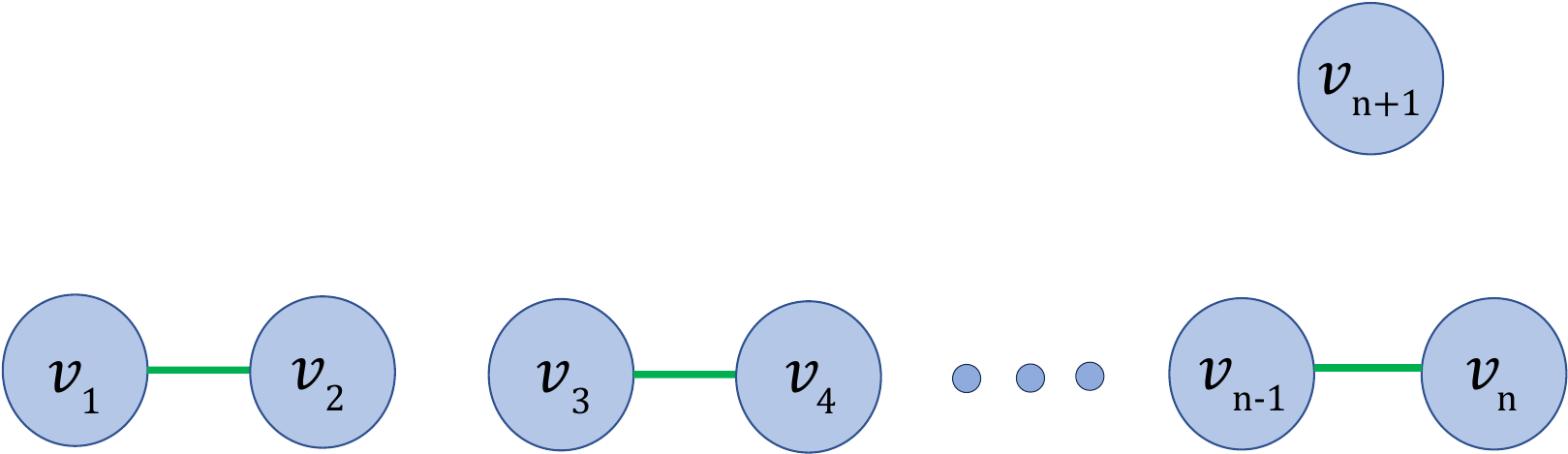}}\\
  (d) $\mathcal{G}_4$
    \end{minipage}
    \caption{CPDAGs $\G_1 \succ \G_2 \succ \G_3 \succ \G_4$ in terms of the models they represent.}
    \label{fig:three_cpdags}
\end{figure}


To address the limitations of SHD, \citet{peters2015structural} proposed an alternative, \emph{task-oriented} approach, where the distance between two graphs is measured according to a concrete task they serve.
For causal graphs, the central task is effect identification, i.e., expressing a causal effect as a functional of the observed distribution valid under any system compatible with the graph. 
Concretely, \citet{peters2015structural} introduced the Structural Intervention Distance (SID), which counts the number of causal effects that would be wrongly inferred by adjusting for the parents of the treatment in one graph, if the true data generating process corresponds to the other graph. \citet{henckel2024gadjid} generalized this approach into a broader class of \emph{adjustment identification distances}. Finally, \citet{wahl2025metric} extended this approach to probabilistic graph classes by proposing a family of \emph{separation distances}, which are counts of conditional independence statements that hold in one graph but not the other. The specific statements considered are determined by a separation strategy, designed to reflect common procedures for selecting separating conditioning sets such as parents. Task-oriented distances are generally fast to compute and useful for benchmarking.
They are, however, not metrics: they allow zero distance between distinct graphs, lack symmetry, and violate the triangle inequality. We show in \cref{app:SID} that a symmetrized version of SID also fails the triangle inequality.  The fact that they are not metrics limits their utility as analytical tools or for defining well-behaved neighborhoods.

\citet{wahl2025metric} also proposed the s/c-metric, a weighted symmetric difference between the sets of all conditional independencies implied by each graph. While very expensive to compute, it is a proper metric and, since probabilistic graphical models are effectively collections of conditional independence statements (CIs), it is naturally model-oriented. It is, however, only applicable to probabilistic models (e.g., those represented by CPDAGs) and is not transparently generalizable to causal models (e.g., those represented by MPDAGs). It also does not adapt to the graph class under consideration (e.g., CPDAGs vs UGs) and ignores which sets of CIs \emph{correspond to} graphs in the class. This leads to some idiosyncrasies, which we discuss in \cref{app:sep-distance}. 

We provide an overview of the characteristics of these distances in \cref{tab:distance_comparison}.


\subsection{Main contributions and organization of the paper}  \label{sec:contrib}
This paper makes the following main contributions:
\begin{enumerate}
\item Motivated by the use of graphs as objects for representing and reasoning about \emph{models} in statistics, we propose a formal framework for defining distance between graphs \emph{through} the models they represent. Our distance is applicable to almost any $(\g, \M)$, where $\g$ is a class of graphs and $\M$ is the \emph{semantics} that injectively maps a graph to a model. 

\item By design, our distance is adaptive to both $\M$ and $\g$. It differentiates between \emph{probabilistic models} (distributions of factual variables) and \emph{causal models} (distributions of counterfactual as well as factual variables). When $\g$ is chosen to be a class of simpler graphs, such as polytrees considered in \cref{sec:polytree-cpdags,sec:polytree-mpdags}, our distance simplifies.

\item We build our distance using the \emph{model-oriented partially ordered set (poset)} associated with $(\g, \M)$, which orders $\G \preceq \G'$ if $\M(\G) \subseteq \M(\G')$ for graphs $\G, \G' \in \g$. Using a notion of neighbors intrinsic to the poset, called \emph{covering relations}, our distance is defined as the shortest-path distance between graphs through neighbors. 
By definition, our distance is a \emph{metric}, which is essential to many statistical applications, such as bounding the error as in \cref{eq:d-triangle} and the sensitivity analysis conducted in \cref{sec:robust-adj}. 
Not all distances in the literature are metrics; see \cref{tab:distance_comparison}.

\item Our approach can be seen as a principled generalization of a related approach, which defines a distance between combinatorial objects, such as rankings \citep{kendall1938new} and clusterings \citep{meila2005comparing} considered in the literature, as the minimum number of \emph{basic operations} to transform one into the other. 
Instead of relying on any ad-hoc specification, the basic operations in our approach are \emph{defined} by the poset in terms of the covering relations. 
Our study of neighborhoods also connects to score-based causal discovery, where an algorithm moves through neighbors to search for a score-minimizing graph; this relationship is further explained in \cref{sec:cpdags}. 

\item Central to our development are several structural properties of model-oriented posets, which we study for a few important graph classes; see \cref{tab:summary} for a summary. Among these, for UGs, DAGs and ADMGs, we show that their posets are simple enough that our distance coincides with \emph{certain versions} of the structural Hamming distance making it cheap to compute. 

\item For CPDAGs and MPDAGs, the corresponding posets are more complex and our distance is markedly different from SHDs. We develop A* algorithms for computing our distance by pruning the search space using informative lower and upper bounds. For CPDAGs, our current implementation scales to graphs with as many as 13 vertices despite $|\g| > 10^{30}$. Numerical findings show that SHDs tend to overestimate the distance between graphs. 

\end{enumerate}

\begin{table}[t]
\caption{Comparison of our model-oriented distance to the Structural Hamming Distance (SHD),
 Structural Intervention Distance (SID), 
Separation Distances (SD) and s/c-Metric.} 
\label{tab:distance_comparison}
\scalebox{0.9}{%
\begin{tabular}{rcccccc}
\toprule
\textbf{Distance} &
\textbf{Metric} &
\textbf{Model-oriented} &
\makecell{\textbf{Suitable for}\\ \textbf{prob. graphs}} &
\makecell{\textbf{Suitable for}\\ \textbf{causal graphs}} &
\makecell{\textbf{Differentiating}\\ \textbf{causal vs.~prob.}} & \makecell{\textbf{Computationally}\\\textbf{cheap}} \\
\midrule
SHD & \checkmark  & \crossmark & \checkmark & \checkmark & \crossmark & \checkmark \\
SID & \crossmark & \crossmark & \crossmark & \checkmark & NA & \checkmark \\
SD  & \crossmark & \crossmark & \checkmark & \crossmark & NA &\checkmark \\
s/c-metric  & \checkmark & \checkmark & \checkmark & \crossmark & NA &\crossmark \\
Ours & \checkmark & \checkmark & \checkmark & \checkmark & \checkmark &\checkmark / \crossmark\\
\bottomrule
\end{tabular}%
}
\end{table}

The rest of the paper is organized as follows. 
In \cref{sec:background}, we introduce notation and background on graphs and the associated probabilistic or causal models. 
In \cref{sec:framework}, we define the model-oriented poset and distance, and study generic structural properties that facilitate distance computation. 
Then, these properties are studied in detail for a few graph classes in \cref{sec:m_distance_Graphs}. 
In \cref{sec:A-star}, we develop a generic A* algorithm for computing our distance, apply it to CPDAGs and present numerical results. 
\cref{sec:robust-adj} briefly describes an application to covariate adjustment in causal inference subject to graph misspecification. 
Finally, we conclude our paper with a discussion in \cref{sec:discussion}. 
All proofs and additional results are deferred to the Supplementary Material.

\section{Background} \label{sec:background}

\subsection{Notation} \label{sec:notation}
We use Roman capital letters (e.g., $\sV, \sB$) to denote a set and lowercase letters (e.g., $v_1, b_2$) to denote its elements. 
For a set $\sA$, we use $2^{\sA}$ to denote its power set. 
For two sets $\sA$ and $\sB$, we use $\sA \triangle \sB := (\sA \setminus \sB) \cup (\sB \setminus \sA)$ to denote their symmetric difference.
We use calligraphic letters (e.g., $\G, \D$) to denote graphs and italic capital letters (e.g., $X, Y$) to denote random variables or vectors. 
Symbol $\g$ is reserved for a class of graphs, which may vary from section to section. 
For random variables or vectors $X,Y,W$, we use $X \indep Y \mid W$ to denote that $X$ and $Y$ are conditionally independent given $W$. 

\subsection{Graphs and related concepts} \label{sec:graph_background}
We consider graphs with vertices representing random variables and edges between them representing probabilistic or causal dependencies between these variables. 
Throughout, we only consider graphs over a \emph{finite} set of vertices. 
Hence, a vertex set $\sV$ can be identified with $\{1,\dots, |\sV|\}$ for convenience. 
We write a graph over $\sV$ with $k \geq 1$ edge types as $\mathcal{G} = (\sV, \sE^1, \dots, \sE^k)$, where each edge type $\sE^i \subseteq (\sV \times \sV)$ consists of unordered or ordered pairs of \emph{distinct} vertices (to disallow self-loops) depending on the edge type: unordered for undirected ($-$) or bidirected ($\leftrightarrow$) edges, and ordered for directed ($\rightarrow$) edges. 

A graph $\mathcal{G}_1 = (\sV_1, \sE_1^1, \dots, \sE_1^k)$ is a \emph{subgraph} of $\mathcal{G}_2 = (\sV_2,\sE_2^1, \dots, \sE_2^k)$ if $\sV_1 \subseteq \sV_2$ and $\sE_1^i \subseteq \sE_2^i$ holds for each edge type $i$. 
For $\sV' \subseteq \sV$, the subgraph of $\G$ induced by $\sV'$ is $\G_{\sV'} := (\sV', \sE^1 \cap (\sV'\times\sV'), \dots, \sE^k \cap (\sV'\times\sV'))$. 
A \emph{cycle} is a sequence $v_1,\dots,v_m$ ($m \ge 3$) with $v_m = v_1$ such that every $(v_{i}, v_{i+1})$ is connected by an edge. 
A \emph{directed cycle} is a cycle of the form $v_1 \to \dots \to v_m$.
A \emph{partially directed cycle} is a cycle $v_1, \dots, v_m$ where either $v_i - v_{i+1}$ or $v_i \to v_{i+1}$ holds for each $i$ and $v_i \to v_{i+1}$ holds for some $i$.
A graph is \emph{chordal} if every cycle of length $\ge 4$ has a chord (an edge connecting two non-consecutive vertices). 
A graph is \emph{simple} if between every pair of vertices there is at most one edge (of any type). 
A \emph{partially directed acyclic graph} (PDAG) is a simple graph with directed and undirected edges and no directed cycles. 
For a PDAG, we use the following standard notation for parents, children, siblings and neighbors:
$\Pa(v):=\{w: w \to v\}, \; \Ch(v):=\{w: w \leftarrow v\}, \; \Sib(v):=\{w: w - v\}, \; \Nb(v) := \Sib(v) \cup \Pa(v) \cup \Ch(v).$
The graph obtained by replacing all directed edges with undirected edges is called the \emph{skeleton}; if it is acyclic, the PDAG is called a \emph{polytree}.
We now describe the graph types that we focus on in this  paper.

\begin{description}
\item[Undirected graph (UG)] A simple graph consisting of undirected edges. 
An undirected graph $\G$ over $\sV$ defines a probabilistic model for the random vector $(X_v: v \in \sV)$ through the semantics known as the \emph{Markov properties} of UGs.
In the literature, there are four such semantics: factorization, as well as the global, local and pairwise Markov properties, which are not equivalent in general; see \citet[\S3.2.1]{lauritzen1996graphical}.
The choice is not essential for our approach. For simplicity, in \cref{sec:models} we will pose a positivity condition so that these four semantics are equivalent. 
For disjoint sets $\sA,\sB,\sC \subset \sV$, we use $\sA \indep \sB \mid \sC$ in $\G$ to denote that $\sA$ and $\sB$ are separated by $\sC$ in the graph.

\item[Directed acyclic graph (DAG)] A PDAG with only directed edges. 
As explained in \cref{sec:models}, a DAG can represent either a probabilistic model or a causal model. 
For the former, a DAG encodes a set of conditional independence (CI) constraints. 
Two DAGs defined on the same vertex set may encode the same set of CIs. If so, they are called \emph{Markov equivalent}; a \emph{Markov equivalence class} consists of the set of DAGs that are Markov equivalent. 
Two DAGs are Markov equivalent if and only if they share the same skeleton  and the same v-structures ($v_1 \rightarrow v_2 \leftarrow v_3$ with no edge between $v_1,v_3$ is called a v-structure at $v_2$) \citep{Verma1990EquivalenceAS,andersson1997characterization}; see \cref{fig:dags}(a) for an example. 
For a DAG $\D$ over vertex set $\sV$ and disjoint sets $\sA,\sB,\sC \subset \sV$, we use $\sA \indep_{d} \sB \mid \sC$ in $\D$ to denote that $\sA$ and $\sB$ are \emph{$d$-separated} by $\sC$ in the graph; see, e.g., \citet[\S 3.2.2]{lauritzen1996graphical}.

\item[Completed partially directed acyclic graph (CPDAG)] 
A PDAG representing a Markov equivalence class of DAGs, constructed as follows. First, let $\G$ have the same skeleton as any (and hence every) DAG in the class. Then, orient an edge if it has the same direction in every DAG in the class; otherwise, leave it undirected. See \cref{fig:dags}(a) for an example. We write $[\G]$ for the Markov equivalence class represented by $\G$. Not every PDAG is a valid CPDAG: it must satisfy the structural constraints in \cref{sec:valid-cpdag}, such as not containing the forbidden structures shown in \cref{fig:dags}(b).


\item[Maximally oriented partially directed acyclic graph (MPDAG)] A PDAG for representing a set of Markov equivalent DAGs subject to certain \emph{background knowledge} about the orientation of one or more edges \citep{hauser2012characterization,Perkovic2019IdentifyingCE,jointly}. For the example in \cref{fig:dags}(a), while the Markov equivalence class represented by the CPDAG consists of 3 DAGs, among them, the two DAGs subject to $v_1 \leftarrow v_2$ are further represented by an MPDAG. 
For an MPDAG $\G$, we similarly use $[\G]$ to denote the set of DAGs it represents. 
By definition, a CPDAG is an MPDAG (without additional background knowledge) and a DAG is also an MPDAG (with full background knowledge).
Not every PDAG is a valid MPDAG: it has to obey the structural constraints listed in \cref{sec:valid-mpdag}.
In \cref{sec:models} and afterwards, we assign different semantics to CPDAGs and MPDAGs: for discussing probabilistic models associated with DAGs, it suffices to use CPDAGs only; for discussing the causal model associated with a DAG from a set of candidate DAGs that are Markov equivalent (but causally distinct), we use MPDAGs instead.

\item[Acyclic directed mixed graph (ADMG)] A graph consisting of directed and bidirected edges that has no directed cycles \citep{richardson2003markov}. Contrary to the previous graph types, an ADMG need not be simple. 
A DAG is also an ADMG. 
For a DAG $\G$ over a vertex set $\sV = \sW \cup \sU$, partitioned into observed vertices $\sW$ and latent vertices $\sU$, a corresponding ADMG $\G(\sW)$ over $\sW$ can be constructed via \emph{latent projection} \citep{pearl1991inferred} to represent the model (both probabilistic and causal) induced on the observed vertices $\sW$. If $\sU = \{u\}$ is a singleton, the latent projection $\G(\sW)$ is constructed as follows. Start with $\G_{\sW}$, namely the subgraph induced by $\sW$, and then add the following edges, if they are not already present, to $\G_{\sW}$:
\[ \begin{Bmatrix} w_1 \leftarrow u \rightarrow w_2 \\ 
w_1 \leftrightarrow u \rightarrow w_2 \\
w_1 \rightarrow u \rightarrow w_2 \end{Bmatrix} \text{ in $\G$} \quad \implies \quad \text{add } 
\begin{Bmatrix} w_1 \leftrightarrow w_2 \\ w_1 \leftrightarrow w_2 \\ w_1 \rightarrow w_2 \end{Bmatrix}. \]
If $|\sU| > 1$, one can show that $\G(\sW)$ is uniquely constructed by iteratively projecting out the vertices in $\sU$ in any order \citep{evans2018markov}. See \cref{fig:dags}(c) for an example. 
Like DAGs, an ADMG can represent both a probabilistic model and a causal model, specifically a model of the observed variables while allowing for latent variables. As probabilistic models, an ADMG is associated with a nested Markov model that posits generalized conditional independence constraints \citep{richardson2023nested}; however, the Markov equivalence classes of such models still remain open. For this reason, we only consider the causal models associated with ADMGs. 
\end{description}

\begin{figure}[htbp]
\subfloat[]{\includegraphics[width=.5\columnwidth]{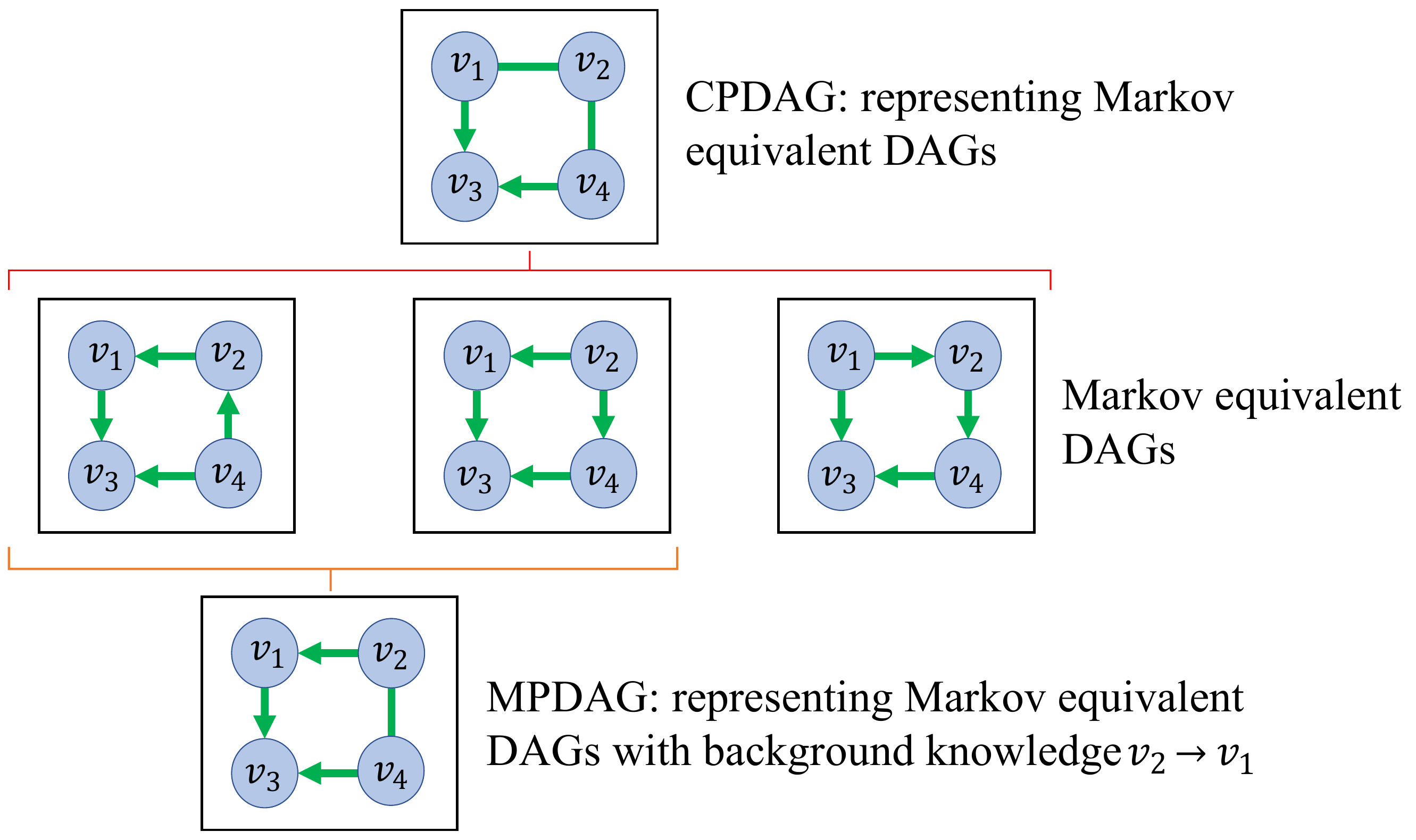}}
\hspace{2em}
\subfloat[]{\raisebox{2em}{\includegraphics[width=.2\columnwidth]{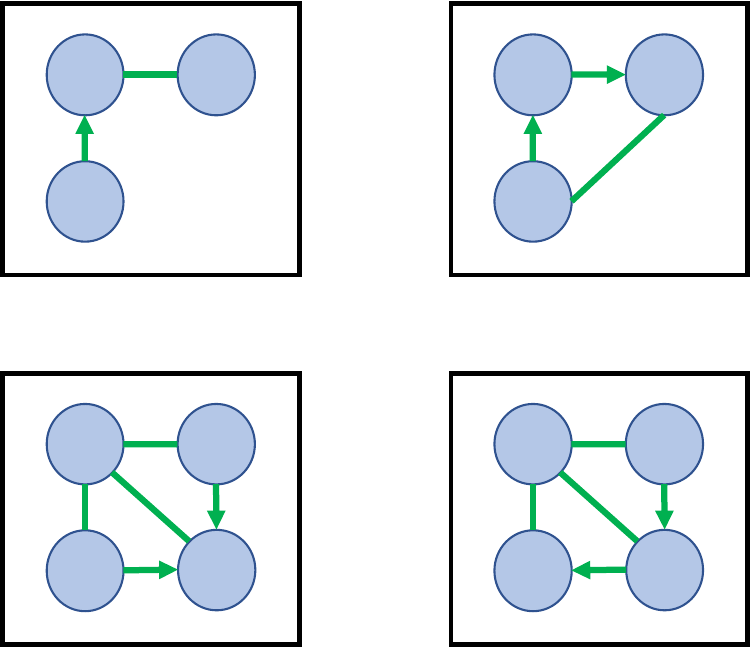}}}
\hspace{4em}
\subfloat[]{\includegraphics[width=.14\columnwidth]{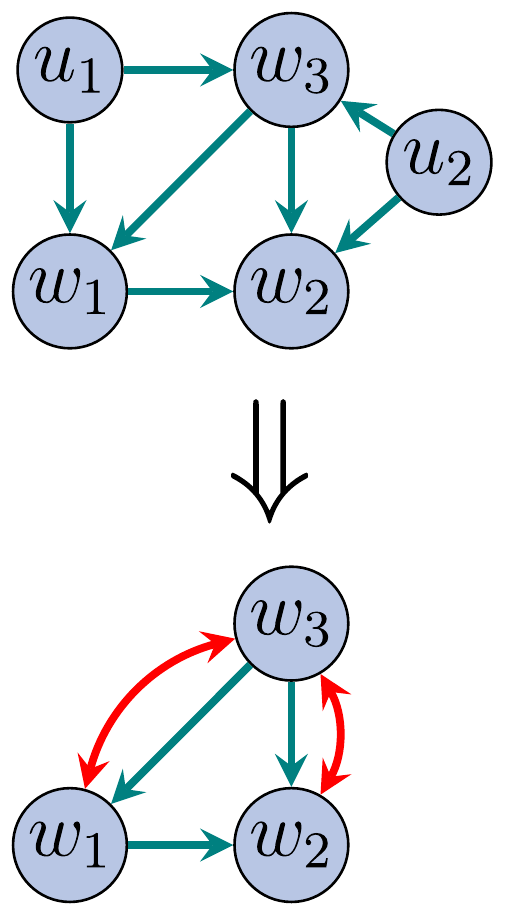}}
\caption{(a) The Markov equivalence class consisting of three DAGs is represented by the CPDAG at the top. 
Among them, two DAGs with background knowledge $v_2 \to v_1$ are further represented by the MPDAG at the bottom.
(b) Structures that are forbidden from a valid MPDAG; see also \cref{sec:valid-mpdag}. 
(c) An ADMG at the bottom formed by projecting out $\{u_1,u_2\}$ from the DAG above. }
\label{fig:dags}
\end{figure}

\subsection{Probabilistic and causal models associated with graphs} \label{sec:models}
For the five types of graphs we consider, we associate with each a \emph{semantic} that maps a graph to either a probabilistic or causal model: a probabilistic model is a set of distributions of the factual variables and a causal model is a set of distributions of the counterfactual variables. 

\subsubsection{Probabilistic models: UGs and CPDAGs} \label{sec:prob-models}
UGs and CPDAGs represent probabilistic models for a random vector $(X_v: v \in \sV)$. 
Suppose $X$ takes values in a product measurable space $\mathcal{X} = \prod_{v \in \sV} \mathcal{X}_v$. Let $\mu = \prod_{v \in \sV} \mu_v$ be a product measure over $\mathcal{X}$. 
Let $P$ be the law of $X$. 
We use the nonparametric model 
\begin{equation} \label{eq:mo}
\Mo := \left\{P: \text{$P$ is absolutely continuous with respect to $\mu$} \right\}
\end{equation}
as our ground set (here `$\mathrm{o}$' stands for observational) so all the models are \emph{subsets} of $\Mo$. 
For $P \in \Mo$, we use $p := \dd P / \dd \mu$ to denote its density. 

For UGs, we pose an additional positivity condition on the ground set such that all four Markov properties, namely factorization as well as the global, local
and pairwise Markov properties, lead to the same model definition \citep[\S3.2.1]{lauritzen1996graphical}. 
Concretely, for UGs let the ground set be $\Mo_{+} := \left\{P \in \Mo: p > 0 \text{ $\mu$-a.e.} \right\}$. 
An undirected graph $\G$ represents the model 
\begin{equation} \label{eq:m-ug}
\Mug(\G) := \big\{P \in \Mo_{+}: \sA, \sB, \sC \subset \sV \text{  disjoint},\; \sA \indep \sB \mid \sC \text{ in } \G  \implies X_{\sA} \indep X_{\sB} \mid X_{\sC} \text{ under $P$} \big\}.
\end{equation}

A CPDAG represents the probabilistic model $\Mcpdag(\mathcal{G})$ consisting of all distributions that factorize according to any (and hence every) DAG $\D \in [\G]$. 
That is, take any $\D \in [\G]$ and we can write $p(x) = \prod_{v\in \sV} p_{v}(x_v|x_{{\text{pa}_{\mathcal{D}}(v)}})$, where $p_v$ is a conditional density function. 
This model can be equivalently described in terms of the global or local Markov property associated with any $\D \in [\G]$; see \citet[\S 3.2.2]{lauritzen1996graphical}.
Since this holds without requiring positivity, we choose $\Mo$ in \cref{eq:mo} as the ground set and define the model through the global Markov property: for any $\D \in [\G]$,
\begin{equation} \label{eq:m-cpdag}
\Mcpdag(\mathcal{G}) := \big\{P \in \Mo: \sA, \sB, \sC \subset \sV \text{ disjoint}, \; \sA \indep_d \sB \mid \sC \text{ in $\D$} 
\implies X_{\sA} \indep X_{\sB} \mid X_{\sC} \text{ under $P$} \big\}.
\end{equation}

\subsubsection{Causal models: ADMGs, DAGs and MPDAGs} \label{sec:causal-models}
We follow \cite{zhao2025statistical} in defining causal models associated with DAGs and MPDAGs. 
Consider a causal system that underlies the observed factual random vector $X = (X_v: v \in \sV)$, which takes values in a product space $\mathcal{X} \equiv \mathcal{X}_{\sV}$ equipped with a suitable product $\sigma$-algebra. 
Here we use notation $\mathcal{X}_{\sI} := \prod_{v \in \sI} \mathcal{X}_v$ for every $\sI \subseteq \sV$. 
A causal model is a set of joint distributions of counterfactual variables as well as factual (i.e., observed) variables, both of which are part of the \emph{potential outcome schedule}: 
\[ X(\cdot) := \left(X_{v}(x_{\sI}): v \in \sV, \sI \subseteq \sV, x_{\sI} \in \mathcal{X}_{\sI} \right), \]
where $X_v(x_{\sI})$ is the potential outcome of variable $v$ under the intervention that imposes value $x_{\sI}$ on $X_{\sI}$. 
When $\sI = \sV$, we simply write $X_v(x_{\sV})$ as $X_v(x)$. 
In the special case when $\sI = \emptyset$, the margin $(X_v(\emptyset): v \in \sV)$ consists of all the factual variables, whose law is called the \emph{observed distribution}. 
The potential outcome schedule is a stochastic process that takes values in $\mathcal{X}^{\prod_{\sI \subseteq \sV} \mathcal{X}_{\sI}}$, which we assume is equipped with a suitable $\sigma$-algebra to form a measurable space. 
We use $P$ to denote the law of $X(\cdot)$, i.e., a probability measure on the measurable space. 
We assume $P$ admits regular conditional probability distributions that appear in the following definition; see \citet[Theorem 8.5]{kallenberg2021foundations} for regularity conditions on the measurable space for ensuring so. 
Similar to the definition of $\Mo$, we define the ground causal model to be those distributions that maintain \emph{consistency}:
\begin{multline} \label{eq:consistency}
\Mc:=\big\{P: P \left(X(x_{\sI}, x_{\sI'}) = X(x_{\sI}) \mid X_{\sI'}(x_{\sI}) = x_{\sI'} \right)=1, \text{ for all disjoint $\sI, \sI' \subseteq \sV$ and each $x \in \mathcal{X}$} \big\}, 
\end{multline}
where `$\mathrm{c}$' stands for causal. 
In words, in a world that imposes value $x_{\sI}$ on $X_{\sI}$, if we observe $X_{\sI'}$ to be $x_{\sI'}$, then the observed $X$ in that world, denoted by $X(x_{\sI})$, equals $X(x_{\sI}, x_{\sI'})$. Every causal model that we discuss next will be a subset of $\Mc$. 

We first introduce the causal models associated with ADMGs and those for DAGs will follow as a special case. Let $\G$ be an ADMG over vertex set $\sV$. 
For $\sA, \sB \subset \sV$, we write $\sA \not \leftrightarrow \sB$ in $\G$ if there is no bidirected edge between any vertex in $\sA$ and any vertex in $\sB$. 
We define 
\begin{equation} \label{eq:m-admg}
\Madmg(\G):= \left\{P \in \Mc: \text{$P$ satisfies (i) and (ii) with respect to $\G$} \right\},\,\text{where}
\end{equation}
\begin{enumerate}[(i)]
\item for every $v \in \sV$, $\sI \subseteq \sV$ and $x \in \mathcal{X}$, it holds that $P\left\{X_v(x_{\sI}) = X_v\left(x_{\Pa(v) \cap \sI}, X_{\Pa(v) \setminus \sI}(x_{\sI}) \right) \right\}=1$; 
\item for disjoint $\sA, \sB \subset \sV$ with $\sA \not \leftrightarrow \sB$, we have $X_{\sA}(x) \indep X_{\sB}(x)$ under $P$ for every $x \in \mathcal{X}$. 
\end{enumerate}
Condition (i) posits that the basic building blocks of the system are the so-called one-step-ahead counterfactuals $\{X_v(x_{\Pa(v)}): v \in \sV, x \in \mathcal{X} \}$ \citep{shpitser2022multivariate}. Further, (ii) states that the one-step-ahead counterfactuals of $i$ and $j$ under the same world $x$ are independent if there is no bidirected edge between $i$ and $j$. The causal model $\Madmg(\G)$ is known as the single-world/FFRCISTG model \citep{robins1986new,richardson2013single} in the literature, which makes fewer assumptions than Pearl's nonparametric structural equation model with independent errors (NPSEM-IE). 

Since a DAG is an ADMG without bidirected edges, $\Madmg$ when specialized to DAGs becomes
\begin{equation} \label{eq:m-dag}
\Mdag(\D):= \left\{P \in \Mc: \text{$P$ satisfies (i) and (iii) with respect to $\D$} \right\},\,\text{where}
\end{equation}
\begin{enumerate}[(i)]
\setcounter{enumi}{2} 
\item for every $x \in \mathcal{X}$, $\left(X_v(x_{\Pa_{\D}(v)}): v \in V \right )$ are mutually independent under $P$.
\end{enumerate}
Accordingly, we define the causal model associated with an MPDAG as a disjunction of DAG models. For any MPDAG $\G$, let
\[\Mmpdag(\G) :=  \bigcup_{\D \in [\G]} \Mdag(\D). \]

\section{Model-oriented distance via partially ordered sets}\label{sec:framework}
Let $\g$ be a collection of graphs 
over a finite vertex set $\sV$. 
Let $\M$ be an associated model map (i.e., \emph{semantics}) that assigns a model to each graph in the class. 
Given $(\g, \M)$, our goal is to define a model-oriented distance $d(\G_s, \G_t)$ for every $\G_s, \G_t \in \g$. 
By definition, $d$ depends on $\g$ and $\M$. 
For example, for $\sV = \{1,2,3,4,5\}$, we can define distances $d_1,d_2$ respectively for $(\g_1, \Mcpdag)$, $(\g_2, \Mcpdag)$ with 
\[ \g_1 := \{\text{all CPDAGs over $\sV$} \}, \quad \g_2 := \{\text{all polytree CPDAGs over $\sV$} \}. \]
In Section~\ref{sec:model_oriented}, we use the model map $\M$ to define the model-oriented poset and subsequently our model-oriented distance between $\G_s$ and $\G_t$. In Section~\ref{sec:simplification}, we describe how the poset structure may enable faster computation of the model-oriented distance. In Section~\ref{sec:edge-by-edge} we discuss cases where the model-oriented distance reduces to SHD.


\subsection{Model-oriented poset and distance}
\label{sec:model_oriented}

\begin{definition}[Poset, cover, connectedness, least element, comparability] \label{defn:basic_prop}
A \emph{poset} (partially ordered set) $\Lp:= (\sL, \preceq)$ is a collection $\sL$ of elements and a relation $\preceq$ that is reflexive ($x \preceq x, ~ \forall x \in \sL$), transitive ($x \preceq y, y \preceq z \Rightarrow x \preceq z, ~ \forall x,y,z \in \sL$), and anti-symmetric ($x \preceq y, y \preceq x \Rightarrow x = y, ~ \forall x,y \in \sL$). {We use notation $x \prec y$ if $x \preceq y$ and $x \neq y$.} An element $y \in \sL$ \emph{covers} $x \in \sL$, {denoted as $x \coverby y$}, if $x \prec y$ and there is no $z \in \sL$ with $x \prec z \prec y$. The poset $\Lp$ is connected if for every $x,y \in \sL$, there exists a sequence of $k$ elements $x=x_1,x_2,\dots,x_k=y$ such that either $x_{i+1} \cover x_i$ or $x_{i+1} \coverby x_i$ for every $i = 1,2,\dots,k-1$. 
An element $x \in \sL$ is a \emph{minimal (maximal) element} if there is no $y \in \sL$ such that $y \prec x$ ($y \succ x$). 
An element $x \in \sL$ is called the \emph{least (greatest) element} if $x \preceq y$ ($x \succeq y$) for all $y \in \sL$. Elements $x, y \in \Lp$ are said to be \emph{comparable} if either $x \preceq y$ or $y \preceq x$; otherwise, they are \emph{incomparable}.
\end{definition}

\begin{condition}[Injectivity] \label{cond:inject}
For any $\G_s, \G_t \in \g$ and $\G_s \neq \G_t$, we have $\M(\G_s) \neq \M(\G_t)$. 
\end{condition}

\begin{condition}[Connectedness] \label{cond:connect}
For every $\G_s,\G_t \in \mathfrak{G}$, there exists a sequence of graphs $\G_s = \G_1, \dots, \G_k = \G_t$ with $k \geq 1$ such that for every $i = 1,\dots,k-1$, either $\mathcal{M}(\G_i) \subseteq \M(\G_{i+1})$ or $\M(\G_{i+1})\subseteq \M(\G_i)$.
\end{condition}

For natural choices of $(\g, \M)$, we argue that \cref{cond:inject,cond:connect} are generally satisfied. In fact we show in \cref{sec:inj} that all the five models $\Mug, \Mdag, \Madmg, \Mcpdag, \Mmpdag$ discussed in this paper satisfy \cref{cond:inject}. 
Connectedness is usually implied by either of the following two conditions. 

\begin{condition}[Least element] \label{cond:least}
There exists a graph $\Glst \in \g$ such that $\M(\Glst) \subseteq \M(\G)$ for every $\G \in \g$. 
\end{condition}
\noindent When it exists, $\Glst$ is the \emph{unique} least element in $\g$ and is usually taken to be the \emph{empty graph}, which represents the smallest model. Again, Condition~\ref{cond:least} is satisfied by all five graph classes considered in this paper with $\Glst$ being the empty graph. This condition also has the following dual.  
\begin{conditionp}{{\ref*{cond:least}$'$}}[Greatest element] 
There exists $\Ggr \in \g$ such that $\M(\Ggr) \supseteq \M(\G)$ for every $\G \in \g$. 
\end{conditionp}

\begin{definition}[Model-oriented poset] \label{def:Lp}
Given a graph class $\g$ and an associated model map $\M$ that satisfies \cref{cond:inject} and \cref{cond:connect}, the corresponding \emph{model-oriented poset} is $\Lp:= (\g, \preceq)$ with `$\preceq$' given by model containment: for $\G_1,\G_2 \in \mathfrak{G}$, we say $\G_1 \preceq \G_2$ if $\M(\G_1) \subseteq \M(\G_2)$. 
\end{definition}

Reflexivity and transitivity of `$\preceq$' hold by definition, while the anti-symmetry follows from the injectivity of $\M(\cdot)$ (\cref{cond:inject}). 
\cref{cond:connect} ensures that $\Lp$ is connected. 
The poset $\Lp$ offers a natural way to organize the graphs $\G \in \mathfrak{G}$ based on the models they represent.  
For example, consider $\g$ to be all the UGs over the vertex set $\{1,2,3\}$ and let $\Mug$ be the model map. 
In this case, the partial order `$\preceq$' can be characterized graphically: for $\G_1, \G_2 \in \g$, one can show that $\G_1 \preceq \G_2$ if and only if $\G_1$ is a subgraph of $\G_2$. 
Further, $\mathcal{G}_2$ covers $\G_1$ if and only if $\mathcal{G}_2$ has \emph{exactly one more edge} than $\mathcal{G}_1$. The least element $\Glst$ in the poset is the empty graph. 

This poset can be visualized with a \emph{Hasse diagram} shown in \cref{fig:hasse}(a): each box is a graph; if $\G_2$ covers $\G_1$, we then place $\G_2$ above $\G_1$ and draw an upward arrow $\G_1  \dashrightarrow \G_2$.
Partial order can be read off from the diagram: $\G_1 \preceq \G_2$ if there exists a directed path (can be of length zero) from $\G_1$ to $\G_2$. 
Hence, the poset $\Lp = (\g, \preceq)$ can be compactly represented by its Hasse diagram, which is a directed acyclic graph over $\g$. 

\begin{definition}[Model-oriented distance] Under \cref{cond:inject,cond:connect}, let $\Lp:= (\g, \preceq)$ be a model-oriented poset. The model-oriented distance between two graphs is the shortest-path distance in the Hasse diagram:
\begin{align*}
& d_{\Lp}(\G_s,\G_t):= \min\left\{\mathrm{length}(p): p \in \sP_{\Lp}(\G_s, \G_t) \right\}, \quad \G_s, \G_t \in \g, \quad \text{where} \\
\sP_{\Lp}(\G_s, \G_t) &:= \left\{p: \;\text{$p$ is a path between $\G_s$ and $\G_t$ in the Hasse diagram of $\Lp$} \right\} \\
&= \left\{(\G_0, \dots, \G_d): \G_0 = \G_s, \G_{d} = \G_t, \,\text{and for $1\leq i \leq d$, $\G_{i-1} \coverby \G_i$ or  $\G_{i-1} \cover \G_{i}$} \right \}.
\end{align*}
\end{definition}


%

The shortest path length is a standard way to measure distance in posets \citep{Monjardet1981MetricsOP, Fishburn1993CombinatoricsAP, Brualdi1995CodesWA}, although its use has primarily been confined to the order theory literature. 
As an alternative interpretation of the model-oriented distance, define
\begin{equation} \label{eqs:neighbor}
\text{Neighbors}(\G) := \{\G': \G' \coverby \G\} \cup \{\G': \G' \cover \G\}, \quad \G \in \g,
\end{equation}
consisting of graphs that either cover or are covered by $\G$. That is, $\G$ and $\G'$ are neighbors if there does not exist another graph whose model is strictly nested between those of $\G$ and $\G'$. Intuitively, $\G'$ is a neighbor of $\G$ if it corresponds to a finest possible refinement or coarsening of the model associated with $\G$. The model-oriented distance between two graphs therefore counts the \emph{minimum number of such elementary refinement or coarsening steps} required to transform one model into the other.

Defining distance through the minimal number of basic operations to transform one object to the other is common in other combinatorial settings. In the context of rankings or permutations, the Kendall's tau distance \citep{kendall1938new} counts the smallest number of adjacent swaps needed to align two rankings. In clustering, the variation of information \citep{meila2005comparing} counts the smallest weighted number of merges or splits needed to align two partitions. In both cases, the distance reflects how many basic, interpretable operations are required to transform one model into another.
Here, we use the notion of neighbor to provide a model-oriented definition of what the basic operations are. 

\begin{proposition}[Metric] Suppose \cref{cond:inject} and \cref{cond:connect} hold. The model-oriented distance is a metric, i.e., it satisfies the following three properties:
(i) symmetry: $d_{\Lp}(\G_s,\G_t) = d_{\Lp}(\G_t,\G_s)$;
(ii) positivity: $d_{\Lp}(\G_s,\G_t)\in [0,\infty)$ and $d_{\Lp}(\G_s,\G_t) = 0$ if and only if $\G_s=\G_t$; and
(iii) triangle inequality:
$d_{\Lp}(\G_s,\G_t)\le d_{\Lp}(\G_s,\G)+d_{\Lp}(\G,\G_t)$ for any $\G\in\mathfrak{G}$.
\end{proposition}
The proof of this proposition follows straightforwardly from the model-oriented distance being a shortest-path distance, which is a metric, defined on an undirected graph with vertex set $\g$ and edges incident to neighbors. Note that \cref{cond:connect} ensures that this graph is connected so that $d_{\Lp}$ is finite. 

The significance of the symmetry and positivity properties is straightforward. 
We underscore the importance of the triangle inequality from both statistical and algorithmic perspectives. Statistically, the triangle inequality is essential for proving properties like the consistency of an estimator. It allows one to bound the total error between an estimate and the truth by introducing an intermediate 'population-level' proxy, ensuring that convergence to the proxy implies convergence to the truth. Specifically:
\begin{equation} \label{eq:d-triangle}
d_{\Lp}(\hat{\G}_n, \G) \leq \underbrace{d_{\Lp}(\hat{\G}_n, \G_{\infty})}_{\text{variability}} + \underbrace{d_{\Lp}(\G_{\infty}, \G)}_{\text{approx. error}}
\end{equation}
where $\hat{\G}_n$ is the estimate, $\G_\infty$ is the intermediate population graph, and $\G$ is the true graph. Algorithmically, the triangle inequality facilitates enumerating graphs. For instance, to enumerate all graphs within distance $k$ of a given graph $\G$, it suffices to expand the neighborhood of all graphs at distance $k-1$ to $\G$. Without this property, the enumeration can be computationally prohibitive. In  
\cref{sec:robust-adj}, we use this important computational property to assess the robustness of adjustment sets to model errors. 




\begin{figure}[htbp]
\centering
\begin{minipage}{\textwidth}
  \centering
  \subfloat[CPDAGs]{\includegraphics[width=.46\columnwidth]{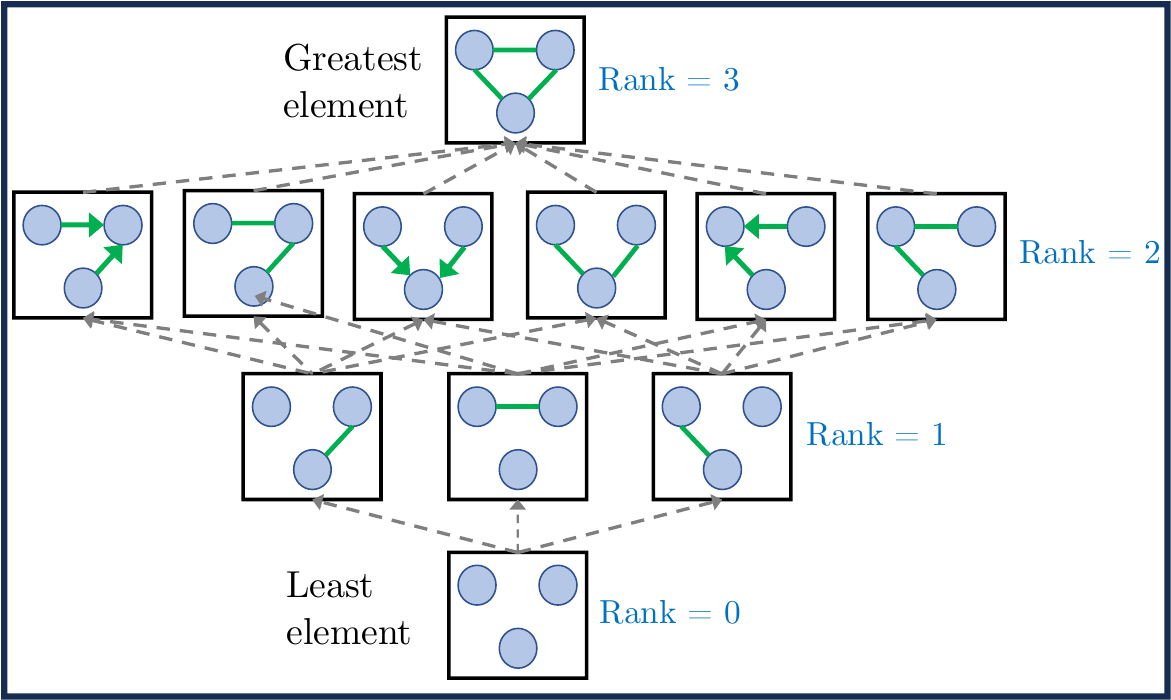}}
  \hfill 
  \subfloat[DAGs]{\includegraphics[width=.505\columnwidth]{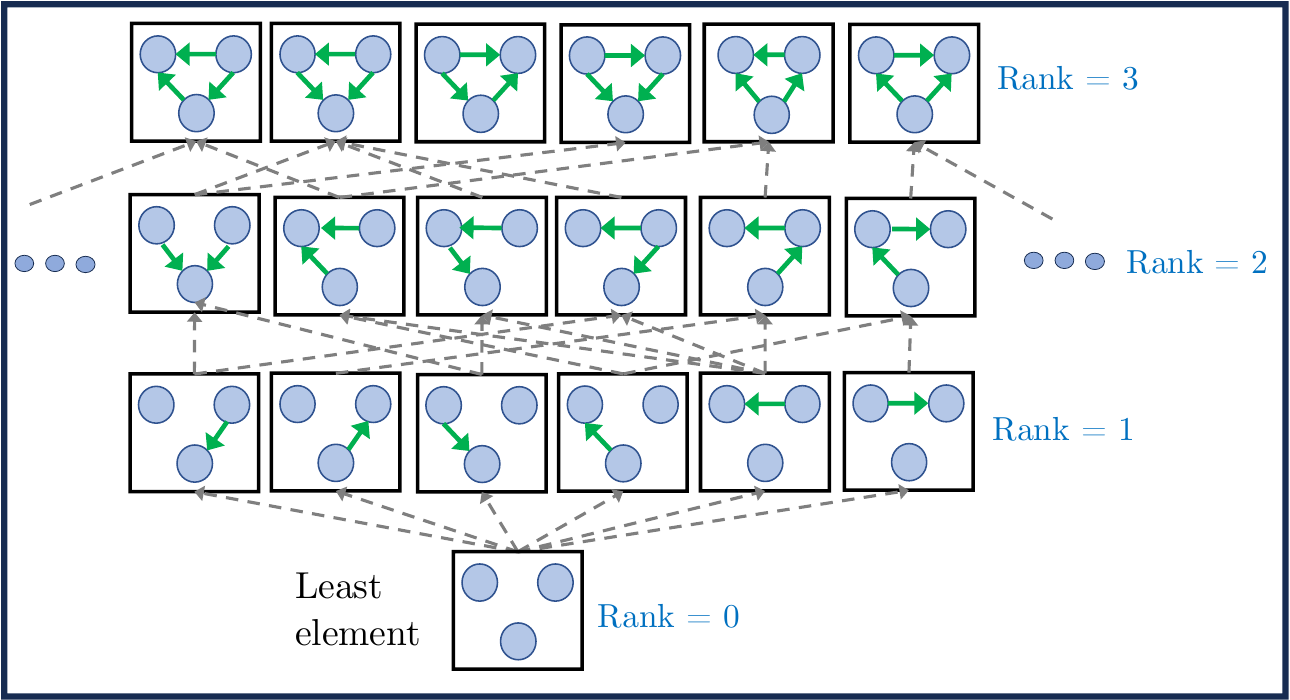}}\\
  \subfloat[UGs]{\includegraphics[width=.46\columnwidth]{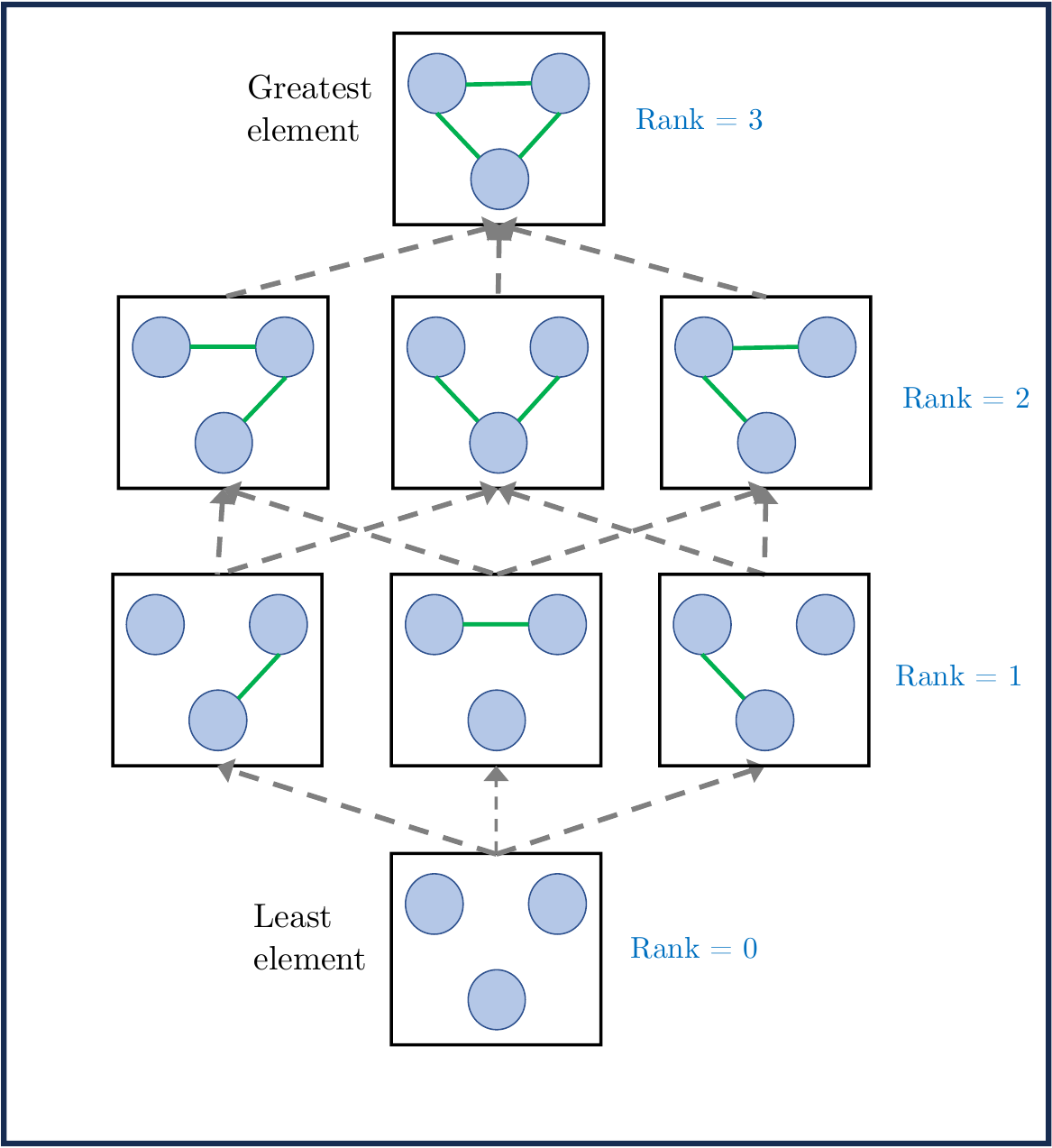}}
  \hfill
  \subfloat[MPDAGs]{\includegraphics[width=.51\columnwidth]{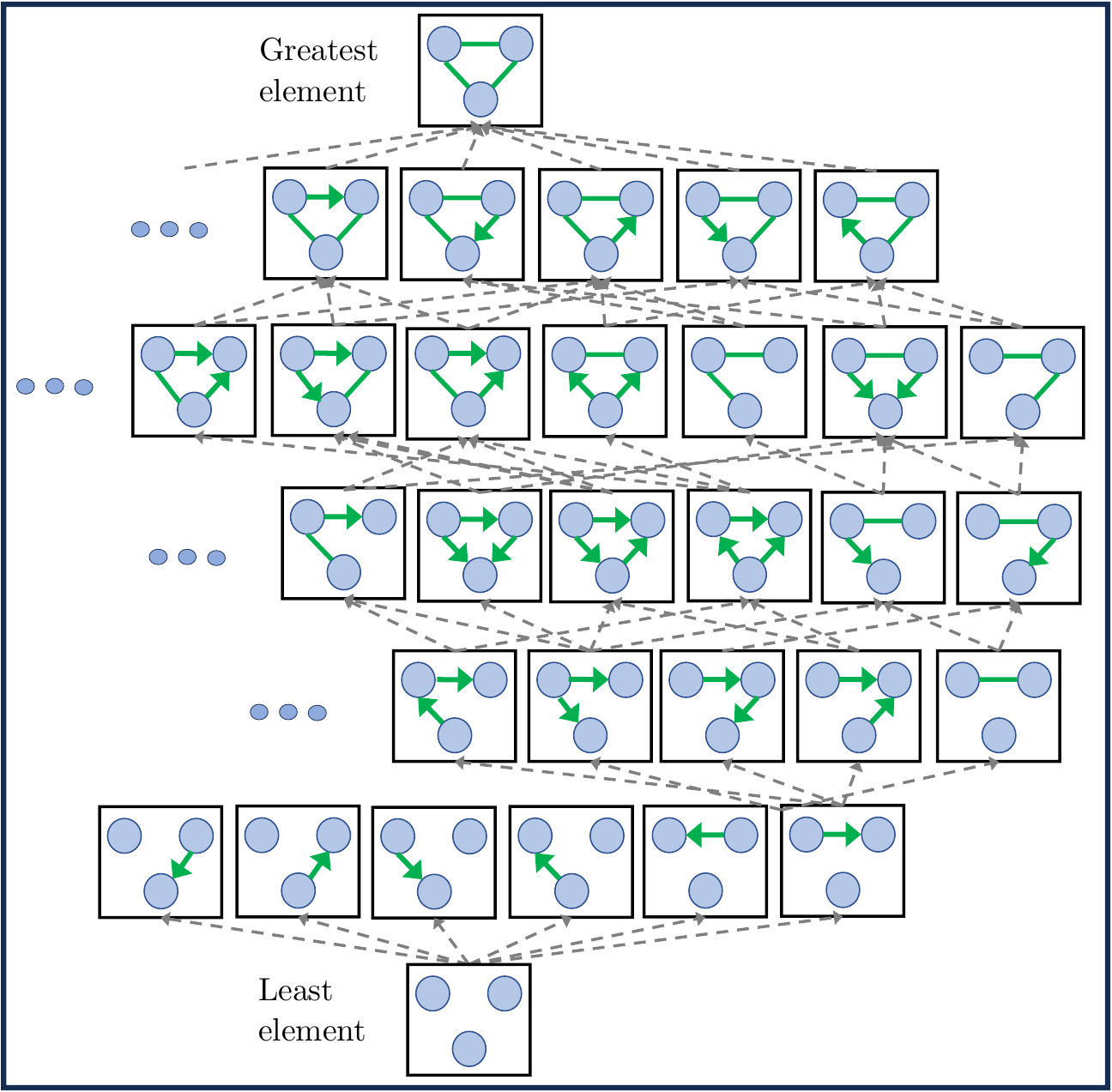}}
\end{minipage}
\caption{Hasse diagrams for four different statistical graph classes over vertex set $\{1,2,3\}$: each box represents a graph, and each dashed arrow represents a covering relation, i.e., we draw $\G_1  \dashrightarrow  \G_2$ pointing upwards if $\G_2$ covers $\G_1$. We can determine whether $\G_1 \preceq \G_2$ by checking whether there is a directed path from $\G_1$ to $\G_2$. In (b) and (d), only part of the poset is shown for simplicity. Each poset has a least element; posets (a), (c) and (d) also have a greatest element. Posets (a), (b) and (c) are graded, while (d) is not graded. 
The poset of ADMGs is not shown here: because it is a Cartesian product of posets (b) and (c), it can be drawn accordingly; see, e.g., \citet[pp.~246-247]{stanley2011enumerative}. }
\label{fig:hasse}
\end{figure}

\subsection{Poset structures that facilitate computing the model-oriented distance}
\label{sec:simplification}

Since the model-oriented distance corresponds to the shortest path in the poset's Hasse diagram, it can, in principle, be computed using search algorithms like A*. However, a naive search is often computationally prohibitive, as it requires repeatedly generating all neighbors for numerous candidate paths. Hence, we investigate structural properties of the poset to save computation --- by deriving a closed-form expression for distance, characterizing the neighbors, or pruning the search space. In this subsection, we analyze several such structural properties, which will be studied in \cref{sec:m_distance_Graphs} for each type of graphs.


%

\subsubsection{Graded poset} 
A connected poset $\Lp$ is \emph{graded} if it admits a function $\rank: \g \rightarrow \mathbb{N}$ such that (i) $\G_1 \coverby \G_2$ implies $\rank(\G_2) = \rank(\G_1)+1$, and (ii) $\rank(\G)=0$ holds for some minimal element $\G$. 
It can be shown that if it exists, such a rank function is unique. 

If the model-oriented poset is graded, we can exploit this property to simplify computation in two ways. First, we obtain the following closed form characterization of the model-oriented distance between graphs comparable within the poset.

\begin{proposition}[Distance between comparable graphs in a graded poset]\label{prop:graded}
Suppose $\Lp= (\mathfrak{G},\preceq)$ is a graded poset with rank function $\rank(\cdot)$. 
For any $\G_s,\G_t \in \mathfrak{G}$ with $\G_s \succeq \G_t$, every downward path from $\G_s$ to $\G_t$ has length $\rank(\G_s)-\rank(\G_t)$, which also equals $d_{\Lp}(\G_s,\G_t)$.
\end{proposition}
This is a well-known result that follows from the property that in a graded poset all \emph{maximal chains} between two comparable elements have equal length. A subset of $\g$ is a \emph{chain} (i.e., totally ordered set) if every pair of elements are comparable. A chain is \emph{maximal} if it is not contained by a larger chain. For completeness, we provide a short proof in \cref{proof:prop_graded}. 
Gradedness and the rank can greatly facilitate the enumeration of neighborhood, which is the basic routine to any search algorithm. Under gradedness, \cref{eqs:neighbor} becomes $\{\G': \G' \preceq \G,\, \rank(\G') = \rank(\G) - 1\} \cup \{\G': \G' \succeq \G,\, \rank(\G') = \rank(\G) + 1\}$.

\subsubsection{Down-up and up-down distance, semi-lattice, semimodularity}
To reduce the number of paths we need to explore, it is natural to consider the shortest path among a restricted set of paths between two graphs in the Hasse diagram. 
Specifically, for $\G_s, \G_t \in \g$, we consider the set of \emph{down-up} and \emph{up-down} paths between them, denoted by $\sP_{\Lp, \du}{(\G_s,\G_t)}$  and $\sP_{\Lp, \ud}{(\G_s,\G_t)}$ respectively. 
A path is down-up if it takes the form of 
\[ \G_s = x_0 \cover \cdots \cover x_{i^\ast} \coverby \cdots \coverby x_d = \G_t, \]
where $i^{\ast} \in \{0,\dots,d\}$ is the inflection point; downward and upward paths are special cases of down-up paths. Conversely, an up-down path takes the form of 
\[ \G_s = x_0 \coverby \cdots \coverby x_{i^\ast} \cover \cdots \cover x_d = \G_t. \]
The next result characterizes when down-up (up-down) paths connect every pair of graphs. 


\begin{proposition}[Existence of down-up and up-down paths] \label{prop:du-up-exist}
The set of down-up paths $\sP_{\Lp, \du}{(\G_1,\G_2)}$ is non-empty for all $\G_1,\G_2\in\mathfrak{G}$ if and only if $\Lp$ has a least element (\cref{cond:least}). Similarly, the set of up-down paths $\sP_{\Lp, \ud}{(\G_1,\G_2)}$ is non-empty for all $\G_1,\G_2\in\mathfrak{G}$ if and only if $\Lp$ has a greatest element (\cref{cond:least}$'$).
\end{proposition} 

The posets of all graph types (UGs, DAGs, ADMGs, CPDAGs and MPDAG) we consider in this paper have the empty graph as the least element. Further, except for DAGs and ADMGs, all of them also have a greatest element; see \cref{tab:summary}. 
The down-up and up-down paths lead to the definition of down-up and up-down distances.

\begin{definition}[Down-up and up-down distance] 
Under \cref{cond:least}, the down-up distance is defined as $d_{\Lp,\du}(\G_s,\G_t) := \min\{\mathrm{length}(p): p \in  \sP_{\Lp, \du}{(\G_s,\G_t)}\}$ for every $\G_s, \G_t \in \g$.
Under \cref{cond:least}$'$,  the up-down distance is defined as $d_{\Lp,\ud}(\G_s,\G_t) := \min\{\mathrm{length}(p): p \in  \sP_{\Lp, \ud}{(\G_s,\G_t)}\}$.
\end{definition}

Unlike $d_{\Lp}$, the down-up and up-down distances may not be metrics as they are not guaranteed to obey the triangle inequality. Nevertheless, by definition, they are upper bounds on $d_{\Lp}$ and are generally easier to compute. When $\Lp$ is graded, the down-up or up-down distances take a closed form.

\begin{proposition}[Down-up and up-down distances for graded posets]
\label{prop:down_up_up_down_graded}
Suppose $\Lp = (\g,\preceq)$ is graded with rank function $\rank(\cdot)$. Then, under \cref{cond:least}, we have
\[ d_{\Lp,\du}(\G_s,\G_t) = \rank(\G_s)+\rank(\G_t)-2\max_{\G \preceq \G_s,\G_t} \rank(\G), \quad \G_s,\G_t \in \g. \]
Similarly, under \cref{cond:least}$'$, we have
\[ d_{\Lp,\ud}(\G_s,\G_t) = 2\min_{\G \succeq \G_s,\G_t} \rank(\G) - \rank(\G_s) - \rank(\G_t), \quad \G_s,\G_t \in \g. \]
\end{proposition}

For a further simplification, we require the notions of meet and join semi-lattice structures: graph $\G_3 \in \g$ is the \emph{least upper bound} (or \emph{join}) of $\G_1$ and $\G_2$, denoted as $\G_1 \join \G_2$, if $\G_3 \succeq \G_1, \G_2$ and $\G_3 \preceq \G'$ for every $\G'$ satisfying $\G' \succeq \G_1, \G_2$. The \emph{greatest lower bound} (or \emph{meet}) of $\G_1$ and $\G_2$ is defined similarly and is denoted as $\G_1 \meet \G_2$. 
Poset $\Lp$ is a \emph{join (meet) semi-lattice} if $\G \join \G'$ ($\G \meet \G'$) exists for every $\G, \G' \in \g$.
\begin{corollary}[Simplification with semi-lattice structure] \label{cor:du-semilat}
Suppose $\Lp = (\g,\preceq)$ is graded with rank function $\rank(\cdot)$. Under \cref{cond:least}, if $\Lp$ is a meet semi-lattice, we have 
\[ d_{\Lp,\du}(\G_s,\G_t) = \rank(\G_s)+\rank(\G_t)- 2 \rank(\G_s \meet \G_t), \quad \G_s,\G_t \in \g. \]
Similarly, under \cref{cond:least}$'$, if $\Lp$ is also a join semi-lattice, we have
\[ d_{\Lp,\ud}(\G_s,\G_t) = 2 \rank(\G_s \join \G_t) - \rank(\G_s) - \rank(\G_t), \quad \G_s,\G_t \in \g. \]
\end{corollary}

As described earlier, down-up or up-down distance is an upper bound on the model-oriented distance. The next key result characterizes the poset structure that makes them equal. 

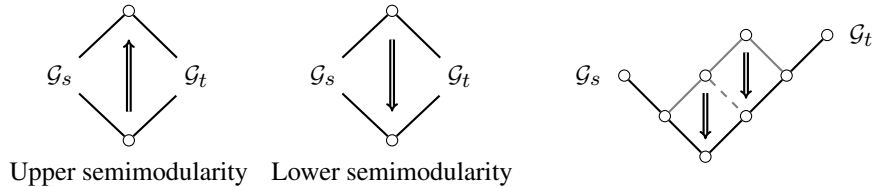
\begin{figure}[htb]
\centering
\begin{tikzpicture}
\tikzset{rv/.style={circle,inner sep=1.5pt,draw,font=\sffamily},
lv/.style={circle,inner sep=1pt,fill=gray!20,draw,font=\sffamily},
fv/.style={rectangle,inner sep=1.5pt,fill=gray!20,draw,font=\sffamily},
node distance=0.8cm, >=stealth}
\begin{scope} 
\node[] (s) {$\G_s$};
\node[rv, above right=of s] (v) {};
\node[rv, below right=of s] (w) {};
\node[below right=of v] (t) {$\G_t$};

\node[below=1mm of w,] (lb) {\small Upper semimodularity};

\draw[-, thick] (s) -- (v);
\draw[-, thick] (t) -- (v);
\draw[-, thick] (s) -- (w);
\draw[-, thick] (t) -- (w);
\draw[->, -{Implies}, double, thick, shorten >= 3mm, shorten <= 3mm] (w)  -- (v);
\end{scope}
\begin{scope}[xshift=3.5cm] 
\node[] (s) {$\G_s$};
\node[rv, above right=of s] (v) {};
\node[rv, below right=of s] (w) {};
\node[below right=of v] (t) {$\G_t$};

\node[below=1mm of w,] (lb) {\small Lower semimodularity};

\draw[-, thick] (s) -- (v);
\draw[-, thick] (t) -- (v);
\draw[-, thick] (s) -- (w);
\draw[-, thick] (t) -- (w);
\draw[->, -{Implies}, double, thick, shorten >= 3mm, shorten <= 3mm] (v)  -- (w);
\end{scope}
\begin{scope}[xshift=7.5cm, node distance=6mm] 
\node[rv] (s) {};
\node[left=1mm of s] (s-label) {$\G_s$};
\node[rv, below right=of s] (a) {};
\node[rv, above right=of a] (b) {};
\node[rv, above right=of b] (c) {};  
\node[rv, below right=of c] (d) {};
\node[rv, above right=of d] (t) {};
\node[rv, below right=of a] (e) {};
\node[rv, above right=of e] (f) {};  
\node[right=1mm of t] (t-label) {$\G_t$};

\draw[->, -{Implies}, double, thick, shorten >= 1.5mm, shorten <= 1.5mm] (b)  -- (e);
\draw[->, -{Implies}, double, thick, shorten >= 1.5mm, shorten <= 1.5mm] (c)  -- (f);
\draw[-, thick] (s) -- (a) -- (e) -- (f) -- (d) -- (t);
\draw[-, thick, color=black!50] (a) -- (b) -- (c) -- (d);
\draw[-, thick, dashed, color=black!50] (b) -- (f);
\end{scope}
\end{tikzpicture}
\caption{Upper and lower semimodularity of a poset. Between any $\G_s, \G_t$ in a lower semimodular poset, a shortest path can always be transformed into a down-up path of the same length.}
\label{fig:semimod}
\end{figure}

\begin{definition}[Semimodularity; see \cref{fig:semimod}] \label{def:semimod}
Poset $\Lp = (\g, \preceq)$ is \emph{lower semimodular} if for every $\G_s, \G_t, \G$ such that $\G \cover \G_s$, $\G \cover \G_t$ and $\G_s \neq \G_t$, there exists some $\G'$ such that $\G_s \cover \G'$ and $\G_t \cover \G'$; similarly, $\Lp$ is \emph{upper semimodular} if the same statement holds when `$\cover$' is replaced with `$\coverby$'.
\end{definition}

\begin{theorem}[Simplification for semimodular posets]\label{thm:semimod}
Given a graph class $\g$ and a model map $\M$ that satisfy \cref{cond:inject}, let $\Lp=(\mathfrak{G},\preceq)$ be the model-oriented poset. 
Under \cref{cond:least} (least element), $\Lp$ is lower semimodular if and only if $d_{\Lp,\du}(\G_s,\G_t) = d_{\Lp}(\G_s,\G_t)$ holds for every $\G_s,\G_t \in \g$.
Similarly, under \cref{cond:least}$'$ (greatest element), $\Lp$ is upper semimodular if and only if $d_{\Lp,\ud}(\G_s,\G_t) = d_{\Lp}(\G_s,\G_t)$ holds for every $\G_s,\G_t \in \g$.
\end{theorem}

\cref{fig:semimod} illustrates the basic mechanism behind \cref{thm:semimod}: when $\Lp$ is semimodular, any shortest path can be transformed into a down-up or up-down path with the same length. 
We remark that, to our best knowledge, \cref{thm:semimod} weakens the conditions for existing results in the literature and is well-suited for studying model-oriented graph distances. For example, \citet[Theorem 8]{Monjardet1981MetricsOP} additionally requires the poset to be graded with a greatest element, while \citet[Proposition 2.5]{foldes2021distances} instead requires the poset to be a semi-lattice. We will use \cref{thm:semimod} to study the model-oriented distance for polytree CPDAGs and MPDAGs in \cref{sec:polytree-cpdags,sec:polytree-mpdags}.  


\subsection{Relation to edge-by-edge distances}
\label{sec:edge-by-edge}
Recall that SHD is given by the number of edges that are mismatched between the two graphs. Different types of mismatches can receive different weights \citep{perrier2008finding}. For PDAGs, for example, there are two common variants of SHD \citep{tsamardinos2006max} defined as follows. Consider encoding a PDAG $\G$ over $\sV$ with adjacency matrix 
\[ A_{i,j} := \begin{cases} 1, \quad &\text{$i - j$ or $i \rightarrow j$} \\ 0, \quad &\text{otherwise} \end{cases}, \quad i,j \in \sV. \]

\begin{definition}[Structural Hamming distance for PDAGs] \label{def:SHD}
Let PDAGs $\G_s, \G_t$ over a vertex set be encoded by adjacency matrix $A^s, A^t$ respectively. The two variants of SHD are given by 
\[ \SHD_{1}(\G_s,\G_t):= \sum_{i<j} \mathbb{I}\left[A^s_{i,j}\neq A^{t}_{i,j} \text{ or }A^s_{j,i}\neq A^{t}_{j,i}\right], \quad
\SHD_{2}(\G_s,\G_t):= \sum_{i,j} \mathbb{I}\left[A^s_{i,j}\neq A^{t}_{i,j}\right], \]
where $\mathbb{I}[\cdot]$ is the indicator function. Note that $\SHD_2 \geq \SHD_1$ holds by definition. 
\end{definition}


These variants of SHD beg the question which is best suited for a graph class, which is addressed by the next lemma. 

\begin{lemma}[Edge-by-edge model-oriented distances] \label{corr:main}
Let $\g$ be a class of graphs with $k \geq 1$ edge types (see \cref{sec:graph_background}). Consider a poset $\Lp$ that orders graphs by subgraph containment, i.e., $\G_s = (\sV, \sE^1_s, \dots, \sE_s^k) \preceq \G_t = (\sV, \sE^1_t, \dots, \sE^k_t)$ if $\sE^i_s \subseteq \sE^i_t$ for $i=1,\dots,k$. 
Suppose $\Lp$ is lower semimodular and graded with $\rank(\G) = \sum_{i=1}^k |\sE^i|$ for each $\G = (\sV, \sE^1, \dots, \sE^k)$. Then, we have $d_{\Lp}(\G_s,\G_t)=\sum_{i=1}^k |\sE^i_s \triangle \sE^i_t|$. 
Further, when $k=1$, $d_{\Lp}$ reduces to $\SHD_1$ if $\g$ is a class of undirected graphs, and to $\SHD_2$ if $\g$ is a class of directed graphs.
\end{lemma}

Effectively, this shows that the model-oriented distance is an edge-by-edge distance whenever the partial ordering can be determined by examining one edge at a time 
(subgraph containment) and the Hasse diagram is sufficiently dense (lower semimodular). 
This typically holds for graph classes that impose few constraints, such as DAGs and ADMGs, 
but not for CPDAGs and MPDAGs. In \cref{sec:A-star} we show that the model-oriented distance differs markedly from SHDs in CPDAGs and MPDAGs. 

\section{Model-oriented distance for probabilistic and causal graphs} \label{sec:m_distance_Graphs}
In this section, we apply the general framework to UGs, DAGs, ADMGs, CPDAGs, MPDAGs as well as simpler subclasses of CPDAGs and MPDAGs. 
\cref{tab:summary} summarizes the structural properties of the poset underlying each graph class, with UGs being the most structured and MPDAGs being the least structured. 
It is worth noting that the lower (meet) and upper (join) properties need not be symmetric. This is due to the fact that, starting from an arbitrary graph, we can more commonly obtain a graph in the same class by removing edges than adding edges.

\begin{table}
\centering
\caption{Structural properties of the model-oriented poset for different graph classes}
\label{tab:summary}
\scalebox{.95}{%
\begin{tabular}{rcccccc} 
\toprule
\textbf{Graph class} & \textbf{Model} & \textbf{Least element} & \textbf{Greatest element} & \textbf{Graded} & \textbf{Semi-lattice} & \textbf{Semimodular} \\
\midrule
UG & prob. & empty graph & full graph& \checkmark & join + meet &  \text{lower + upper}\\
\hline
DAG & causal & empty graph & {\crossmark}& \checkmark & meet  & lower\\
\hline
ADMG & causal & empty graph & {\crossmark}&\checkmark & meet  & lower\\
\hline
CPDAG & prob. & empty graph & full graph&\checkmark & {\crossmark} & \crossmark \\
Polytree CPDAG & prob. & empty graph & {\crossmark}&\checkmark & {\crossmark} & lower \\
\hline
MPDAG & causal & empty graph & full graph&{\crossmark} & {\crossmark} &{\crossmark}\\
Polytree MPDAG & causal & empty graph &{\crossmark} &\checkmark & {\crossmark}& lower \\
\bottomrule
\end{tabular}}
\end{table}

\subsection{Probabilistic UGs} \label{sec:ugs}
As a warm-up, consider $\g$ the class of UGs over a finite vertex set $\sV$ representing probabilistic models given by $\Mug$ in \cref{eq:m-ug}. The next result states that model inclusion becomes graph inclusion.

\begin{proposition}[UG comparability] \label{prop:order-ug}
For UGs $\G_1,\G_2 \in \mathfrak{G}$, we have $\Mug(\G_1) \subseteq \Mug(\G_2)$ if and only if $\G_1$ is a subgraph of $\G_2$.
\end{proposition}
Because graph containment in $\g$ is isomorphic to set containment in the power set of $\{(i,j) \in \sV \times \sV: i<j\}$, we have a model-oriented poset $\Lp$ that is highly structured. In fact, $\Lp$ is known as the \emph{Boolean lattice}, which gives the next result. \cref{fig:hasse}(c) shows its Hasse diagram when $|\sV| = 3$. 

\begin{theorem}[UG poset] \label{prop:poset-ugs}
Let $\Lp$ be the model-oriented poset associated with $(\g, \Mug)$, where $\g$ is the class of UGs over a finite vertex set $\sV$. We have the following results. 
\begin{enumerate}
\item $\Glst$ is the empty graph; $\Ggr$ is the full graph.
\item $\Lp$ is graded with $\rank(\cdot)$ equal to the number of edges. 
\item We have $\G_1 \coverby \G_2$ if and only if $\G_1$ is a subgraph of $\G_2$ with exactly one fewer edge. 
\item $\Lp$ is a lattice (i.e., both a meet and a join semi-lattice).
\item $\Lp$ is both lower and upper semimodular.
\end{enumerate}
\end{theorem}

From \cref{prop:poset-ugs,corr:main}, we obtain a closed form for the model-oriented distance.

\begin{corollary} \label{cor:d-ugs}
For UGs $\G_1 = (\sV, \sE_1)$ and $\G_2 = (\sV, \sE_2)$, we have $d_{\Lp}(\G_1, \G_2) = |\sE_1 \triangle \sE_2| = \SHD_1(\G_1, \G_2)$.
\end{corollary}

\subsection{Causal DAGs} \label{sec:dags}
Let $\g$ be the class of DAGs over a finite vertex set. 

\begin{proposition}[DAG comparability] \label{prop:order-dags}
For DAGs $\D_1,\D_2 \in \mathfrak{G}$, we have $\Mdag(\D_1) \subseteq \Mdag(\D_2)$ if and only if $\D_1$ is a subgraph of $\D_2$.
\end{proposition}
Using this result, the model-oriented poset is characterized by the next result. \cref{fig:hasse}(b) shows the Hasse diagram of this poset when the vertex set is of size 3. 

\begin{theorem}[DAG poset] \label{prop:poset-dags}
Let $\Lp$ be the model-oriented poset associated with $(\g, \Mdag)$, where $\g$ is the class of DAGs over a finite vertex set $\sV$. We have the following results. 
\begin{enumerate}
\item $\D_{\hat{0}}$ is the empty graph. There is no $\D_{\hat{1}}$ when $|\sV|>1$. 
\item $\Lp$ is graded with $\rank(\cdot)$ equal to the number of edges.
\item We have $\D_1 \coverby \D_2$ if and only if $\D_1$ is a subgraph of $\D_2$ with exactly one fewer edge.
\item $\Lp$ is a meet semi-lattice.
\item $\Lp$ is lower semimodular.
\end{enumerate}
\end{theorem}

When $|\sV| > 1$, $\Lp$ is neither a join semi-lattice nor upper semimodular. By \cref{prop:poset-dags,corr:main}, we obtain the following closed form characterization of the model-oriented distance between DAGs.


\begin{corollary} \label{cor:d-dags}
The model-oriented distance between two DAGs is given by 
\[ d_{\Lp}(\D_1, \D_2) = |\sE_1 \triangle \sE_2| = \SHD_2(\D_1, \D_2), \quad \D_1 = (\sV, \sE_1), \,\D_2 = (\sV, \sE_2) \in \g. \]
Further, $d_{\Lp}(\D_1, \D_2) \geq \SHD_1(\D_1, \D_2)$ holds; whenever $\D_1$ and $\D_2$ have at least one edge in opposite directions, we have $d_{\Lp}(\D_1, \D_2) > \SHD_1(\D_1, \D_2)$.
\end{corollary}
This result is valuable for two reasons. First, the closed-form expression is cheap to compute as it coincides with a specific SHD variant. Second, it provides evidence that $\SHD_2$ is more appropriate for DAGs than $\SHD_1$.


\subsection{Causal ADMGs} \label{sec:admgs}
Let $\g$ be the class of ADMGs over a finite vertex set $\sV$. 

\begin{proposition} [ADMG comparability] \label{prop:order-admgs}
For ADMGs $\G_1, \G_2 \in \g$, we have $\Madmg(\G_1) \subseteq \Madmg(\G_2)$ if and only if $\G_1$ is a subgraph of $\G_2$. 
\end{proposition}

Let $\Lp$ be the model-oriented poset associated with $(\g, \Madmg)$. 
For an ADMG $\G = (\sV, \sE^{\leftrightarrow}, \sE^{\to})$, the bidirected edges $\sE^{\leftrightarrow}$ and the directed edges $\sE^{\to}$ are \emph{variation independent} in the sense that any bidirected graph $\B:=(\sV, \sE^{\leftrightarrow})$ and any DAG $\D:=(\sV, \sE^{\to})$ lead to a valid ADMG. 
Then, in light of \cref{prop:order-admgs}, we see that $\Lp$ is a \emph{Cartesian product} of posets $\Lp^{\leftrightarrow}$ and $\Lp^{\to}$ with $\G_1 \preceq \G_2$ if and only if $\B_1 \preceq \B_2$ and $\D_1 \preceq \D_2$; see, e.g., \citet[\S3.2]{stanley2011enumerative}. 
The poset $\Lp^{\leftrightarrow}$ with elements $\B$ is the Boolean lattice over unordered vertex pairs, which is isomorphic to the UG poset. 
The poset $\Lp^{\to}$ with elements $\D$ is the DAG poset. 
The following properties can then be deduced from \cref{prop:poset-ugs,prop:poset-dags}.

\begin{corollary}[ADMG poset] \label{cor:poset-admgs} 
The following results hold:
\begin{enumerate}
\item $\G_{\hat{0}}$ is the empty graph. There is no $\G_{\hat{1}}$ when $|\sV|>1$. 
\item $\Lp$ is graded with $\rank(\cdot)$ equal to the total number of edges.
\item We have $\G_1 \coverby \G_2$ if and only if $\G_1$ is a subgraph of $\G_2$ with exactly one fewer edge.
\item $\Lp$ is a meet semi-lattice.
\item $\Lp$ is lower semimodular.
\end{enumerate}
\end{corollary}

By \cref{cor:poset-admgs,corr:main}, we obtain our distance between ADMGs in a closed form.
\begin{corollary} \label{cor:d-admgs}
The model-oriented distance between two ADMGs is given by 
\[ d_{\Lp}(\G_1, \G_2) = |\sE_1^{\leftrightarrow} \triangle\,\sE_2^{\leftrightarrow}| + |\sE_1^{\to} \triangle\, \sE_2^{\to}|, \quad  \G_1 = (\sV, \sE_1^{\leftrightarrow},\sE_1^{\to}), \;\G_2 = (\sV, \sE_2^{\leftrightarrow},\sE_2^{\to}) \in \g. \]
\end{corollary}

\subsection{Probabilistic CPDAGs} \label{sec:cpdags}

Let $\g$ be the class of CPDAGs over a finite vertex set $\sV$. 
For $\G \in \g$, we use $[\G]$ to denote the set of Markov equivalent DAGs it represents. 
Here, model containment is characterized by a basic result due to \citet{chickering2002optimal}, restated as follows. 
\begin{lemma}[CPDAG comparability]\label{prop:order-cpdags}
For CPDAGs $\G_s,\G_t \in \mathfrak{G}$, the following are equivalent: 
\begin{enumerate}
\item $\Mcpdag(\G_s) \subseteq \Mcpdag(\G_t)$.
\item There exist DAGs $\D_s \in [\G_s]$ and $\D_t \in [\G_t]$ such that the set of d-separations in $\D_t$ are also present in $\D_s$. 
\item There exist DAGs $\D_s \in [\G_s]$ and $\D_t \in [\G_t]$ such that there is a sequence of DAGs $\D_s =: \D_0, \D_1, \dots, \D_k := \D_t$ $(k \geq 0)$, where each $\D_i$ is obtained from $\D_{i-1}$ by one of the two operations: (i) reversing a covered edge, or (ii) adding an edge. For operation (i), $\D_{i-1}$ and $\D_{i}$ are Markov equivalent; for (ii), the d-separations in $\D_{i-1}$ contain those in $\D_i$. 
\end{enumerate}
\end{lemma}

Here, a directed edge $u \rightarrow v$ is \emph{covered} in directed acyclic graph $\D$ if $\Pa_{\D}(v) = \Pa_{\D}(u) \cup \{u\}$. Using the result above, we can derive the properties of the model-oriented poset over the collection $\g_{\textsc{cpdag}}$ listed in \cref{prop:poset-cpdags}. \cref{fig:hasse}(a) shows the Hasse diagram of this poset when $|\sV|=3$.



\citet{Taeb2023ModelSO} highlighted some properties of the CPDAG poset, which we restate as follows.
\begin{theorem}[CPDAG poset] \label{prop:poset-cpdags}
Let $\Lp$ be the model-oriented poset associated with $(\g, \Mcpdag)$, where $\g$ is the class of CPDAGs over a finite vertex set $\sV$. We have the following results. 
\begin{enumerate}
\item $\Glst$ is the empty graph; $\Ggr$ is the full graph.
\item For $\G_1,\G_2 \in \g$, we have $\G_1 \coverby \G_2$ if and only if there exist DAGs $\D_1 \in [\G_1]$, $\D_2 \in [\G_2]$ such that $\D_1$ is a subgraph of $\D_2$ with exactly one fewer edge.
\item $\Lp$ is graded with $\rank(\cdot)$ equal to the number of edges. 
\end{enumerate}
\end{theorem}
Further, $\Lp$ is neither a join or meet semi-lattice nor upper or lower semimodular; see \cref{app:cpdags}. 
As a result, it is more challenging to compute our model-oriented distance here than was the case for UGs, DAGs or ADMGs. First, \cref{corr:main} no longer applies and indeed, as we demonstrate in \cref{ex:SHD}, SHD behaves very differently from our distance. Further, by \cref{thm:semimod}, because the poset is not semimodular, up-down and down-up distances can only provide upper bounds on the model-oriented distance. 
Indeed, there are cases where the shortest path between two CPDAGs is zigzag --- neither down-up nor up-down; see \cref{app:cpdag_ex} for an example.

Due to these difficulties, we propose using the A* algorithm to compute $d_{\Lp}$ (see \cref{sec:A-star}). As an informed search algorithm, A* relies on good, custom-made lower and upper bounds on the distance to speed up the search: the lower bound, often called the \emph{heuristic function} in the literature, allows one to disregard some paths without discarding a potentially optimal path; the upper bound enables early termination and further pruning of suboptimal paths. 
In \cref{app:upper_lower}, we derive upper and lower bounds on $d_{\Lp}$ between CPDAGs. 

\subsubsection{Restriction to polytrees} \label{sec:polytree-cpdags}
We have seen that the model-oriented distance for probabilistic CPDAGs does not admit a simple, closed-form expression in general. 
We now show that the distance does, however, take a simpler form if we restrict the graph class to \emph{polytrees}, that is, CPDAGs whose skeletons do not contain cycles. 
Due to their relative tractability, polytrees have been studied extensively in the causal discovery literature; see, e.g., \citet{Rebane1987TheRO,Sepehr2019AnAT,Grttemeier2021OnTP,Tramontano,Amendola2021ThirdorderMV}.

Formally, let $\g_{\poly}$ be the class of CPDAGs over a vertex set $\sV$ that are polytrees and let $\Lp_{\poly}$ be the model-oriented poset associated with $(\g_{\poly}, \Mcpdag)$. 
The poset $\Lp_{\poly}$ is called the \emph{subposet} of $\Lp$ \emph{induced} by $\g_{\poly}$. 
This subposet has the following properties. 

\begin{theorem}[Polytree CPDAG poset] \label{thm:poset-polytree-cpdags}
Let $\Lp_{\poly}$ be the model-oriented poset associated with $(\g_{\poly}, \Mcpdag)$, where $\g_{\poly}$ is the class of polytree CPDAGs over a finite vertex set $\sV$. The following results hold.
\begin{enumerate}
\item $\Glst$ is the empty graph. There is no $\Ggr$ when $|\sV|>2$.
\item For $\G_1,\G_2 \in \g_{\poly}$, we have $\G_1 \preceq \G_2$ if and only if there exist DAGs $\D_1 \in [\G_1]$, $\D_2 \in [\G_2]$ such that $\D_1$ is a subgraph of $\D_2$.
\item $\Lp_{\poly}$ is graded with $\rank(\cdot)$ equal to the number of edges. 
\item $\Lp_{\poly}$ is lower semimodular. 
\item $d_{\Lp_{\poly}}(\G_s,\G_t) = d_{\Lp_{\poly},\du}(\G_s,\G_t)$ holds for every $\G_s,\G_t \in \g_{\poly}$.
\end{enumerate}
\end{theorem}
Importantly, the last result in this theorem says that we can restrict to down–up paths when finding the model-oriented distance, leading to substantial computational savings.


\subsection{Causal MPDAGs} \label{sec:mpdags}
Let $\g$ be the set of MPDAGs over a finite vertex set $\sV$.
Recall from \cref{sec:background} that an MPDAG $\G$ represents a subset of Markov equivalent (but causally distinct) DAGs that satisfy certain background knowledge; we use $[\G]$ to denote the collection of DAGs that $\G$ represents. 
The partial order in this context is characterized graphically by the following result; compare with \cref{prop:order-cpdags}.

\begin{proposition}[MPDAG comparability] \label{prop:order-mpdags}
Consider two MPDAGs $\G_s$ and $\G_t$. 
We have $\Mmpdag(\G_s) \subseteq \Mmpdag(\G_t)$ if and only if for every DAG $\mathcal{D}_s \in [\G_s]$, there exists a DAG $\mathcal{D}_t \in [\G_t]$ such that $\mathcal{D}_s$ is a subgraph of $\mathcal{D}_t$. 
\end{proposition}

\cref{fig:hasse}(d) shows the corresponding Hasse diagram when $|\sV|= 3$ and 
\cref{tab:summary} summarizes its structural properties. In particular, the least element $\Glst$ is the empty graph and the greatest element $\Ggr$ is the fully connected undirected graph. However, unlike UGs, DAGs and CPDAGs, the MPDAG poset does not admit the total number of edges as the rank function, as illustrated by the next example. 

\begin{example} \label{ex:mpdag}
Consider the MPDAGs in \cref{fig:mpdag_illus} over $n$ vertices. 
One can use \cref{prop:pseudo-covering} to show $\G_1 \coverby \G_2$.
It also holds that $\G_1 \coverby\G_3$. 
To see this, first observe that $\G_1 \prec \G_3$ by \cref{prop:order-mpdags}; then, note that there cannot be a graph $\G$ that satisfies $\G_1 \prec \G \prec \G_3$ because $\G$ must result from removing one or more directed edge from $\G_3$, which would introduce a forbidden structure to the MPDAG (see \cref{fig:dags}).
Hence, despite both being neighbors of $\G_1$, the graph $\G_2$ has one more edge than $\G_1$, while the graph $\G_3$ has $n-1$ more edges than $\G_1$. 
\end{example}

\begin{figure}[htbp]
  \subfloat[$\G_1$]{\includegraphics[width=.25\columnwidth, angle = 0]{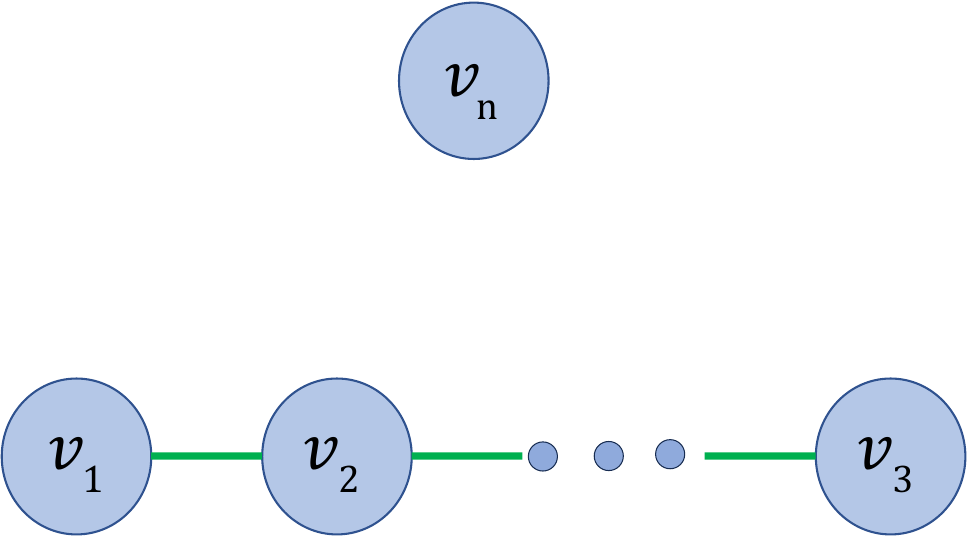}} \hspace{0.5in}\  \subfloat[$\G_2$: a neighbor of $\G_1$]{\includegraphics[width=.25\columnwidth, angle = 0]{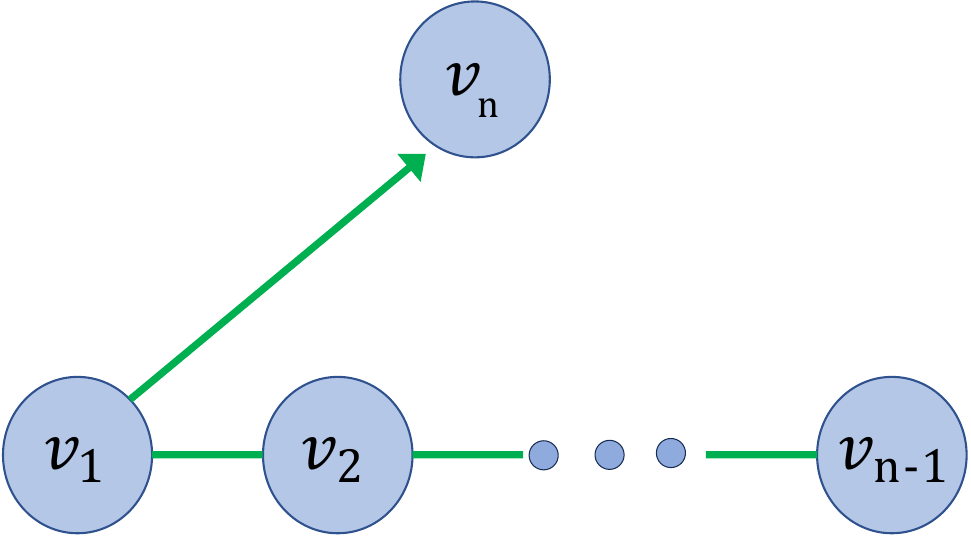}} \hspace{0.5in} 
        \subfloat[$\G_3$: a neighbor of $\G_1$]{\includegraphics[width=.25\columnwidth, angle = 0]{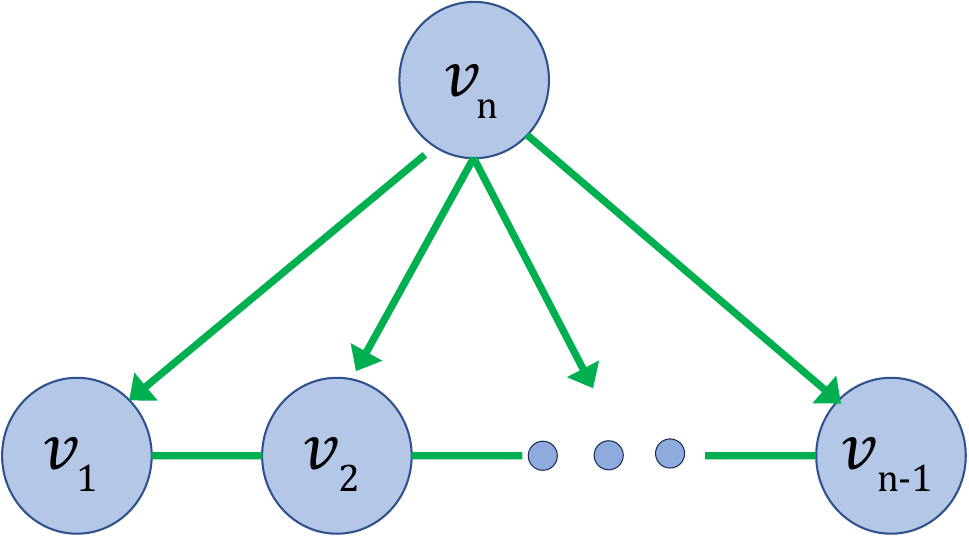}}\hspace{0.2in}  
            \caption{MPDAGs $\G_1, \G_2, \G_3$ considered in \cref{ex:mpdag}.}
\label{fig:mpdag_illus}
\end{figure}

The following result establishes that the poset does in fact not admit any rank function.

\begin{proposition}[MPDAG non-gradedness] \label{prop:not_graded_mpdags}
Let $\g$ be the class of MPDAGs over a finite vertex set $\sV$ with {$|\sV| \geq 4$}. The model-oriented poset of $(\g, \Mmpdag)$ is not graded. 
\end{proposition}

In addition to $\Lp$ not being graded, we do not have an explicit graphical characterization of the neighborhood (i.e., the covering relation); indeed, as we have seen from \cref{ex:mpdag}, such a characterization is likely challenging as the structures of neighboring graphs may vary greatly. 
Hence, it is difficult to enumerate neighbors; an essential component of any search algorithm. Moreover, the down-up and up-down distances are less useful in the context of MPDAGs. First, $\Lp$ is neither upper nor lower semimodular (see \cref{app:mpdags}), so $d_{\Lp,\du}$ and  $d_{\Lp,\ud}$ are merely upper bounds of $d_{\Lp}$ (see \cref{thm:semimod}). Second, the lack of gradedness and semi-lattice structures (see \cref{app:mpdags}) prevents us from reducing these distances to a more computable form (see \cref{prop:down_up_up_down_graded,cor:du-semilat}). 

 To remedy these challenges, we introduce the following notion of \emph{pseudo-rank}.

\begin{definition}[MPDAG pseudo-rank] The pseudo-rank of an MPDAG $\G$ is defined as
\[ \pseudorank(\G):= \text{number of directed edges} +  \text{number of undirected edges} \times 2 . \]
\end{definition}
This pseudo-rank serves as a surrogate for the rank in a non-graded poset and reflects an intuitive measure of structural complexity, as shown by the next two results. \cref{prop:pseudo-covering}, which is proved in \cref{proof:mpdag_pseudo}, may be of independent interest.

\begin{proposition}[Pseudo-rank and covering for MPDAGs] \label{prop:pseudo-covering}
For MPDAGs $\G_1,\G_2$ {such that $\G_1 \preceq \G_2$, we have $\pseudorank(\G_1) \leq \pseudorank(\G_2)$, where the equality holds if and only if $\G_1 = \G_2$.} 
Additionally, if $\pseudorank(\G_1)+1=\pseudorank(\G_2)$, we have $\G_1 \coverby \G_2$. 
\end{proposition}

\begin{proposition}[MPDAG model complexity]\label{prop:mpdag_upper}
Between any MPDAG $\G$ and the empty graph $\Glst$, there exists a maximal chain (i.e., of the form $\Glst=:\G_0 \coverby \cdots \coverby \G_l:=\G$) of length $\pseudorank(\G)$. 
\end{proposition}

Based on \cref{prop:mpdag_upper}, we define $\sP_{\Lp,\pseudo}(\G_s,\G_t)$ as the subset of paths between $\G_s$ and $\G_t$ such that every consecutive pair of graphs on the path differ by unit pseudo-rank, namely
\begin{multline*}
\sP_{\Lp,\pseudo}(\G_s,\G_t):=\big\{(\G_s=:\G_0, \dots, \G_d:=\G_t) \in \sP_{\Lp}(\G_s, \G_t):\; \text{for $i=1,\dots,d$, }\\ |\pseudorank(\G_i) - \pseudorank(\G_{i-1})|=1 \big\}. 
\end{multline*}
In other words, $\sP_{\Lp,\pseudo}(\G_s,\G_t)$ only consists of paths through \emph{restricted neighborhoods}, i.e., those neighboring graphs that differ by exactly one edge between a certain pair of vertices $a,b$: the difference is either $a ~~~ b$ (no edge) versus $a \rightarrow b$, or $a \rightarrow b$ versus $a-b$.
We then define the distance
\begin{equation} \label{eqs:d-pseudo}
d_{\Lp,\pseudo}(\G_s,\G_t):= \min\{\mathrm{length}(p): p \in \sP_{\Lp,\pseudo}(\G_s,\G_t)\}, \quad \G_s, \G_t \in \g.
\end{equation}
Note that for any $\G_s, \G_t$, the set $\sP_{\Lp,\pseudo}(\G_s,\G_t)$ is non-empty because it contains the path that concatenates two maximal chains: first from $\G_s$ to $\Glst$, then from $\Glst$ to $\G_t$. Thus, by \cref{prop:mpdag_upper}, we conclude the following result.
\begin{corollary} \label{cor:mpdag_upper}
The distance $d_{\Lp,\pseudo}$ is a metric. Further, it holds that $d_{\Lp}(\G_s,\G_t) \leq d_{\Lp,\pseudo}(\G_s,\G_t)$ for every $\G_s, \G_t$.
\end{corollary}
Due to the ease of enumerating graphs in the restricted neighborhood, $d_{\Lp,\pseudo}$ is easier to compute than $d_{\Lp}$, $d_{\Lp,\du}$ or $d_{\Lp,\ud}$. We also compare it with SHD in the next result. 

\begin{proposition}[MPDAG distance to SHD relationship] \label{prop:shd_pseudo} 
For any MPDAGs $\G_1$ and $\G_2$, it holds that $\SHD_1(\G_1,\G_2) \leq \SHD_2(\G_1,\G_2) \leq d_{\Lp,\pseudo}(\G_1,\G_2)$.
\end{proposition}

See also \cref{fig:mpdag-numerical} in the Appendix that compares these distances computed for polytree MPDAGs over 5 vertices. 
In \cref{sec:A-star}, we will use $\SHD_2$ as a lower bound in the A* algorithm for computing $d_{\Lp, \pseudo}$. 
Note that neither of the two standard SHDs lower bound the original $d_{\Lp}$. In the example of \cref{fig:mpdag_illus}, we have $d_{\Lp, \pseudo}(\G_1, \G_3) = 1$ but $\SHD_1(\G_1, \G_3) = \SHD_2(\G_1, \G_3) = n-1$.

\subsubsection{Restriction to polytrees} \label{sec:polytree-mpdags}
We consider again restricting the class of graphs to $\mathfrak{G}_{\poly}$, which in this context denotes all MPDAGs that are polytrees over a vertex set $\sV$. 
This induces the subposet $\Lp_{\poly}=(\mathfrak{G}_{\poly},\preceq)$ and further the corresponding shortest-path distance $d_{\Lp_\poly}$ between every pair of graphs in $\mathfrak{G}_{\poly}$. 
Furthermore, let $d_{\Lp_{\poly},\pseudo}$ be the analogue of $d_{\Lp,\pseudo}$ in \cref{eqs:d-pseudo} but defined on the subposet $\Lp_{\poly}$ instead of $\Lp$. 
By an argument similar to the proof of \cref{cor:mpdag_upper} (there still exists a path by concatenating two maximal chains of polytrees), we have $d_{\Lp_\poly}(\G_s,\G_t) \leq d_{\Lp_{\poly},\pseudo}(\G_s,\G_t)$ for every polytree $\G_s, \G_t$.
In fact, the result below shows that this inequality is an equality. 

\begin{theorem}[Polytree MPDAG poset] \label{thm:poset-polytree-mpdags}
Let $\Lp_{\poly}$ be the model-oriented poset associated with $(\g_{\poly}, \Mmpdag)$, where $\g_{\poly}$ is the class of polytree MPDAGs over a finite vertex set $\sV$. We have the following results.
\begin{enumerate}
\item $\Glst$ is the empty graph. There is no $\Ggr$ when $|\sV|>2$.
\item $\Lp_{\poly}$ is graded with rank function equal to $\pseudorank(\cdot)$.
\item $\Lp_{\poly}$ is lower semimodular.
\item For all $\G_s,\G_t \in \g_{\poly}$, we have $d_{\Lp_{\poly},\du}(\G_s,\G_t) = d_{\Lp_{\poly}}(\G_s,\G_t) = d_{\Lp_{\poly},\pseudo}(\G_s,\G_t)$.
\item For all $\G_s,\G_t \in \g_{\poly}$, we have $\SHD_1(\G_s,\G_t) \leq \SHD_2(\G_s,\G_t) \leq d_{\Lp_{\poly},\pseudo}(\G_s,\G_t)$.
\end{enumerate}
\end{theorem}

\section{Algorithms for computing model-oriented distances} \label{sec:A-star}

\subsection{Generic A* algorithm with branch and bound} \label{sec:general_procedure}
A* \citep{hart1968formal} is a widely-used search algorithm to find the shortest path between two vertices on a graph. 
By employing a \emph{heuristic function} to guide the search, it is typically more efficient than naive breadth-first or depth-first search. 
Suppose we want to find the distance between $\G_s$ (source) and $\G_t$ (target). 
Starting from $\G_s$, the algorithm maintains a priority queue (known as the \emph{open set}) for the graphs to be visited, sorted by the index $f(\G) := g(\G) + h(\G)$ for each graph $\G$ in the queue. 
Here, $g(\G)$ is the distance from $\G_s$ to $\G$ and $h(\G)$ is the \emph{heuristic} that estimates from below the distance from $\G$ to $\G_t$. 
For each iteration, a graph with the smallest index is popped from the queue and its neighbors are pushed to the queue, upon which their $f$ and $g$ values are updated if smaller values are found. 
The algorithm terminates when $\G_t$ is popped from the queue and returns $g(\G_t)$ as the distance between $\G_s$ and $\G_t$. 
When $\G_s$ and $\G_t$ are connected, A* is guaranteed to terminate. 
It can be shown that A* outputs the correct distance if the heuristic is \emph{admissible} in the sense that $0 \leq h(\G) \leq d_{\Lp}(\G,\G_t)$ for every $\G \in \mathfrak{G}$.
Depending on how the queue is implemented, A* reduces to breadth-first or depth-first search if a trivial heuristic $h \equiv 0$ is supplied. By using an admissible heuristic with smaller $d_{\Lp}(\G,\G_t) - h(\G)$, A* becomes more efficient. 

To further save computation with branch and bound, we also employ an upper bound $u(\G)$ on the distance from $\G$ to $\G_t$ ($u \equiv +\infty$ is a trivial upper bound). 
By using $u(\cdot)$ to maintain $u^{\ast}$, the smallest upper bound on the distance between $\G_s$ and $\G_t$ witnessed so far, the algorithm can terminate early whenever $f(\G) \geq u^{\ast}$ due to the relation $f(\G) \leq d_{\Lp}(\G_s,\G_t) \leq u^{\ast}$. Moreover, when expanding the neighborhood of $\G$, those graphs with $f$ no smaller than $u^{\ast}$ can be pruned immediately. Naturally, one can use $d_{\Lp,\ud}$, $d_{\Lp,\du}$ or their minimum as the upper bound. We present this generic A* in \cref{alg:general_astar}, which requires three subroutines: lower bound $h(\cdot)$, upper bound $u(\cdot)$ and $\textsc{EnumNeighbors}()$ that enumerates the neighborhood of a graph in the model-oriented poset. 
\begin{theorem}[Algorithm correctness] \label{thm:alg}
Suppose $\g$ is a finite graph class and the associated model-oriented poset $\Lp=(\g, \preceq)$ is connected. For \cref{alg:general_astar}, suppose the subroutines $\textsc{EnumNeighbors}(\cdot)$, $u(\cdot)$ and $h(\cdot)$ satisfy the conditions specified at the beginning of the algorithm. Then, the algorithm will terminate and return $d_{\Lp}(\G_s, \G_t)$. 
\end{theorem}

\begin{algorithm}[!tb]
\DontPrintSemicolon
\SetKwFunction{Dict}{Dict}
\SetKwFunction{EnumNeighbors}{EnumNeighbors}
\SetKwFunction{PriorityQueue}{PriorityQueue}
\SetKwFunction{Pop}{Pop}
\SetKwFunction{Push}{Push}
\SetKwFunction{Keys}{Keys}
\SetKwInOut{Require}{Require}
\SetKwData{openSet}{openSet}
\SetKwData{fcurr}{f}
\SetKwData{gcurr}{g}
\SetKwData{ucurr}{u}
\SetKw{And}{and}
\SetKw{Or}{or}
\SetKw{Continue}{Continue}
\LinesNumbered

\KwIn{Graphs $\G_s,\G_t$ in a finite graph class $\g$ with model-oriented poset $\Lp=(\g, \preceq)$.}
\Require{Subroutine $\EnumNeighbors(\G)$ that outputs the neighbors of $\G$ in $\Lp$;\\
subroutine $h(\G)$ such that $0 \leq h(\G) \leq d_{\Lp}(\G,\G_t)$ for every $\G \in \g$; \\
subroutine $u(\G)$ such that $d_{\Lp}(\G,\G_t) \leq u(\G) \leq +\infty$ for every $\G \in \g$.}
\KwOut{Model-oriented distance $d_{\Lp}(\G_s,\G_t)$.}
\BlankLine
	$u^{\ast} \gets u(\G_s)$ \tcp*{Upper bound on $d_{\Lp}(\G_s,\G_t)$} 
    $g \gets \Dict[]$ \tcp*{Length of the shortest path from $\G_s$ found so far} 
    $g[\G_s] \gets 0$ \\
    $\openSet \gets \PriorityQueue\left((\G_s, 0), h(\G_s) \right)$ \tcp*{Item $(\G_s, g[\G_s])$ with index $g[\G_s]+h(\G_s)$}

    \While{$\openSet \neq \emptyset$}{
    $(\G,\gcurr), \fcurr \gets \Pop(\openSet)$  \tcp*{Pop an item with the smallest index}
    
    \If{$\fcurr \geq u^{\ast}$}{
    	\Return{$u^{\ast}$} \tcp*{Can terminate early because $\fcurr \leq d_{\Lp}(\G_s,\G_t) \leq u^{\ast}$}
    }
    
    \If{$\gcurr > g[\G]$}{ 
    	\Continue \tcp*{Skip a stale entry in the queue}
    }
    
    \If{$\G = \G_t$}{ 
    	\Return{\gcurr} \tcp*{Reached the target}
    }
    
\ForEach{$\G' \in \EnumNeighbors(\G)$}{
            $\gcurr' \gets \gcurr + 1$; \tcp*{Increment current distance by one}
            $\ucurr' \gets \gcurr' + u(\G')$ \\
            \If{$\ucurr' < u^{\ast}$}{
            $u^{\ast} \gets \ucurr'$  \tcp*{Tighten the upper bound} }
            \If{$\G' \notin \Keys(g)$ \Or $\gcurr' < g[\G']$}{
                $g[\G'] \gets \gcurr'$ \tcp*{Path through $\G'$ is created or shortened}
				$\fcurr' \gets \gcurr' + h(\G')$ 
				
				\If{$\fcurr' < u^{\ast}$}{
				$\Push(\openSet, (\G', \gcurr'), \fcurr')$ \tcp*{Push a potentially better solution}
				}
			}
		}
	}
\Return{$-1$}  \tcp*{Error: poset is not connected}
\caption{Generic A* algorithm with branch and bound for computing the model-oriented distance}
\label{alg:general_astar}
\end{algorithm}

\subsection{Numerical Results for CPDAGs} \label{sec:alg-cpdags}
We specialize our A* algorithm to CPDAGs; numerical experiments for MPDAGs can be found in \cref{app:MPDAG_numerical}. 
We describe the details of the algorithm including subroutines for enumerating neighbor, an admissible heuristic and an upper bound in \cref{sec:additional_algs}.

\begin{figure}[htbp]
\begin{center}
\includegraphics[width=.7\textwidth]{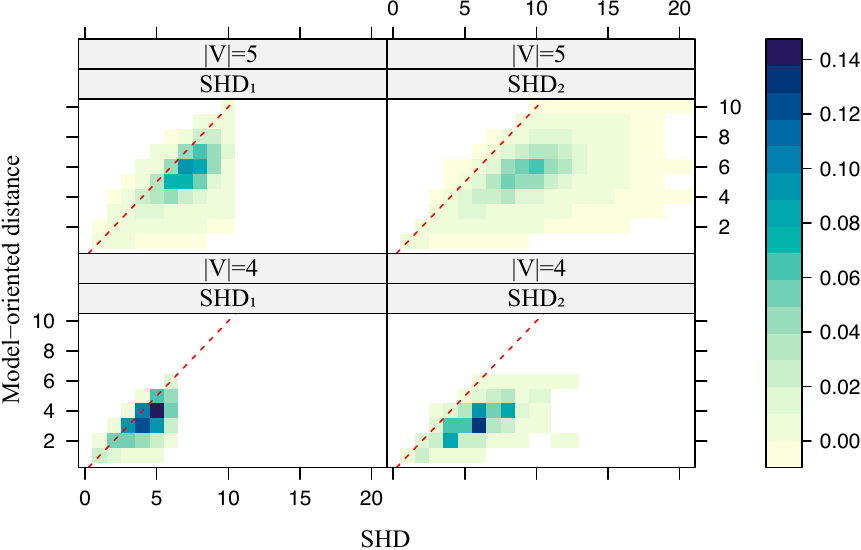}
\end{center}
\caption{Joint distribution of our distance and the structural Hamming distance computed over all pairs of CPDAGs over 4 or 5 vertices, where the identity line is drawn as dashed.}
\label{fig:cpdag_results}
\end{figure}

To demonstrate our algorithm, we report results from a small numerical experiment. 
There are 185 CPDAGs over 4 vertices and 8782 CPDAGs over 5 vertices \citep{gillispie2001enumerating}. 
We use our algorithm to compute the model-oriented distance between every pair of CPDAGs over 4 or 5 vertices. 
\cref{fig:cpdag_results} compares our distance with the structural Hamming distance: both $\SHD_1$ and $\SHD_2$ are typically larger than the model-oriented distance. 
In particular, \cref{fig:cpdag_results_4nodes} highlights some pairs of graphs with large discrepancy between our distance and the SHDs.  

We implement our algorithm in C++ and the results are obtained on an ARM64 machine with 4 CPU cores and 32 GB memory. 
The current version scales to a moderate number of vertices. 
For instance, for a pair of CPDAGs over 13 vertices randomly drawn from an Erd\"os--R\'enyi model with edge probability 0.2, the algorithm takes about 0.1 second on average to compute the distance.
It is worth mentioning that in this case $\g$ contains around $5 \times 10^{30}$ graphs \citep{oeis2025}, so the computation is only feasible by  algorithmically exploiting the structures of the poset. 
From our experiments, a significant amount of computation is saved by early termination upon attaining the upper bound (line 7 of \cref{alg:general_astar}). For CPDAGs over 4 or 5 vertices, around 96\% or 89\% of the pairs have $d_{\Lp} = \min(d_{\Lp,\ud}, d_{\Lp,\du})$.

\section{An application to covariate adjustment} \label{sec:robust-adj}
\cref{fig:robust} presents an application of our method to covariate adjustment in causal inference. 
Suppose we are interested in estimating the causal effect of $v_1$ on $v_2$ from observational data.
Suppose we specify a causal DAG $\G_0$ shown at the bottom of the Hasse diagram, where $v_3, v_4$ are the observed baseline covariates. 
According to the graph and the back-door criterion \citep{pearlBayesianAnalysisExpert1993}, the causal effect can be identified by adjusting for just $v_3$, or just $v_4$ or both $v_3$ and $v_4$ --- these three sets of covariates are the valid \emph{adjustment sets} under the specified graph. 
Yet, since the graph may be misspecified, it is of interest to determine an adjustment set that is most \emph{robust} to misspecifications, in particular directed or bidirected edges missing from the specification.  
The Hasse diagram depicts the model-oriented poset of $(\g, \Madmg)$, where $\g$ contains the specified graph and its supergraphs subject to the restriction that $v_3,v_4$ temporally precede $v_1$. 
For each graph in the diagram, the valid adjustment sets are color-coded.
We can see that the adjustment set $\{v_3,v_4\}$ is the most robust: it remains valid whenever the other sets do and is able to tolerate misspecifications as large as $d_{\Lp} = 5$. 
This sensitivity analysis depends on the fact that $d_{\Lp}$ is a metric so that the set of graphs $\{\G \in \g: d_{\Lp}(\G, \G_0) \leq d\}$ can be iteratively enlarged. 


\begin{figure}[htbp]
\definecolor{colNONE}{HTML}{A9A9A9}
\definecolor{colC1C2}{HTML}{FED976}
\definecolor{colC2C1C2}{HTML}{FD8D3C}
\definecolor{colC1C1C2}{HTML}{E31A1C}
\definecolor{colALL}{HTML}{800026}
\DeclareRobustCommand{\coloredcircle}[1]{%
    \tikz[baseline=-0.5ex]\draw[fill=#1] (0,0) circle (0.6ex);%
}
\centering
\includegraphics[width=1.\textwidth]{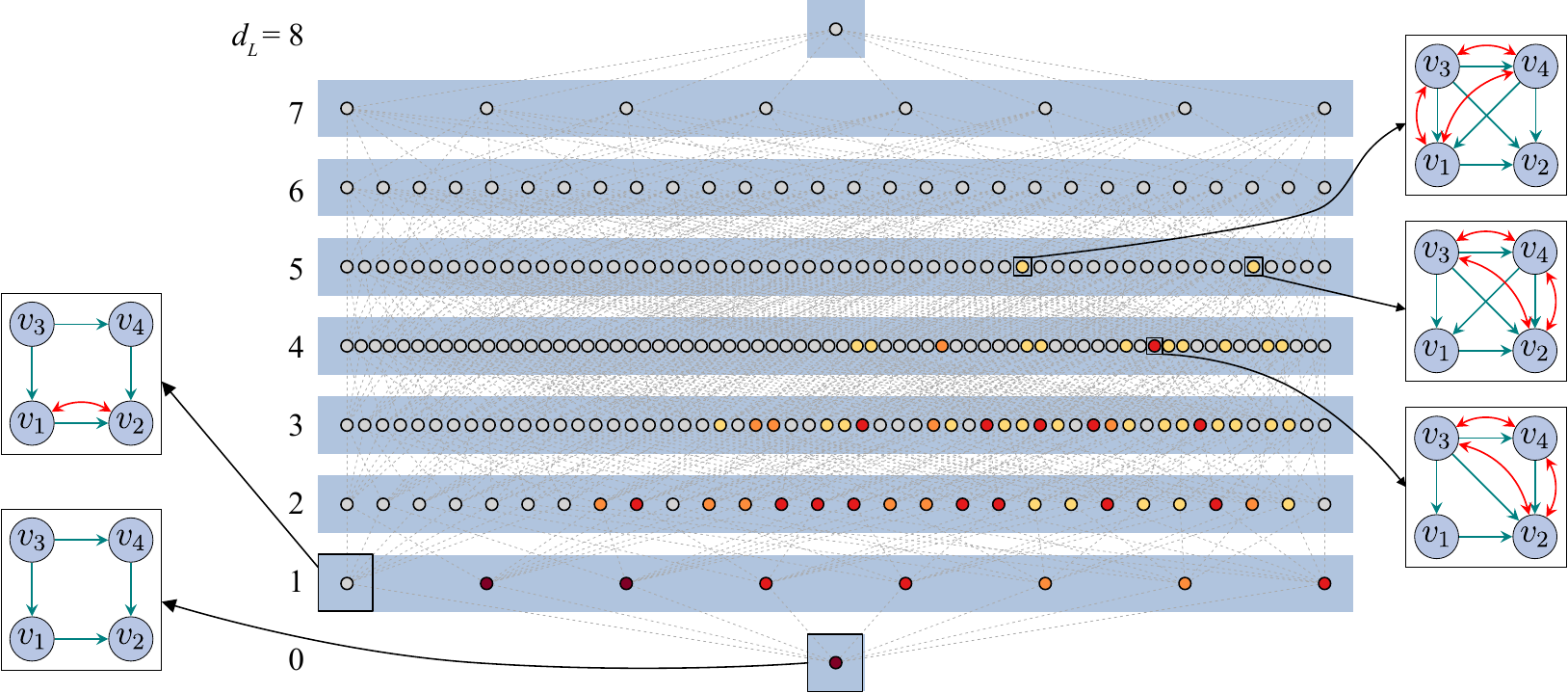}
\caption{An example on selecting an adjustment set that is robust to misspecification of the causal graph. In the Hasse diagram, the least element is the specified graph and each graph above represents a supergraph containing extra directed or bidirected edges. While adjustment sets $\{v_3\},\{v_4\},\{v_3,v_4\}$ are all valid for identifying the effect of $v_1$ on $v_2$ when the graph is correctly specified, the set $\{v_3,v_4\}$ is the most robust against misspecification. Valid adjustment sets are color-coded for each graph: \coloredcircle{colALL} all valid, \coloredcircle{colC2C1C2} $\{v_4\}$ and $\{v_3,v_4\}$, \coloredcircle{colC1C1C2} $\{v_3\}$ and $\{v_3,v_4\}$, \coloredcircle{colC1C2} $\{v_3,v_4\}$, \coloredcircle{colNONE} none.}
\label{fig:robust}
\end{figure}

\section{Discussion} \label{sec:discussion}
This paper introduces a model-oriented framework for defining a distance between graphs  relative to a class of graphs in consideration.
By treating each graph as a statistical (i.e., probabilistic or causal) model, we organize the class of graphs as a poset ordered by model inclusion. 
The poset induces a notion of neighborhood, which then yields a shortest path distance through neighbors as a metric in the space of graphs. 
By studying and exploiting the poset structures that are summarized in \cref{tab:summary}, we develop A* algorithms for computing the model-oriented distance that can scale up to a moderate number of vertices. 

By bridging graphical modeling with order theory, we open several avenues for future research. 
First, to compute the model-oriented distance between two MPDAGs, it remains an open problem to obtain an explicit graphical characterization of the neighborhood. 
Second, our distance can be extended to incorporate weights on the covering relations in the poset. In particular, one may introduce information-theoretic weights similar to those considered by \citet{meila2005comparing} for partitions. 
Third, it is of interest to consider model selection in the spirit of \citet{Taeb2023ModelSO}. 
Moreover, as an ongoing project, we are exploring ways to scale up the A* algorithm to larger graphs through better heuristics, upper bounds, memory management and parallelization. Finally, our distance is in principle applicable to many other statistical graphs, such as maximal ancestral graphs \citep{Richardson2002AncestralGM}, partial ancestral graphs \citep{spirtes2000causation} and local independence graphs for event data \citep{Didelez2007GraphicalMF,mogensen2020markov}.

\begin{acks}[Acknowledgments]
The authors thank undergraduate students Jian Kang, Carter Lembo and Arnav Mazumder at the University of Washington for helping with the implementation of our algorithms. FRG and AT would like to thank the Isaac Newton Institute for Mathematical Sciences, Cambridge, for support and hospitality during the programme \emph{Causal inference: From theory to practice and back again} (supported by EPSRC grant no EP/K032208/1) where part of the work on this paper was undertaken. AT acknowledges funding from NSF grant DMS-2413074.  
\end{acks}

\begin{supplement}
\stitle{Supplement to ``Model-oriented graph distances via partially ordered sets''}
\sdescription{The supplement contains auxiliary lemmas, proofs and additional numerical results.}
\end{supplement}
\bibliographystyle{imsart-nameyear}
\bibliography{main.bib}
\newpage
\begin{appendix}
\appendixnumbering
\crefalias{section}{appendix}
\crefalias{subsection}{appendix}
\crefalias{subsubsection}{appendix}

\section{Proofs: model maps and associated graphs} \label{sec:pf}
\subsection{Injective model maps} \label{sec:inj}
We show that the models $\Mug, \Mcpdag, \Mdag, \Mmpdag$, as defined in \cref{sec:models}, are injective in the sense that no two different graphs represent the same model.   

\begin{assumption} \label{ass:inj}
$\sV$ is a finite vertex set. For every $v \in \sV$, $\mathcal{X}_v$ is either infinite or finite with two or more elements. 
\end{assumption}

\begin{theorem}
Under \cref{ass:inj}, we have the following results.
\begin{enumerate}[(i)]
\item $\Mug: \g = \{\text{all UGs over vertex set $\sV$}\} \rightarrow 2^{\Mo_+}$ is injective.
\item $\Mcpdag: \g = \{\text{all CPDAGs over vertex set $\sV$}\} \rightarrow 2^{\Mo}$ is injective.
\item $\Madmg: \g = \{\text{all ADMGs over vertex set $\sV$}\} \rightarrow 2^{\Mc}$ is injective. 
\item $\Mdag: \g = \{\text{all DAGs over vertex set $\sV$}\} \rightarrow 2^{\Mc} $ is injective.
\item $\Mmpdag: \g = \{\text{all MPDAGs over vertex set $\sV$}\} \rightarrow 2^{\Mc} $ is injective. 
\end{enumerate}
\end{theorem}

\begin{proof} 
The proof for $|\sV| = 1$ is trivial. In what follows, we assume $|\sV| \geq 2$. 
\begin{enumerate}[(i)]
\item For every $\G = (\sV, \sE) \in \g$, $\Mug(\G)$ can be characterized by the pairwise Markov property \citep[\S3.2.1]{lauritzen1996graphical}
\[ (u,v) \notin \sE \implies X_u \indep X_v \mid X_{\sV \setminus \{u,v\}}, \quad \forall u, v \in \sV, u \neq v. \]
Further, under \cref{ass:inj}, there exists $P \in \Mug(\G)$ that is faithful to $\G$, i.e., 
\[ (u,v) \in \sE \implies X_u \centernot{\indep} X_v \mid X_{\sV \setminus \{u,v\}} \text{ under $P$}, \quad \forall u, v \in \sV. \]
This shows that the map must be injective. 

\item This directly follows from the characterization of Markov equivalent DAGs; see \citet{andersson1997characterization}.

\item Fix $\G \in \g$. For showing injectivity, it suffices to exhibit a law $P_0 \in \Madmg(\G)$ such that $P_0 \notin \Madmg(\G')$ for every $\G' \neq \G$. 
Fix $x \in \mathcal{X}$ and consider $(X_v(x): v \in \sV)$, namely the naturally occurring values of $X$ under the intervention that imposes $x_v$ on $X_v$ for each $v$. 
By (i) in \cref{eq:m-admg} and $|\sV| < \infty$, we have
\begin{equation} \label{eq:admg-inj-1}
P\left( X_v(x) = X_v(x_{\Pa(v)}),\; \forall v \in \sV \right) = 1, \quad \forall P \in \Madmg(\G).
\end{equation}
Further, we consider the set of non-degenerate laws
\[ \Gamma_1:=\left\{P \in \Madmg(\G): P\left(X_v(x) = X_v(x_{\sW}) \right) < 1, \quad \forall v \in \sV, \; \sW \subseteq \sV, \; \sW \neq \Pa(v) \right\}. \]
By (ii) in \cref{eq:m-admg}, we have 
\begin{equation} \label{eq:admg-inj-2}
\forall i, j \in \sV,\; i \neq j,\; i \not \leftrightarrow j \implies X_i(x) \indep X_j(x) \text{ under } P, \quad \forall P \in \Madmg(\G).
\end{equation}
Consider the set of faithful laws 
\[\Gamma_2:=\left\{ P \in \Madmg(\G): \forall i, j \in \sV,\; i \neq j,\; i \leftrightarrow j \implies X_i(x) \centernot{\indep} X_j(x) \text{ under } P \right\}. \]
Take any $P_0 \in \Gamma_1 \cap \Gamma_2 \neq \emptyset$ and observe that $P_0$ cannot be contained by $\Madmg(\G')$ for any other $\G' \in \g$ by \cref{eq:admg-inj-1,eq:admg-inj-2}.

\item This is a direct consequence of (iii) because $\Mdag = \Madmg$ for any DAG. 

\item Consider $\G_1, \G_2 \in \g$, $\G_1 \neq \G_2$. By the completeness of Meek's rules \citep{meek1995causal}, the sets of DAGs represented by $\G_1$ and $\G_2$ must be different. Without loss of generality, suppose there is a DAG $\D_1$ that is represented by $\G_1$ but not by $\G_2$. Fix $x \in \mathcal{X}$. By the proof of (iii), there exists $P' \in \Mdag(\D_1) \subseteq \Mmpdag(\G_1)$ such that 
\[\sW \neq \Pa_{\D_1}(v) \implies P'\left( X_v(x) = X_v(x_{\sW}) \right) < 1, \quad \forall v \in \sV, \; \sW \subseteq \sV.\]
It then follows that $P' \notin \Mdag(\D_2)$ for every DAG $\D_2$ represented by $\G_2$, and hence $P'  \notin \Mmpdag(\G_2)$. We have shown that $\Mmpdag(\G_1) \neq \Mmpdag(\G_2)$.
\end{enumerate}
\end{proof}

\subsection{Graphical characterization of model containment}
\label{sec:proof_ordering_mpdags}

\begin{proof}[Proof of \cref{prop:order-ug}]
The `$\Leftarrow$' direction directly follows from \cref{eq:m-ug}. For the `$\Rightarrow$' direction, we prove by contradiction. Suppose $\Mug(\G_1) \subseteq \Mug(\G_2)$ but $\G_1$ contains an edge $(u,v)$ that is not in $\G_2$. Then, there exists a distribution $P \in \Mug(\G_1)$ under which $X_u$ and $X_v$ are not conditionally independent given $X_{\sV \setminus \{u,v\}}$. However, $P \in \Mug(\G_2)$ contradicts the conditional independence posed by $\Mug(\G_2)$.
\end{proof}

\begin{proof}[Proof of \cref{prop:order-dags}]
By definition of the causal model associated with a DAG (see \cref{sec:causal-models}), if $\D_1$ is a subgraph of $\D_2$, then $\Mdag(\D_1) \subseteq \Mdag(\D_2)$. We now prove the other direction by contradiction. Suppose $\Mdag(\D_1) \subseteq \Mdag(\D_2)$ but $\D_1$ contains an edge $i \rightarrow j$ that is absent from $\D_2$. Recall that $\mathcal{X} = \prod_{v \in \sV} \mathcal{X}_v$ is the sample space.
Let us choose $\mathcal{X}_i = \mathbb{R}$ for every $i \in V$. We can choose a law $P \in \Mdag(\D_1)$ of $X(\cdot)$ such that $\E[X_j(x)] = x_i$ holds for every $x \in \mathcal{X}$. However, under $\Mdag(\D_2)$, $\E[X_j(x)]$ (when it exists) can only depend on $x_{\Pa_{\D_2}(j)}$, which contradicts $i \notin \Pa_{\D_2}(j)$.
\end{proof}

\begin{proof} [Proof of \cref{prop:order-admgs}]
The `if' part follows from the definition of $\Madmg$. We now show the `only if' part by contradiction.
Let us choose $\mathcal{X}_v = \mathbb{R}$ for every $v \in \sV$.
Suppose $\Madmg(\G_1) \subseteq \Madmg(\G_2)$ but there exists an edge between $i$ and $j$ present in $\G_1$ but not in $\G_2$. 
\begin{enumerate}
\item If the edge is $i \to j$, then we have $i \in \Pa_{\G_1}(j)$ and $i \notin \Pa_{\G_2}(j)$. Let us choose a law $P \in \Madmg(\G_1)$ of $X(\cdot)$ such that $\E[X_j(x)] = x_i$ holds for every $x \in \mathcal{X}$. However, this contradicts $P \in \Madmg(\G_2)$ since $\E[X_j(x)]$ can only depend on $\Pa_{\G_2}(j)$ under $\G_2$. 

\item If the edge is $i \leftrightarrow j$, then fix $x \in \mathcal{X}$ and choose a law $P \in \Madmg(\G_1)$ such that $X_i(x) \centernot{\indep} X_j(x)$. However, this contradicts $P \in \Madmg(\G_2)$ because $\G_2$ posits $X_i(x) \indep X_j(x)$. 
\end{enumerate}
\end{proof}

\begin{proof}[Proof of \cref{prop:order-mpdags}]
Recall that $\Mmpdag(\G_i) := \cup_{\D \in [\G_i]} \Mdag(\D)$ for $i=1,2$. It follows from the definition of $\Mmpdag$ and \cref{prop:order-dags} that, if for every $\D_1 \in [\G_1]$, there exists a certain $\D_2 \in [\G_2]$ such that $\D_1$ is a subgraph of $\D_2$, then $\Mmpdag(\G_1) \subseteq \Mmpdag(\G_2)$. We now prove the other direction. Let us choose $\mathcal{X}_i = \mathbb{R}$ for every $i \in V$. Fix any $\D_1 \in [\G_1]$. We want to show that $\D_1$ is a subgraph of some $\D \in [\G_2]$. We can choose a law $P \in \Mdag(\D_1)$ of $X(\cdot)$ such that for every $i \in V$, $x \in \mathcal{X}$, it holds that
\begin{equation} \label{eqs:pf-mpdags}
\E[X_i(x)] = \begin{cases} \beta_i^{\T} x_{\Pa_{\D_1}(i)}, & \quad \Pa_{\D_1}(i) \neq \emptyset \\ 0, &\quad \Pa_{\D_1}(i) = \emptyset \end{cases},
\end{equation}
where $\beta_i \in \mathbb{R}^{|\Pa_{\D_1}(i)|}$ is a fixed vector with non-zero coordinates. Since $P \in \Mdag(\D_1) \subseteq \Mmpdag(\G_1) \subseteq \Mmpdag(\G_2) = \cup_{\D \in \G_2} \Mdag(\D)$, there must exist $\D_2 \in [\G_2]$ such that $P \in \Mdag(\D_2)$. By definition, under $\Mdag(\D_2)$, $\E[X_i(x)]$ (when it exists) can only depend on $x_{\Pa_{\D_2}(i)}$. Hence, \cref{eqs:pf-mpdags} implies that $\Pa_{\D_1}(i) \subseteq \Pa_{\D_2}(i)$ for every $i \in V$. That is, $\D_1$ is a subgraph of $\D_2$. 
\end{proof}

\section{Structural constraints of graph types} \label{sec:valid-graph}
In this appendix, we describe the conditions for a PDAG to be a valid CPDAG or MPDAG. 
\subsection{CPDAG} \label{sec:valid-cpdag}
For a PDAG $\G$ to be a valid CPDAG, it must also satisfy the following conditions \citep[Theorem 4.1]{andersson1997characterization}.
\begin{enumerate}[(i)]
\item It contains no partially directed cycle.
\item Each connected component spanned by undirected edges is chordal;
\item $v_1 \rightarrow v_2 - v_3$ does not occur as an induced subgraph of $\G$; 
\item Every directed edge $v_1 \rightarrow v_2$ in $\G$ occurs in at least one of following four structures as an induced subgraph of $\G$:
\begin{center}
\includegraphics[width=0.7\linewidth]{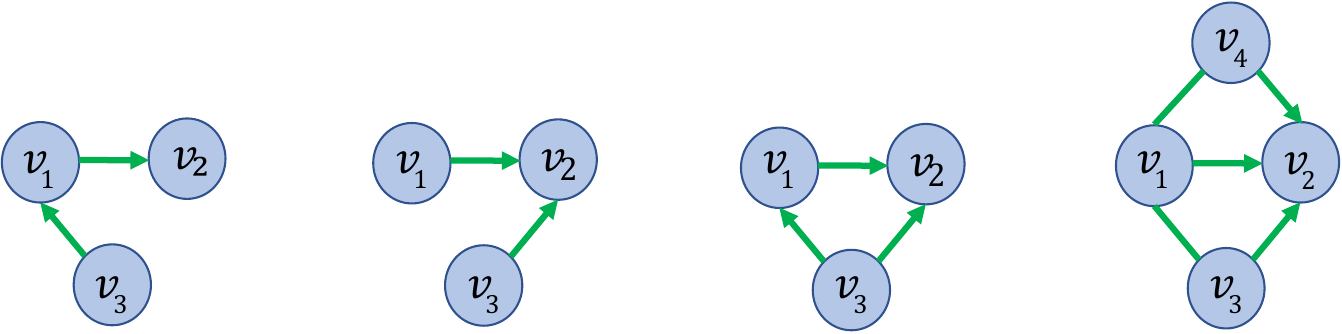}
\end{center}
\end{enumerate}

\subsection{MPDAG} \label{sec:valid-mpdag}
A PDAG $\G$ is a valid MPDAG if the following two conditions \citep{wienobst2021extendability} hold: 
\begin{enumerate}[(i)]
\item For each connected component $\sC $ spanned by undirected edges, the skeleton of $\G_{\sC}$ ($\G_{\sC}$ is the subgraph induced by $\sC$) is chordal;
\item The four structures in \cref{fig:dags}(b) do not occur as induced subgraphs of $\G$. 
\end{enumerate}
To construct an MPDAG from a CPDAG, one first orients certain undirected edges according to background knowledge, and then use the rules due to \citet{meek1995causal} to remove the forbidden structures.

\section{Proofs for generic poset results} \label{sec:proof-poset}
In this appendix, we prove the results in \cref{sec:m_distance_Graphs}.
\subsection{Proof of \cref{prop:graded}}
\label{proof:prop_graded}
Consider any path $p \in L(\G_s,\G_t)$ from $\G_s$ to $\G_t$ where $\G_s \preceq \G_t$. Let $(x_1:= \G_s, x_2,\dots,x_k:= \G_t)$ be the sequence of graphs specifying this path. By definition, for each $i = 2,\dots,k$, either $x_{i}$ covers $x_{i-1}$ or $x_{i-1}$ covers $x_i$. Since the poset is graded, we know that for each step, the rank increases or decreases by one, i.e., $|\rank(x_i)-\rank(x_{i-1})|=1$.  Note, by telescoping sum, we have:
$$\rank(\G_t)-\rank(\G_s) = \sum_{i=2}^k \rank(x_{i})-\rank(x_{i-1})$$
A term in the decomposition above is equal to $1$ when the change from $x_{i-1}$ to $x_i$ is an upward move in the poset ($x_{i-1} \preceq x_i$), and is $-1$ when the change from $x_{i-1}$ to $x_i$ is downward. Hence, 
\begin{eqnarray*}
\begin{aligned}
\rank(\G_t)-\rank(\G_s) &= \# \text{ upward moves} - \# \text{ downward moves} \\
&\leq \# \text{ upward moves} + \# \text{ downward moves} = \mathrm{length}(p).
\end{aligned}
\end{eqnarray*}
We have thus shown that the length of any path from $\G_s$ to $\G_t$ is greater than or equal $\rank(\G_t)-\rank(\G_s)$. Note that the lower-bound can be achieved by considering an upward path from $\G_s$ to $\G_t$ with the sequence $(x_1:= \G_s, x_2,\dots,x_k:= \G_t)$ where $x_{i}$ covers $x_{i-1}$ for every $i=2,3,\dots,k$. Here, the inequality above becomes an equality since $\# \text{ downward moves} = 0$.

\subsection{Proof of \cref{prop:du-up-exist}}
\begin{proof}
It is clear that the existence of a least element guarantees a down-up path between any pair of graphs in $\mathfrak{G}$, since all elements share a common lower bound. Conversely, suppose $\Lp$ does not have a least element. Then there must exist at least two minimal elements $\G_1,\G_2$ ($\G_1 \neq \G_2$), which share no common lower bound. 
Hence, no down-up path can exist between $\G_1$ and $\G_2$. 
The proof for the up-down paths follows similarly. 
\end{proof}

\subsection{Proof of \cref{prop:down_up_up_down_graded}}
\begin{proof}
Any down-up path between $\G_s$ and $\G_t$ moves downward from $\G_s$, and then upward from an inflection point $\G$ to $\G_t$. Since $\G \preceq \G_s$ and $\G \preceq \G_t$, Proposition~\ref{prop:graded} implies that the length of the path is $\rank(\G_s) + \rank(\G_t) - 2\rank(\G)$. Therefore, the shortest down-up path corresponds to an inflection point $\G$ that maximizes $\rank(\G)$ over all $\G \preceq \G_s, \G \preceq \G_t$. The proof for the up-down distance follows similarly.
\end{proof}

\subsection{Proof of \cref{cor:du-semilat}}
\begin{proof}
The result follows from the fact that $\G_s \meet \G_t$ ($\G_s \join \G_t$) has the maximum (minimum) rank among common lower (upper) bounds of $\G_s, \G_t$. 
\end{proof}

\subsection{Proof of \cref{corr:main}}
\begin{proof}
Since the poset is lower semimodular, we have from \cref{thm:semimod} that $d_{\Lp}(\G_s,\G_t) = d_{\Lp,\du}(\G_s,\G_t)$. By construction, $\Lp$ is a meet semi-lattice with $\G_s \wedge \G_t = (\sV, \sE_s^1 \cap \sE_t^1, \dots, \sE_s^k \cap \sE_t^k)$. Appealing to \cref{cor:du-semilat}, we then have $d_{\Lp}(\G_s,\G_t)  = \rank(\G_s)+\rank(\G_t)-2\rank(\G_s \wedge \G_t)=\sum_{i=1}^k \left(|\sE^i_s| + |\sE^i_t| - 2 |\sE^i_s \cap \sE^i_t|\right) =\sum_{i=1}^k |\sE^i_s \triangle \sE^i_t|$. 
\end{proof}

\subsection{Proof of Theorem~\ref{thm:semimod}} \label{proof:semimod}
In what follows, we prove that under \cref{cond:least} the following holds: 
\[ \text{$\Lp$ is lower semimodular} \iff  \text{down-up distance $d_{\Lp,\du}$ equals the model-oriented distance $d_{\Lp}$}.  \]
The case for the up-down distance follows from a dual argument. 

\begin{proof} 
\textbf{`$\Rightarrow$' direction.~} We prove by induction. Consider $\G_s, \G_t \in \g$. When $\G_s = \G_t$, both distances are zero; when $d_{\Lp}(\G_s, \G_t) = 1$, we must have either $\G_s \cover \G_t$ or $\G_s \coverby \G_t$ so $d_{\Lp,\du}(\G_s, \G_t) = 1$. Now, suppose $d_{\Lp,\du}(\G_s, \G_t) = d_{\Lp}(\G_s, \G_t)$ holds for all $\G_s, \G_t$ such that $d_{\Lp}(\G_s, \G_t) \leq k$ for a fixed $k \geq 1$. We argue by contradiction that the equality in distance must also hold for all $\G_s, \G_t$ with $d_{\Lp}(\G_s, \G_t) = k+1$, if any. Suppose this is not true for some $\G_s, \G_t \in \g$. Then, the set of shortest paths between $\G_s$ and $\G_t$ must not contain one of down-up shape. From the set, pick path $p = (\G_s = \G_0, \G_1, \dots, \G_d = \G_t)$ that has the smallest number of up-down segments, and this number is at least one because $p$ cannot be down-up. Let $(\G_a, \G_b, \G_c)$ be the maximal up-down segment that is first encountered when traversing from $\G_s$ to $\G_t$ on $p$. That is, $p$ is of the form
\[ \G_s=\G_0 \cover \cdots \cover \G_a \coverby \cdots \coverby \G_b \cover \cdots \cover \G_c \coverby \cdots \G_d = \G_t, \]
where $0 \leq a < b < c \leq d$. There are two cases. 
\begin{enumerate}[i.]
\item $a=0$ and $c=d$. Path $p$ is up-down (but not a chain). Using lower semimodularity, $p$ can be transformed into a down-up path $p'$ with the same length (see \cref{fig:semimod}) so we have a contradiction. 

\item $a>0$ or $c<d$. The subpath $p(\G_a, \G_c)$ has length smaller or equal to $k$. By our induction hypothesis, the subpath can be replaced by a down-up subpath without altering its length, but this leads to a shortest path between $\G_s, \G_t$ with one less up-down segment. This contradicts the definition of $p$. 
\end{enumerate}
Hence, we must have a contradiction and the result is proven by induction. 

 \textbf{`$\Leftarrow$' direction.~} Consider $\G_s,\G_t, \G \in \g$ such that $\G_s \neq \G_t$ and $\G \cover \G_s, \G_t$. By definition, we have $1 \leq d_{\Lp}(\G_s, \G_t) \leq 2$. 
Further, observe that $\G \cover \G_s, \G_t$ implies $\G_s$ and $\G_t$ are incomparable ($\G \cover \G_s \succ \G_t$ $\implies$ $\G$ cannot cover $\G_t$)  and hence $d_{\Lp}(\G_s,\G_t) > 1$.
Therefore, we have $d_{\Lp}(\G_s, \G_t)=2$. 
By our hypothesis, there exists a down-up path of length 2 between $\G_s$ and $\G_t$. Using again the fact that $\G_s,\G_t$ are incomparable, this path must be of the form $\G_s \cover \G' \coverby \G_t$ for some $\G' \in \g$, which concludes the proof. 
\end{proof}

\section{Posets properties for different graph classes}
This appendix contains proofs and additional results for the graph classes considered in \cref{sec:m_distance_Graphs}.

\subsection{DAGs} \label{app:dags}

\begin{proof}[Proof of \cref{prop:poset-dags}] \hfill
\begin{enumerate}
\item[1, 2 and 3.] They directly follow from \cref{prop:order-dags} and the fact that given an arbitrary $\D \in \g$, every subgraph of $\D$ is also in $\g$.
\setcounter{enumi}{3}

\item For any $\D_1 = (\sV, \sE_1)$ and $\D_2 = (\sV, \sE_2)$ in $\g$, we have $\D_1 \meet \D_2 = (\sV, \sE_1 \cap \sE_2)$.

\item To see lower semimodularity, suppose $\D_3 = (\sV, \sE_3)$ covers two different DAGs $\D_1 = (\sV, \sE_1)$ and $\D_2 = (\sV, \sE_2)$. By 3., we have $\sE_1 \not \subseteq \sE_2$, $\sE_2 \not \subseteq \sE_1$ and $|\sE_1 \triangle \sE_2| = 2$. It follows that $\D_1, \D_2 \cover \D_4:=(\sV, \sE_1 \cap \sE_2)$.
\end{enumerate}
\end{proof}

It is worth mentioning that when $|\sV| > 1$, $\Lp$ is neither a join semi-lattice nor upper semimodular. To see this, consider two DAGs that have an edge in opposite directions. They have no common upper bound by \cref{prop:order-dags}.

\subsection{CPDAGs} \label{app:cpdags}
\subsubsection{Proof of \cref{prop:poset-cpdags}}
\begin{proof}
\begin{enumerate}
\item This directly follows from definition of $\Mcpdag$. 

\item We first show the `$\Rightarrow$' direction. Suppose $\G_1 \coverby \G_2$, which implies $\G_1 \prec \G_2$. By {3.} of \cref{prop:order-cpdags}, there exists a sequence of DAGs $\D_0, \D_1, \dots, \D_k$ ($k \geq 1$) with $\D_0 \in [\G_1]$, $\D_k \in [\G_2]$ and every adjacent pair corresponds to either operation (i) or (ii). Let the subsequence $\D_{j_1}, \D_{j_2}, \dots, \D_{j_l}$ $(1 \leq j_1 < \dots < j_l \leq k)$ be those result from operation (ii). First, observe that $l \geq 1$ because otherwise only operation (i) is involved and $\D_0, \D_k$ must be Markov equivalent, contradicting $\G_1 \prec \G_2$. Then, we argue that $l < 2$. To see this, suppose $l \geq 2$. Let $\G_3$ be the CPDAG that represents the Markov equivalence class of $\D_{j_1}$. It then follows that $\G_1 \prec \G_3 \prec \G_2$, which contradicts $\G_1 \coverby \G_2$. Hence, we conclude that $l=1$. Because operation (i) stays in the Markov equivalence class, we have $\D_{j_1-1} \in [\G_1]$, $\D_{j_1} \in [\G_2]$ and $\D_{j_1-1}$ is a subgraph of $\D_{j_1}$ with exactly one edge fewer. 

\smallskip \noindent We now show the `$\Leftarrow$' direction. For $\G_s = (\sV, \sE_s^{-}, \sE_s^{\rightarrow})$, $\G_t = (\sV, \sE_t^{-}, \sE_t^{\rightarrow})$, suppose $\D_1 \in [\G_s]$, $\D_2 \in [\G_t]$ and $\D_1$ is a subgraph of $\D_2$ with exactly one edge fewer, which implies 
\[(|\sE_t^{-}| + |\sE_t^{\rightarrow}|) - (|\sE_s^{-}| + |\sE_s^{\rightarrow}|) = 1.\]
We want to show $\G_s \coverby \G_t$. By {3.} of \cref{prop:order-cpdags}, we know $\G_s \prec \G_t$. Suppose $\G_s$ is not covered by $\G_t$. Then, the chain can be `saturated' as $\G_s = \G_1 \coverby \G_2 \coverby \dots \coverby \G_k = \G_t$ with $k > 2$. Applying the result in the `$\Rightarrow$' direction above to every adjacent pair in the chain, we have 
\[ (|\sE_t^{-}| + |\sE_t^{\rightarrow}|) - (|\sE_s^{-}| + |\sE_s^{\rightarrow}|) = (k-1), \]
contradicting the previous display. Hence, we have $\G_s \coverby \G_t$. 

\item By {2.}, every covering relation corresponds to increasing the number of edges by one. Further, by {1.}, the least element $\Glst$ has no edge. This fulfills the definition of $\rank(\cdot)$. 

\end{enumerate}
\end{proof}

We now show that the model-oriented poset need not be a meet or join semi-lattice in general. Let $\Lp$ be the poset for CPDAGs over $\sV = \{1,2,3\}$. Consider the graphs
\[ \G_1: 1 - 2 \phantom{-} 3, \quad \G_2: 1 \phantom{-} 2 - 3, \quad \G_3: 1 \rightarrow 2 \leftarrow 3, \quad \G_4: 1 - 2 - 3.\]
By {2.} of \cref{prop:poset-cpdags}, we have $\G_1, \G_2 \coverby \G_3, \G_4$. To see that $\Lp$ is not a join semi-lattice, observe that the upper bounds on $\G_1$ and $\G_2$ consist of $\G_3,\G_4$ and the full graph $\Ggr$, where $\G_3, \G_4 \prec \Ggr$ but $\G_3
$ and $\G_4$ are incomparable. Similarly, to see that $\Lp$ is not a meet semi-lattice, observe that the lower bounds on $\G_3$ and $\G_4$ consists of $\G_1$, $\G_2$ and the empty graph $\Glst$. However, while $\G_1, \G_2 \succ \Glst$, $\G_1$ and $\G_2$ are incomparable. 

Further, we show that the poset need not be upper or lower semimodular. Specifically, let $\Lp'$ be the model-oriented poset for CPDAGs over $\sV' = \{1,2,3,4\}$. For upper semimodularity, consider graphs
\[ \G_1': \begin{array}{ccc}
1 & \leftarrow & 2 \\
 \uparrow &    & |  \\
4 & -           & 3   
\end{array}, \qquad \G_2': \begin{array}{ccc}
1 & \rightarrow & 2 \\
| &             & \uparrow  \\
4 &     -       & 3   
\end{array}, \qquad \G_3': \begin{array}{ccc}
1 &  & 2 \\
|  &   & |  \\
4 & -  & 3   
\end{array}.\]
By {2.} of \cref{prop:poset-cpdags}, we know $\G_1', \G_2' \cover \G_3'$. If $\Lp'$ is upper semimodular, then there must exist a graph that covers both $\G_1'$ and $\G_2'$. Suppose $\G'_4 \cover \G_1', \G_2'$. By the same proposition, the \emph{skeleton} of $\G'_4$ must be either that of $\G_5'$ or that of $\G_6'$:
\[ \G_5': \begin{array}{ccc}
1 & - & 2 \\
|  & \diagdown   & |  \\
4 &      -      & 3   
\end{array}, \qquad \G_6': \begin{array}{ccc}
1 & - & 2 \\
|  & \diagup   & |  \\
4 &      -      & 3   
\end{array}.\]
Observe that $\G_4'$ cannot have a v-structure at $1$ because if so, then every DAG in $[\G_4']$ has $1 \leftarrow 2$ and thus cannot be a supergraph of $\G_2'$. By the same logic, we can see that $\G_4'$ has no v-structure. Consequently, $\G_4'$ is identical to either $\G_5'$ or $\G_6'$. However, by model containment, $\G_4'$ must satisfy both $1 \not \indep 3 \mid 2,4$ (holds in $\G_2'$) and $2 \not \indep 4 \mid 1,3$ (holds in $\G_1'$). One can check that these two conditions cannot simultaneously hold under $\G_5'$ or $\G_6'$.

Now we argue that $\Lp'$ is not lower semimodular. For this purpose, consider again $\G_5'$ and $\G_6'$ above.
\cref{prop:poset-cpdags} implies that the full graph covers $\G_5'$ and $\G_6'$. If $\Lp$ is lower semimodular, then there must exist $\G_7'$ that is covered by both $\G_5'$ and $\G_6'$. By the same result, the skeleton of $\G_7'$ is a 4-cycle and hence $\G_7'$ must be either 
\[ \begin{array}{ccc}
\circ & \rightarrow & \circ \\
\downarrow  &   & \uparrow  \\
\circ &      \leftarrow      & \circ   
\end{array} \quad \text{or} \quad \begin{array}{ccc}
\circ & \rightarrow & \circ \\
|  &   & \uparrow  \\
\circ &      -      & \circ   
\end{array}\]
up to graph isomorphism. One can check that neither can be covered by both $\G_5'$ and $\G_6'$. 

\subsubsection{A zigzag example} \label{app:cpdag_ex}
\cref{fig:zigzag} gives an example where the shortest path between two CPDAGs is a zigzag, neither up-down nor down-up. 
\begin{figure}[!ht]
    \includegraphics[width=0.6\columnwidth]{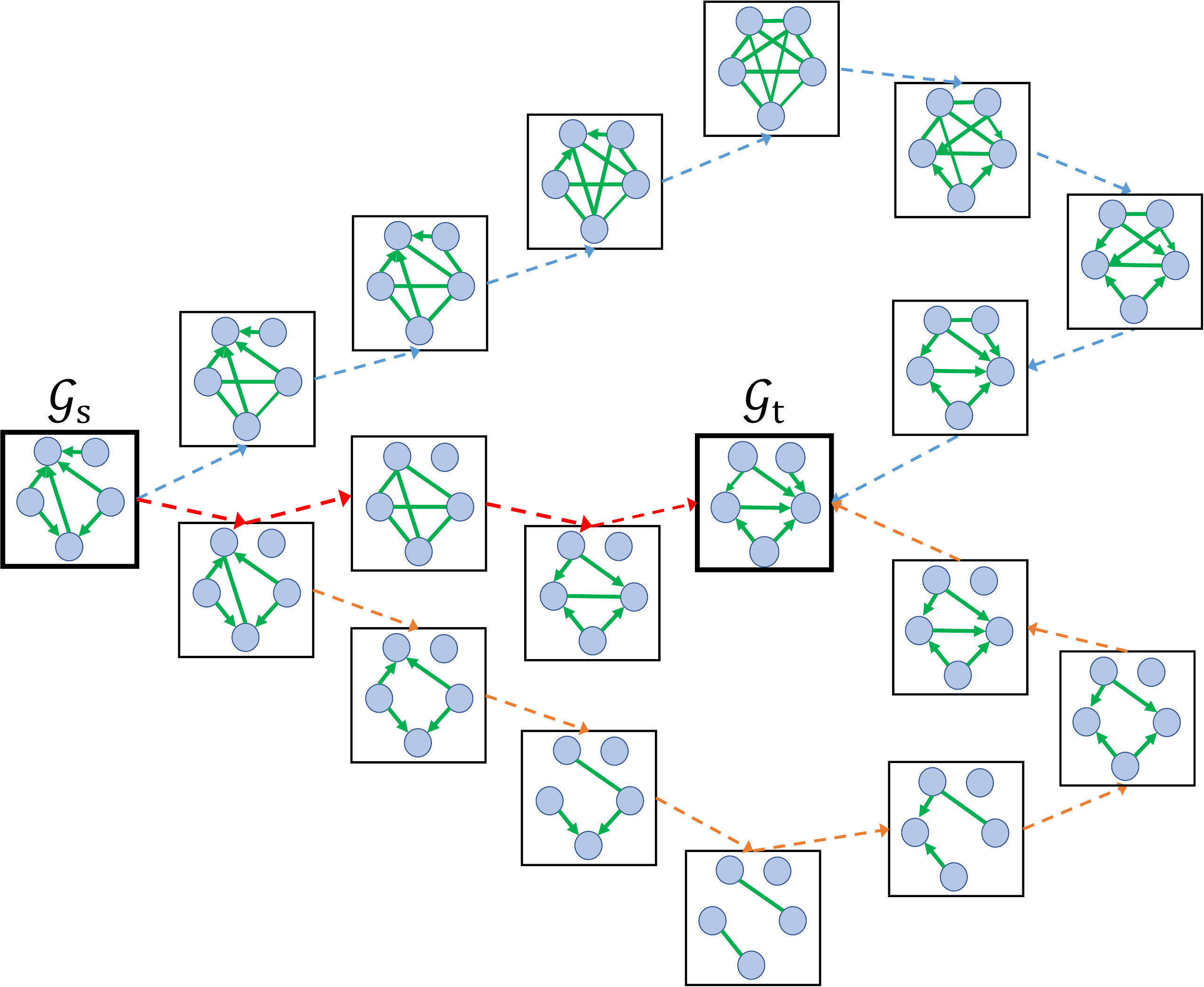}
    \caption{The model-oriented distance between probabilistic CPDAGs $\G_s$ and $\G_t$ is 4, as shown by the shortest path drawn in red. Meanwhile, the up-down (blue path) and the down-up (orange path) distances are both 8.}
    \label{fig:zigzag}
\end{figure}

\subsubsection{Lower and upper bounds on the model-oriented distance}
\label{app:upper_lower}
To formulate the lower bound, we need some notation. 
For a PDAG $\G$, we use $\sk(\G)$ to denote its \emph{skeleton}, which is the set of edges upon undirecting all the edges in $\G$. 
We also use $\vstr(\G)$ to denote the set of triplets $(a,b,c)$ that form a v-structure at $b$ in $\G$ (see \cref{sec:graph_background} for its definition); here we do not distinguish between $(a,b,c)$ and $(c,b,a)$. 
Let us fix two CPDAGs $\G_s$ and $\G_t$ over a vertex set $\sV$. Define the set 
\begin{multline*}
\sS(\G_s,\G_t) := \Big\{(a,b,c) \subseteq \sV \times \sV \times \sV: (a,b),(b,c) \in \sk(\G_s) \cap \sk(\G_t),\\(a,c) \not \in \sk(\G_s) \cup \sk(\G_t), \text{ and } (a,b,c) \in \vstr(\G_s) \triangle \vstr(\G_t) \Big\},
\end{multline*}
which consists of v-structures that appear in only one of the two graphs. 
We \emph{partition} $\sS$ as
\[\sS(\G_s,\G_t) = \sS_1 \cup \dots \cup \sS_L, \quad L \geq 1, \] 
where two triplets $(a,b,c)$ and $(d,e,f)$ are placed in the same partition if they share any common vertex (e.g., $a=d$ or $a=e$). 

For a pair of vertices $(i,j)$, we define operations $\OP_{(i,j),\text{CE}}$ (cover edge) and $\OP_{(i,j),\text{ER}}$ (edge removal) that map a subset of $\sS(\G_s,\G_t)$ to a subset of $\sS(\G_s,\G_t)$ as follows: 
for $\sT \subseteq \sS(\G_s,\G_t)$, 
\[\OP_{(i,j),\text{CE}}(\sT) := \sT \setminus \{(i,k,j): k \in \sV\},\quad
\OP_{(i,j),\text{ER}}(\sT) := \sT \setminus \{(i,j,k),(j,i,k): k \in \sV\}. \]
Here, $\OP_{(i,j),\text{CE}}$ removes v-structures of the form $i \rightarrow k \leftarrow j$ by adding a \emph{covering edge} (see \cref{prop:order-cpdags}) between $i$ and $j$; $\OP_{(i,j),\text{ER}}$ removes v-structures of the form $i \rightarrow j \leftarrow k$ or $j \rightarrow i \leftarrow k$ by removing the edge between $i$ and $j$.
For each partition $\sS_l$, consider the minimum number of operations taken to eliminate all the v-structures in the partition, namely
\begin{equation} \label{eqs:mop}
m_{\text{op}}(\sS_l) := \min \left\{m \in \mathbb{N}: \exists~\mathrm{op}_1,\dots,\mathrm{op}_m,\; \mathrm{op}_1 \circ \dots \circ \mathrm{op}_m(\sS_l) = \emptyset \right\}, \quad l=1,\dots,L,
\end{equation}
where each $\mathrm{op}_k$ ($1\leq k \leq m$) is either $\OP_{(i,j),\text{CE}}$ or $\OP_{(i,j),\text{ER}}$ for some $(i,j)$ that can depend on $k$. Because the v-structures from two different sets of the partition have no common vertex by construction and this property is maintained by both operations, we conclude that 
\[ m_{\text{op}}(\sS(\G_s,\G_t)) = \sum_{l=1}^L m_{\text{op}}(\sS_l), \]
which lower bounds the number of steps taken to align the v-structures between $\G_s$ and $\G_t$. 

\begin{proposition}[CPDAG $d_{\Lp}$ bounds]\label{prop:cpdags}
For CPDAGs $\G_s,\G_t \in \g$, we have upper bound
\[d_{\Lp}(\G_s,\G_t) \leq \min\left\{d_{\Lp,\du}(\G_s,\G_t), d_{\Lp,\ud}(\G_s,\G_t) \right\},\]
where $d_{\Lp,\du}, d_{\Lp,\ud}$ can be computed using \cref{prop:down_up_up_down_graded}.
Meanwhile, we have lower bound
\[d_{\Lp}(\G_s,\G_t) \geq |\sk(\G_s) \triangle \sk(\G_t)| + 2 \, m_{\mathrm{op}}(\sS(\G_s,\G_t)). \]
\end{proposition}

The upper bound follows from the definition of $d_{\Lp}$ and the fact that both the down-up and up-down distances are well-defined for CPDAGs, as guaranteed by the existence of the least and the greatest element in its poset.

\begin{proof}[Proof of the lower bound in \cref{prop:cpdags}] 

Any path from $\G_s$ to $\G_t$ must involve removing edges that are in the skeleton of $\G_s$ and not in the skeleton of $\G_t$ and adding edges that are in the skeleton of $\G_t$ but not in the skeleton of $\G_s$. Since every edge addition or deletion increases the path length by one, we have that $d_{\Lp}(\G_s,\G_t) \geq |\sk(\G_s) \triangle \sk(\G_t)|$. 

The lower bound above does not account for v-structures that are present in one but not in the other. Specifically, suppose $\G_s$ has a v-structure $i \rightarrow j \leftarrow k$ with $i$ and $k$ not connected, and $\G_t$ does not have this v-structure, but has an edge between $i$ and $j$, an edge between $j$ and $k$, and no edge between $i$ and $k$. Thus, any path from $\G_s$ to $\G_t$ must remove this v-structure, which can only be done if and only if at some step at least one of the following happens: (i) the edge between $i$ and $j$ and/or between $j$ and $k$ is removed (this is the operation $\OP_{(i,j),\text{ER}}$); or (ii) there is a covering edge that is added between $i$ and $k$ (this is the operation $\OP_{(i,j),\text{CE}}$). For both (i) and (ii), an extra operation afterward is needed: for (i), it is to add back the edge; and for (ii), it is to remove the covered edge. Similar argument can be made if $\G_t$ has a v-structure that $\G_s$ does not.

The set $\sS(\G_s,\G_t)$ precisely defines the set of relevant triplets described in the previous paragraph.  Let $\mathrm{OP} := \{\OP_{(i,j),\text{CE}},\OP_{(i,j),\text{ER}} \forall i,j \in \mathrm{V}\}$. Our goal then becomes finding the smallest number of operations, from the set of operations $\mathrm{OP}$ such that, when applied consecutively to $S(\G_s,\G_t)$, the resulting set is empty. In other words, we obtain that:
$$d_{\Lp}(\G_s,\G_t) \geq |\mathrm{sk}(\G_s)\setminus \mathrm{sk}(\G_t)|+|\mathrm{sk}(\G_t)\setminus \mathrm{sk}(\G_s)|+2m_{\text{op}}(\sS(\G_s,\G_t)).$$
where 
$$m_{\mathrm{op}}(\sS):= \min\{m \in \mathbb{N}: \exists \mathrm{op}_1,\dots,\mathrm{op}_{m} \text{ such that }\mathrm{op}_1 \circ \dots \circ \mathrm{op}_m(\sS)=\emptyset\}.$$
Consider the partition $\sS(\G_s,\G_t) = \sS_1 \cup \dots \sS_L$. Since triplets in distinct partitions do not share any common vertices, we have that:
$$m_{\mathrm{op}}(\sS(\G_s,\G_t)) = \sum_{\ell=1}^L m_{\mathrm{op}}(\sS_{\ell}).$$
\end{proof}

\subsubsection{Additional results} \label{sec:algorithm_theory}
\begin{lemma} For the class of CPDAGs, and for any pair of CPDAGs $\G_1$ and $\G_2$, we have:
$$\max_{\G \preceq \G_1, \G \preceq \G_2}\rank(\G) \leq |\mathrm{sk}(\G_1)\cap \mathrm{sk}(\G_2)|.$$
\label{lemma:cpdag_rho}
\end{lemma}
\begin{proof}[Proof of Lemma~\ref{lemma:cpdag_rho}]
Notice that for any $\G \preceq \G_1$, we have that $\mathrm{sk}(\G)\subseteq \mathrm{sk}(\G_1)$; otherwise, conditional dependencies of $\G$ are not a subset of those of $\G_1$. Similarly, we have that for any $\G \preceq \G_2$, $\mathrm{sk}(\G)\subseteq \mathrm{sk}(\G_2)$. This allows us to conclude that for $\G$ with $\G \preceq \G_1$ and $\G \preceq \G_2$, $\mathrm{sk}(\G)\subseteq \mathrm{sk}(\G_1) \cap \mathrm{sk}(\G_2)$. Since $\rank(\G) = |\mathrm{sk}(\G)|$, we have the desired result. 
\end{proof}

\begin{lemma} For the class of CPDAGs, and for any pair of CPDAGs $\G_1$ and $\G_2$, we have:
$$\min_{\G_1 \preceq \G, \G_2 \preceq \G}\rank(\G) \geq |\mathrm{sk}(\G_1)\cup \mathrm{sk}(\G_2)|.$$
\label{lemma:cpdag_rho2}
\end{lemma}
\begin{proof}[Proof of Lemma~\ref{lemma:cpdag_rho2}]
Based on the same argument as proof of Lemma~\ref{lemma:cpdag_rho}, we have that for 
$\G$ with $\G_1 \preceq \G$ and $\G_2 \preceq \G$, $\mathrm{sk}(\G_1) \cup \mathrm{sk}(\G_2) \subseteq \mathrm{sk}(\G)$. Since $\rank(\G) = |\mathrm{sk}(\G)|$, we have the desired result. 
\end{proof}

\subsubsection{Polytree CPDAGs} \label{sec:apdx-polytree-cpdag}

\begin{proof}[Proof of \cref{thm:poset-polytree-cpdags}] \hfill
\begin{enumerate}
\item This directly follows from definition. 

\item The `if' part follows from the definition of $\Mcpdag$ and now we show the `only if' part. 
First, note that $\sk(\G_1) \subseteq \sk(\G_2)$; otherwise conditional dependencies of $\G_1$ are not a subset of those of $\G_2$. 
Further, if $\G_1$ has a v-structure $i \rightarrow j \leftarrow k$, then $\G_2$ must also have this v-structure. 
Supposing not, then, since $\G_2$ cannot have cycles and its skeleton is a superset of $\G_1$, then $\G_2$ must have an edge between $i$ to $j$ and $j$ to $k$ without an edge between $i$ to $k$. 
This leads to a contradiction, however, since $\G_1$ encodes the conditional dependencies $X_i \not\indep X_j \mid X_{\sS}$ for every set $\sS$ that contains $k$, while $\G_2$ does not encode these dependencies. 
Similarly, if $\G_2$ has a v-structure $i \rightarrow j \leftarrow k$, then either $\G_1$ has the same v-structure or the edge $i$ to $j$ or $j$ to $k$ is missing in $\G_1$. 
Suppose not, that is $i$ and $j$, and $j$ and $k$ are connected but not in a v-structure form. 
Then, $\G_1$ implies the conditional dependency $X_i \not\indep X_j \mid \sS$ for every $\sS$ that does not contain $k$. 
However, $\G_2$ implies the conditional independence $X_i \indep X_j \mid X_{\sS}$ for some $\sS$ that does not contain $k$. 

Using the facts above, we conclude the desired result. 
Let $\sF = \{(i,j)\}$ be collection of vertices that are connected in $\G_2$ but not in $\G_1$. 
By the analysis above, any v-structure in $\G_1$ is also present in $\G_2$. Furthermore, every v-structure $i \rightarrow k \leftarrow j$  in $\G_2$ is either present in $\G_1$ or at least one of the edges between pair of vertices $(i,k)$ and $(j,k)$ is missing. 
Consider any DAG $\mathcal{D}_2$ represented by $\G_2$. Let $\mathcal{D}_1$ be the DAG that is obtained by deleting the edges in $\mathcal{D}_2$ among pairs of vertices $\sF$. 
By construction, $\mathcal{D}_1$ has the same skeleton as any DAG in $\G_1$. Since $\mathcal{D}_1$ is a polytree, every structure present in $\mathcal{D}_1$ or in $\G_1$ are also in $\G_2$. 
Suppose as a point of contradiction that there exists a v-structure $i \rightarrow k \leftarrow j$ that is in $\mathcal{D}_1$ and not in $\G_1$. 
Since this v-structure must also be present in $\G_2$, and from the previous statement, then one of the edges between pairs of vertices $(i,k)$ or $(j,k)$ must be missing in $\G_1$, which contradicts $\mathcal{D}_1$ and $\G_1$ have the same skeleton. 
A similar argument allows us to conclude that it is not possible that $\G_1$ has a v-structure that is not present in $\mathcal{D}_1$. 
Thus, $\mathcal{D}_1$ and $\G_1$ have the same skeleton and v-structures. 
We conclude that $\mathcal{D}_1$ must be a DAG represented by $\G_1$. 

\item Let $\rank(\cdot)$ be the number of edges. 
It suffices to show that $\rank(\G)+1 = \rank(\G')$ holds for every $\G \coverby \G'$ in $\Lp_{\poly}$. 
By the previous result, $\G \coverby \G'$ implies $\rank(\G') \geq \rank(\G) + 1$. 
For contradiction, suppose there exists $\G_1 \coverby \G_2$ with $\rank(\G_1)+1 < \rank(\G_2)$. 
By the previous result, for every DAG $\mathcal{D}_2 = (\mathrm{V},\sE_2)$ represented by $\G_2$, there exists a DAG $\mathcal{D}_1 = (\mathrm{V},\sE_1)$ represented by $\G_1$ such that $\sE_1 \subseteq \sE_2$. 
Since $\rank(\G_1)+1 < \rank(\G_2)$, there exists $\sE_3$, satisfying $\sE_1 \subset \sE_3 \subset \sE_2$. 
Let $\G_3$ be the CPDAG obtained by completing the DAG $\mathcal{D}_3 = (\mathrm{V},\sE_3)$. 
By construction, we have $\G_1 \preceq \G_3 \preceq \G_2$ with $\G_3$ being distinct from $\G_1$ and $\G_2$, contradicting $\G_1 \coverby \G_2$. 

\item Consider any $\G_1,\G_2,\G_3 \in \g_{\poly}$ with $\G_1 \coverby \G_3$ and $\G_2 \coverby \G_3$. We want to show that there is a graph that is covered by both $\G_1$ and $\G_2$. 
Consider any DAG $\mathcal{D}_3=  (\sV,\sE_3) \in [\G_3]$. 
By Result 2., there exist DAGs $\mathcal{D}_1 = (\sV,\sE_1) \in [\G_1]$ and $\mathcal{D}_2 = (\sV,\sE_2) \in [\G_2]$ with $\sE_1 \subseteq \sE_3$ and $\sE_2 \subseteq \sE_3$.
Applying Result 3., we have $\sE_3\setminus \sE_1 = e_1$ and $\sE_3 \setminus \sE_2=e_2$ for two separate edges $e_1, e_2$. 
Consider DAG $\tilde{\mathcal{D}} := (\sV, \sE_1 \cap \sE_2)$ and let $\tilde{\G}$ be the CPDAG such that $\tilde{\mathcal{D}} \in [\tilde{\G}]$. 
By construction, $\tilde{\G}$ is a polytree CPDAG that satisfies $\tilde{\G} \preceq \G_1$ and $\tilde{\G} \preceq \G_2$. 
Since $\rank(\tilde{\G}) = \rank(\G_1)-1=\rank(\G_2)-1$, we have $\tilde{\G} \coverby \G_1$ and $\tilde{\G} \coverby \G_2$. 

\item This follows from the previous result and \cref{thm:semimod}. 
\end{enumerate}
\end{proof}

\subsection{MPDAGs} \label{app:mpdags}
\subsubsection{Proofs for the MPDAG poset}
We first show that the model-oriented poset underlying MPDAGs is not a meet or join semi-lattice in general. 
Specifically, let $\Lp$ be the model-oriented poset underlying MPDAGs over $\sV=\{1,2,3\}$. Consider the following four MPDAGs:
\[ \G_1: \begin{array}{ccc}
1 & \rightarrow & 2 \\
  &    & \uparrow \\
 &           & 3   
\end{array}, \qquad \G_2:\begin{array}{ccc}
1 & - & 2 \\
  &    & |\\
 &           & 3   
\end{array}, \qquad \G_3:\begin{array}{ccc}
1 & \rightarrow & 2 \\
  &    & \\
 &           & 3   
\end{array} \qquad \G_4:\begin{array}{ccc}
1 & \phantom{-} & 2 \\
  &    & \uparrow\\
 &           & 3   
\end{array},\]
which satisfy $\G_3, \G_4 \prec \G_1, \G_2$. The upper bounds of $\G_3$ and  $\G_4$ consist of $\G_1, \G_2$ and $\Ggr$, where $\G_1$ and $\G_2$ are incomparable --- this shows that $\Lp$ is not a join semi-lattice. Similarly, the lower bounds of $\G_1$ and $\G_2$ consist of $\G_3, \G_4$ and $\Glst$, where $\G_3, \G_4$ are incomparable --- this shows that $\Lp$ is not a meet semi-lattice.

Then, we show that $\Lp$ is neither lower nor upper semimodular. Consider graphs 
\[ \G_1': \begin{array}{ccc}
1 & \rightarrow & 2 \\
  &    & \downarrow\\
 &           & 3   
\end{array}, \qquad \G_2':\begin{array}{ccc}
1 &  & 2 \\
  &  \nwarrow  & \downarrow\\
 &           & 3   
\end{array}, \qquad \G_3':\begin{array}{ccc}
1 &\phantom{-} & 2 \\
  &    & \downarrow\\
 &           & 3   
\end{array},\]
which satisfy $\G_3' \coverby \G_1', \G_2'$. However, there is no graph that covers both  $\G_1'$ and $\G_2'$ due to the acyclicity constraint, establishing that $\Lp$ is not upper semimodular. Further, consider the following graphs: 
\[ \G_1'': \begin{array}{ccc}
1 & - & 2 \\
  &   & |\\
4 &           & 3   
\end{array}, \qquad \G_2'':\begin{array}{ccc}
1 &{-} & 2 \\
\uparrow  &  \nearrow &\downarrow \\
4 &  \rightarrow         & 3   
\end{array}, \qquad \G_3'':\begin{array}{ccc}
1 &{-} & 2 \\
\uparrow  &  \nearrow &| \\
4 &  \rightarrow         & 3   
\end{array}.\]
By construction, $\G_1''$ and $\G_2''$ are both covered by $\G_3''$ (see \cref{sec:mpdags}). 
However, there exists no graph covered by both $\G_1''$ and $\G_2''$. To see why, the only two graphs that are covered by $\G_1''$ are $\G_1'''$ and $\G_2'''$ below:
\[ \G_1''': \begin{array}{ccc}
1 & \leftarrow & 2 \\
  &   & |\\
4 &           & 3   
\end{array},  \qquad \G_2''':\begin{array}{ccc}
1 &{-} & 2 \\
  &   &\downarrow \\
4 &          & 3   
\end{array}, \qquad \G : \begin{array}{ccc}
1 & - & 2 \\
 \uparrow & \nearrow   & \downarrow\\
4 &           & 3   
\end{array}.\]
Note that $\G_1'''$ does not satisfy $\G_1'''\preceq \G_2''$. Moreover, there exists the graph $\G$ which is intermediate between $\G_2'''$ and $\G_2''$. Hence, $\Lp$ is not lower semimodular. 

 Finally, we show that the MPDAG poset need not be graded.

\begin{proof}[Proof of \cref{prop:not_graded_mpdags}]
To prove that the poset $\Lp$ underlying MPDAGs is not graded, it suffices to find two maximal chains between a pair of graphs $\G_s$ and $\G_t$ of unequal lengths. 
We first show this under $|\sV|=4$.
\cref{fig:mpdag_not_graded_1} presents a sequence of graphs that specifies a maximal chain from $\G_s$ to $\G_t$. 
This chain has length $3$. 
Meanwhile, \cref{fig:mpdag_not_graded_2} presents another maximal chain from $\G_s$ to $\G_t$ that has length 5. 
This shows that $\Lp$ is not graded. When $|\sV| > 4$, the same argument works by adding in the remaining vertices as isolated vertices with degree zero. 
\end{proof}

\begin{figure}[!ht]
  \subfloat[$\G_1:=\G_s$]{\includegraphics[width=.2\columnwidth, angle = 0]{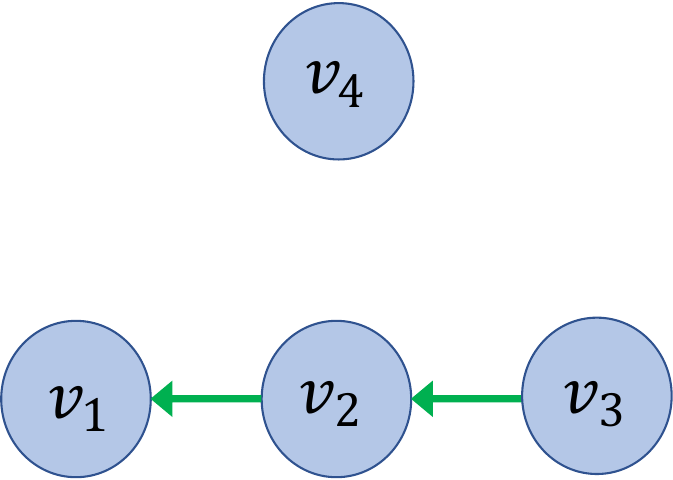}} \hspace{0.3in}\  \subfloat[$\G_2$]{\includegraphics[width=.2\columnwidth, angle = 0]{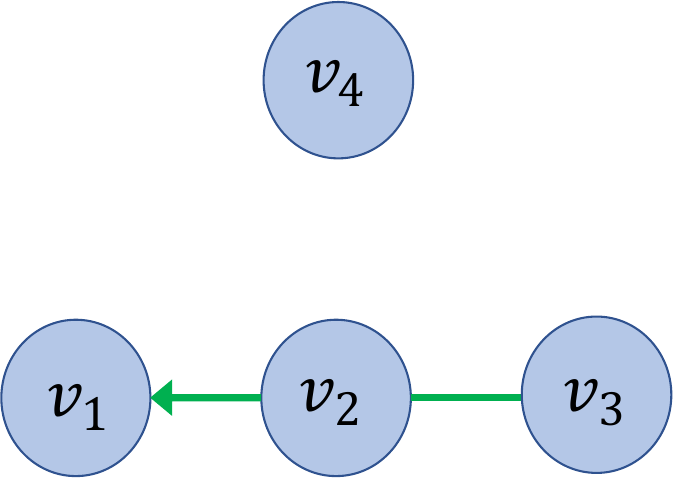}} \hspace{0.3in} 
        \subfloat[$\G_3$]{\includegraphics[width=.2\columnwidth, angle = 0]{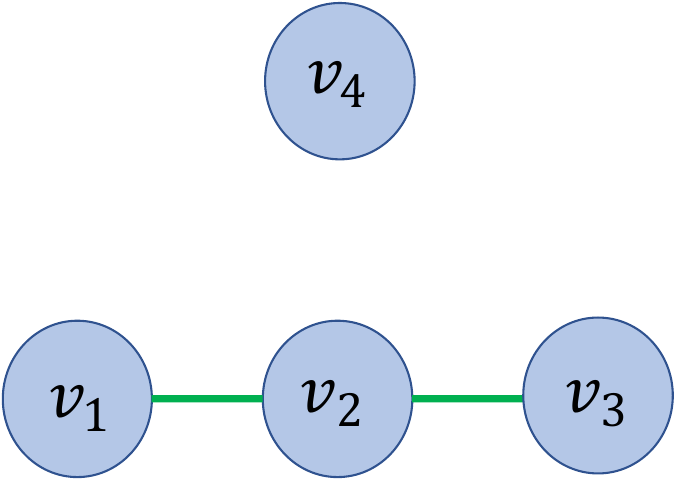}}\hspace{0.3in}  
        \subfloat[$\G_4:=\G_t$]{\includegraphics[width=.2\columnwidth, angle = 0]{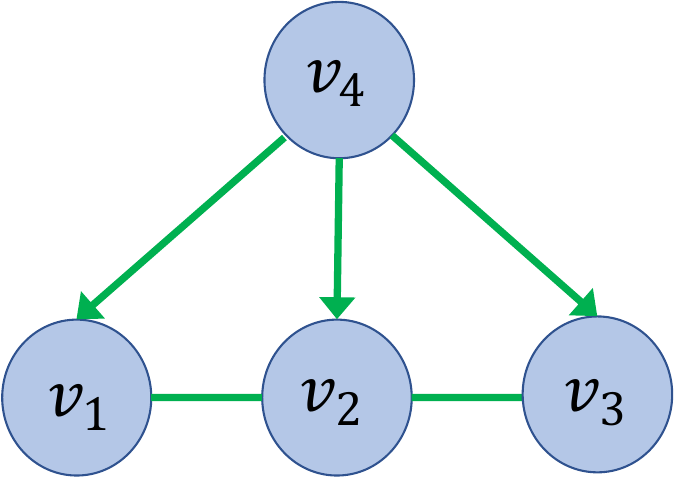}}
\caption{A maximal chain between MPDAGs $\G_s$ and $\G_t$, where each graph covers its predecessor. To see $\G_3 \coverby \G_{4}$, observe that (i) $\G_{3} \prec \G_{4}$ and (ii) removing any directed edge from $\G_{4}$ introduces a structure that is forbidden from any MPDAG (see \cref{fig:dags}), so no graph sits between $\G_{3}$ and $\G_{4}$. }
\label{fig:mpdag_not_graded_1}
\end{figure}

\begin{figure}[!ht]
  \subfloat[$\G_1:=\G_s$]{\includegraphics[width=.2\columnwidth, angle = 0]{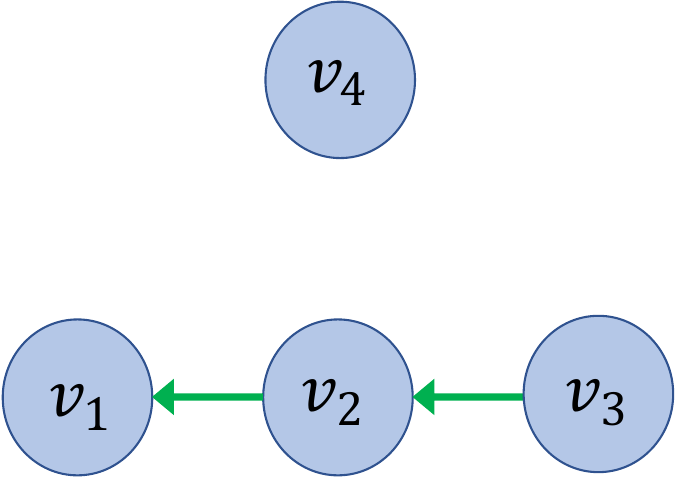}} \hspace{0.3in}\  \subfloat[$\G_2$]{\includegraphics[width=.2\columnwidth, angle = 0]{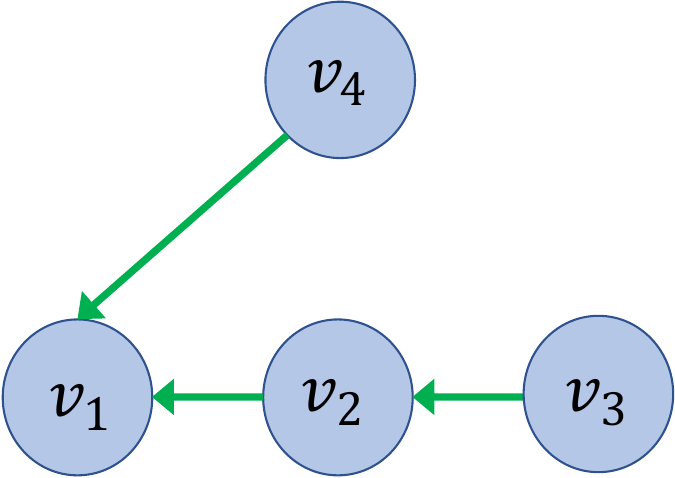}} \hspace{0.3in} 
        \subfloat[$\G_3$]{\includegraphics[width=.2\columnwidth, angle = 0]{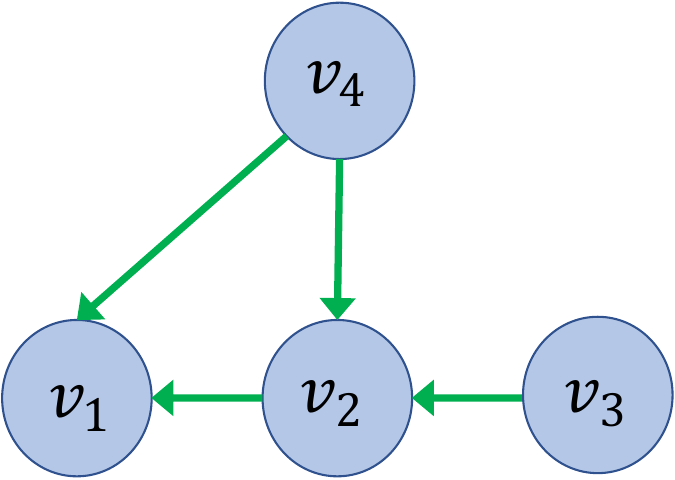}}\hspace{0.3in}  
        \subfloat[$\G_4$]{\includegraphics[width=.2\columnwidth, angle = 0]{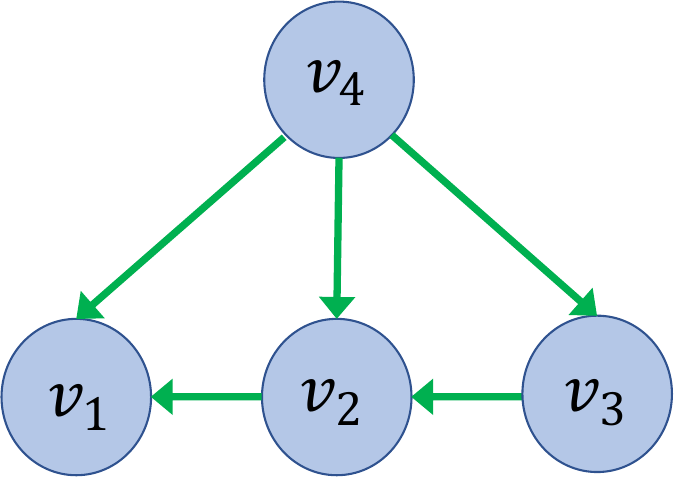}}\hspace{0.3in}
             \subfloat[$\G_5$]{\includegraphics[width=.2\columnwidth, angle = 0]{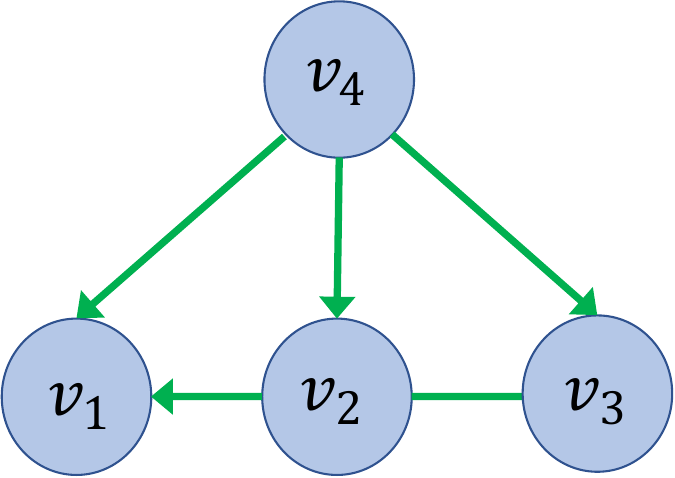}}\hspace{0.3in}
             \subfloat[$\G_6:=\G_t$]{\includegraphics[width=.2\columnwidth, angle = 0]{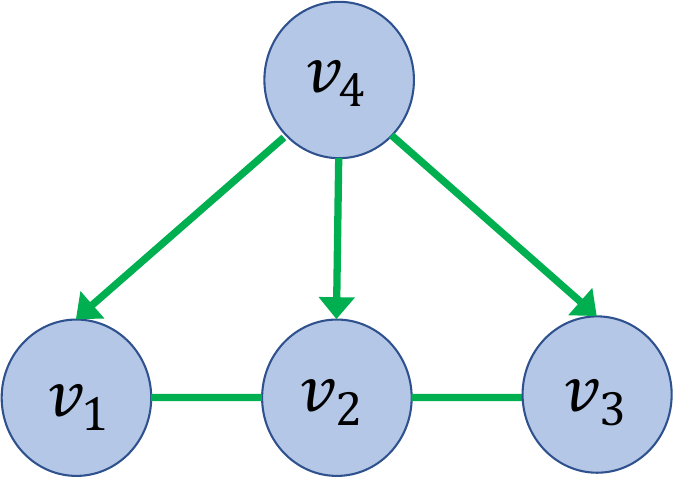}}
\caption{A maximal chain between MPDAGs $\G_s$ and $\G_t$, where each graph covers its predecessor.}
\label{fig:mpdag_not_graded_2}
\end{figure}

\subsubsection{Proof of \cref{prop:pseudo-covering}}
\label{proof:mpdag_pseudo}

\begin{proof}[Proof of \cref{prop:pseudo-covering}]
Given $\G_1 \preceq \G_2$, from \cref{prop:order-mpdags} it follows that every $v_1 \rightarrow v_2$ present in $\G_1$ corresponds to either $v_1 \rightarrow v_2$ or $v_1-v_2$ in $\G_2$, and every $v_3-v_4$ present in $\G_1$ corresponds to $v_3-v_4$ in $\G_2$. The first statement then follows from definition.

Further, suppose $\pseudorank(\G_1)+1=\pseudorank(\G_2)$ but there exists $\G$ that satisfies $\G_1 \prec \G \prec \G_2$. The skeletons of $\G_1$ and $\G_2$ are either identical or different by exactly one edge, which we discuss in three cases: (i) If their skeletons coincide, then $\G_2$ must have exactly one undirected edge that is directed in $\G_1$, so the intermediate $\G$ cannot exist. 
(ii) If $\sk(\G_1) = \sk(\G)$ has one fewer edge than $\sk(\G_2)$, then the extra edge in $\G_2$ must be directed and all the other edges in $\G_2$ are identical to those in $\G_1$. Because $\G_1 \prec \G$, $\G$ must have an undirected edge that corresponds to a directed edge in $\G_1$ and hence in $\G_2$, but this contradicts $\G \prec \G_2$. (iii) Similarly, we also have contradiction if $\sk(\G) = \sk(\G_2)$ has one more edge than $\sk(\G_1)$. Therefore, by ruling out the intermediate $\G$, we conclude $\G_1 \coverby \G_2$. 
\end{proof}

\subsubsection{Proof of \cref{prop:mpdag_upper}}
\label{proof:mpdag_upper}

We will use the following terminology in this subsection. 
A vertex set is called a \emph{clique} if every pair of vertices in the set are adjacent.
A vertex $v$ in $\G$ is called \emph{simplicial} if $\Nb_{\G}(v)$ is a clique. 
The topological order is a partial order on the set of vertices: for two vertices $u$ and $v$, we say $u \preceq v$ if $u \rightarrow \dots \rightarrow v$. 
Given a vertex set $\sS$ that contains a vertex $u$, we say $u$ is topologically minimal in $\sS$ if there exists no vertex $v$ in $\sS$ such that $v \prec u$. 
We first prove two supporting lemmas.

\begin{lemma} \label{lem:simplicial}
Let $\H$ be an MPDAG where every pair of vertices are connected by an undirected path. Suppose $u$ is a simplicial vertex and $u$ has a child $v$ with $u - s - v$ in $\H$. Then $v$ is also simplicial. 
\end{lemma}

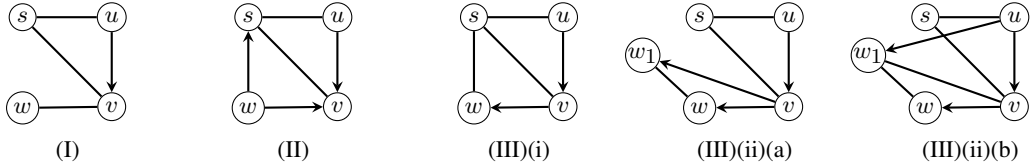
\begin{figure}[htb]
\begin{tikzpicture}
\tikzset{node distance=0.8cm, >=stealth}
\begin{scope} 
\node[rv] (s) {$s$};
\node[rv, right=of s] (u) {$u$};
\node[rv, below=of u] (v) {$v$};
\node[rv, below=of s] (w) {$w$};
\draw[-, thick] (s) -- (u);
\draw[-, thick] (s) -- (v);
\draw[->, thick] (u) -- (v);
\draw[-, thick] (w) -- (v);
\node[below=1mm of w, xshift=6mm] (l) {(I)};
\end{scope} 
\begin{scope}[xshift=3cm] 
\node[rv] (s) {$s$};
\node[rv, right=of s] (u) {$u$};
\node[rv, below=of u] (v) {$v$};
\node[rv, below=of s] (w) {$w$};
\draw[-, thick] (s) -- (u);
\draw[-, thick] (s) -- (v);
\draw[->, thick] (u) -- (v);
\draw[->, thick] (w) -- (v);
\draw[->, thick] (w) -- (s);
\node[below=1mm of w, xshift=6mm] (l) {(II)};
\end{scope}
\begin{scope}[xshift=6cm] 
\node[rv] (s) {$s$};
\node[rv, right=of s] (u) {$u$};
\node[rv, below=of u] (v) {$v$};
\node[rv, below=of s] (w) {$w$};
\draw[-, thick] (s) -- (u);
\draw[-, thick] (s) -- (v);
\draw[->, thick] (u) -- (v);
\draw[->, thick] (v) -- (w);
\draw[-, thick] (w) -- (s);
\node[below=1mm of w, xshift=6mm] (l) {(III)(i)};
\end{scope}
\begin{scope}[xshift=9cm] 
\node[rv] (s) {$s$};
\node[rv, right=of s] (u) {$u$};
\node[rv, below=of u] (v) {$v$};
\node[rv, below=of s] (w) {$w$};
\node[rv, inner sep=0pt, left=of w, xshift=5mm, yshift=7mm] (w1) {$w_1$};
\draw[-, thick] (s) -- (u);
\draw[-, thick] (s) -- (v);
\draw[->, thick] (u) -- (v);
\draw[->, thick] (v) -- (w);
\draw[-, thick] (w) -- (w1);
\draw[->, thick] (v) -- (w1);
\node[below=1mm of w, xshift=6mm] (l) {(III)(ii)(a)};
\end{scope}
\begin{scope}[xshift=12cm] 
\node[rv] (s) {$s$};
\node[rv, right=of s] (u) {$u$};
\node[rv, below=of u] (v) {$v$};
\node[rv, below=of s] (w) {$w$};
\node[rv, inner sep=0pt, left=of w, xshift=5mm, yshift=7mm] (w1) {$w_1$};
\draw[-, thick] (s) -- (u);
\draw[-, thick] (s) -- (v);
\draw[->, thick] (u) -- (v);
\draw[->, thick] (v) -- (w);
\draw[-, thick] (w) -- (w1);
\draw[-, thick] (v) -- (w1);
\draw[->, thick] (u) -- (w1);
\node[below=1mm of w, xshift=6mm] (l) {(III)(ii)(b)};
\end{scope}

\end{tikzpicture}
\caption{Graphs that appear in the proof of \cref{lem:simplicial}.}
\label{fig:proof-simplicial}
\end{figure}

\begin{proof}
Because $u$ is simplicial, it suffices to show that $w \in \{u\} \cup \Nb(u)$ holds for every $w \in \Nb(v)$. For each $w \in \Nb(v)$, there are three cases as illustrated in \cref{fig:proof-simplicial}. For each case, the argument is based on the four induced subgraphs listed in \cref{fig:dags}(b) that are forbidden in any MPDAG. 
\begin{enumerate}[(I)]
\item $w \in \Sib(v)$. The result holds trivially if $w=s$. For a contradiction, suppose $w \neq s$ and $w \notin \Nb(u)$. Now observe that $u \rightarrow v - w$, as an induced subgraph, is forbidden.

\item $w \in \Pa(v)$. Suppose $u,w$ are not adjacent. To avoid forbidden structure $w \rightarrow v - s$, we have an edge between $s$ and $w$. Further, the edge cannot be undirected because otherwise the graph induced by $\{u,v,w,s\}$ is a forbidden structure. Therefore, the edge must be directed, in particular, as $w \rightarrow s$ due to the undirected edge $s-v$. However, this induces another forbidden structure $w \rightarrow s - u$. 

\item $w \in \Ch(v)$. Because every pair of vertices is connected by an undirected path in $\H$, let $p$ be a shortest undirected path that connects $w$ and $\{u,s,v\}$. We prove $w \in \Nb(u)$, in particular, $u \rightarrow w$, by induction on the length of $p$. 

\begin{enumerate}[(i)]
\item Base case: $p$ is of length 1. Observe that in this case we must have an undirected edge between $s$ and $w$. Without an edge between $u$ and $w$, the graph induced by $\{s,u,v,w\}$ is a forbidden structure. Hence, we have an edge between $u$ and $v$, in particular, $u \rightarrow w$ by acyclicity. 

\item Suppose the result holds when $p$ is of length $l$ or shorter for $l \geq 1$. Now we show that the result still holds when $p$ is of length $l+1$. We write the path as $p = \langle w \equiv w_0, w_1, \dots, w_{l+1} \rangle$ for $w_{l+1} \in \{s,u,v\}$. To avoid the forbidden structure $v \rightarrow w - w_1$, we know $v$ and $w_1$ are adjacent. There are two further cases. 

\begin{enumerate}[(a)]
\item $v \rightarrow w_1$. Observe that $w_1$ plays the role of $w$ and $w_1$ is connected to $\{s,u,v\}$ by an undirected path of length $l$. By our induction hypothesis, we have $u \rightarrow w_1$. To avoid the forbidden structure $u \rightarrow w_1 - w$, there must be an edge between $u$ and $w$. In particular, we have $u \rightarrow w$.

\item $v - w_1$. By case (I), we know $u$ and $w_1$ are adjacent. To prove by contradiction, suppose there is no edge between $u$ and $w$. Observe that the edge between $u$ and $w_1$ cannot be undirected since otherwise $\{u,v,w,w_1\}$ induces a forbidden structure. Further, due to $w_1 - v$, we cannot have $w_1 \rightarrow u$. Hence, we must have $u \rightarrow w_1$. However, we see a contradiction because this leads to a forbidden structure $u \rightarrow w_1 - w$. Hence, there must be an edge between $u$ and $w$, in particular, $u \rightarrow w$. 

\end{enumerate}

\end{enumerate}

\end{enumerate}
\end{proof}

\begin{figure}[t!]
\includegraphics[width=.6\columnwidth]{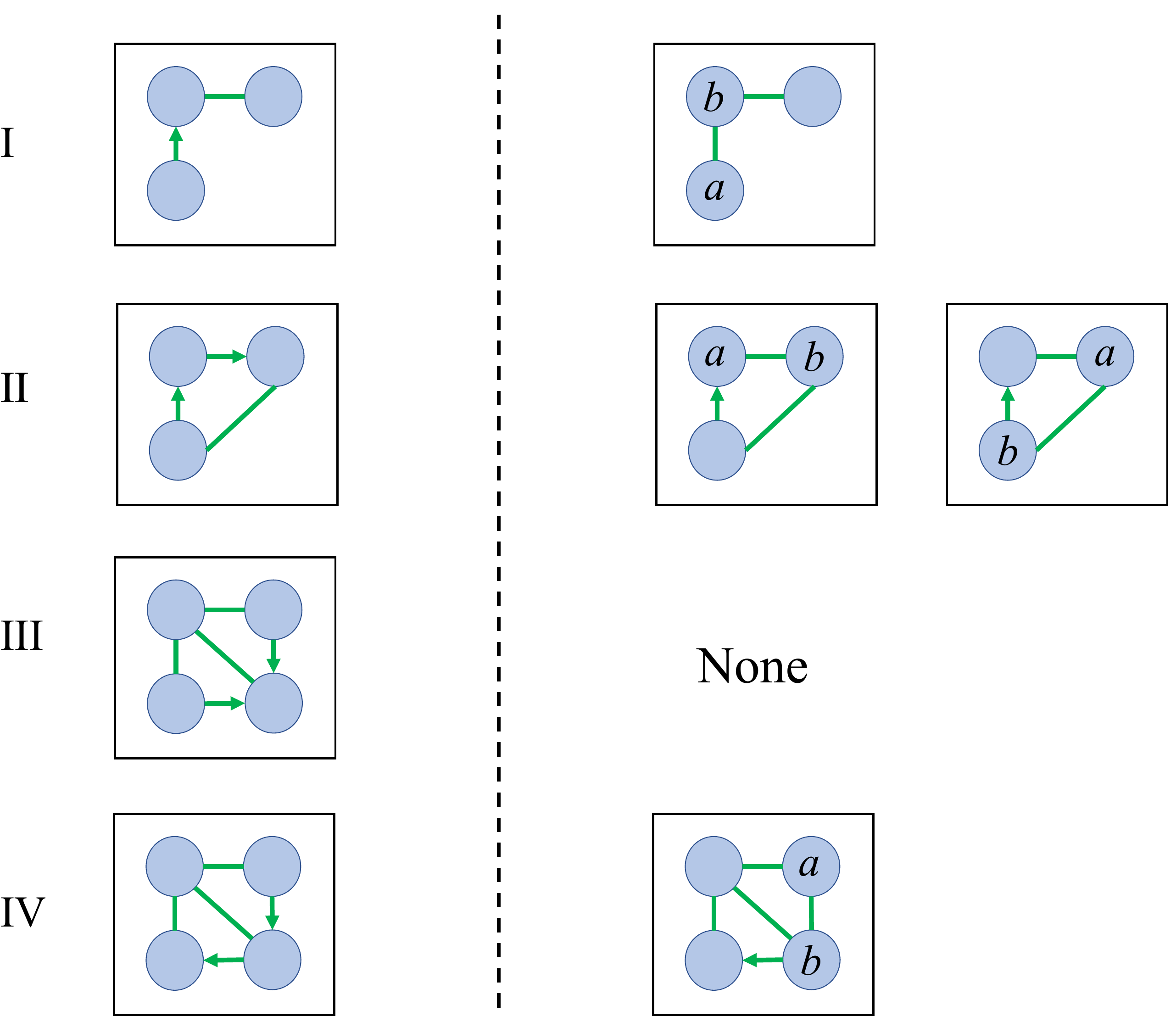}
\caption{Induced subgraphs that are forbidden in an MPDAG (left) and their corresponding pre-structures (right). Each pre-structure can exist as an induced subgraph in an MPDAG, but upon directing $a - b$ to $a \rightarrow b$, it becomes a forbidden structure.}
\label{fig:pre-forbid}
\end{figure}

\begin{lemma}[MPDAG single-edge orientation] \label{lem:one-by-one}
Let $\G$ be any MPDAG with one or more undirected edges. 
Then, there exists an undirected edge in $\G$ such that by directing it and leaving all the other edges unchanged, we obtain another MPDAG $\tilde{\G}$.
\end{lemma}
\begin{proof}[Proof of \cref{lem:one-by-one}]
The proof consists of two parts. First, we show that by choosing an undirected edge $a-b$ that meets certain criteria (assuming such $a-b$ exists in $\G$) and directing it to $a \rightarrow b$ and leaving everything else intact, we obtain an MPDAG $\tilde{\G}$. Second, we argue that an edge $a-b$ that meets our selection criteria always exists in $\G$. 

 We show the first part. 
We select an edge $a-b$ in $\G$ that meets the following criteria:
\begin{enumerate}[(i)]
\item $\Sib_{\G}(b) \subseteq \{a\} \cup \Sib_{\G}(a) \cup \Ch_{\G}(a)$, and 
\item there is no vertex $c$ such that $b \rightarrow c$ and $b - s - c$ exist in $\G$.
\end{enumerate}
Let $\tilde{\G}$ be a copy of $\G$ except that $a-b$ is directed to $a \rightarrow b$. 
We argue that $\tilde{\G}$ is an MPDAG by showing that it satisfies the graphical constraints listed in \cref{sec:valid-mpdag}.
The chordality constraint is satisfied because $\G$ is already an MPDAG and we have only directed an undirected edge. 
Thus, it remains to show that the four forbidden structures listed in \cref{fig:dags}(b) are not introduced. 
By inspecting these forbidden structures, we list in \cref{fig:pre-forbid} the only possible ways that a forbidden structure can be introduced upon directing a single edge: forbidden structure III cannot be introduced; for structures I, II and IV to be introduced, a corresponding pre-forbidden structure must exist in $\G$ and $a-b$ must be chosen as specified in the figure. 
To finish proving the first part, we argue that our selection criteria prevent $a-b$ from being chosen in such ways: 
\begin{itemize}
\item[Structure I.] This violates (i);
\item[Structure II.] The first case violates (i) and the second case violates (ii);
\item[Structure IV.] This violates (ii). 
\end{itemize}

We now prove the second part, namely the existence of $a-b$ in $\G$ that meets our selection criteria. 
By our assumption, $\G$ has one or more undirected edge. 
Accordingly, let $\sC$ be a (maximal) connected component spanned by undirected edges that has two or more vertices and let $\H := \G_{\sC}$.

We first choose $b$ in $\H$ that satisfies criterion (ii). 
By \cref{sec:valid-mpdag}, $\H$ is an MPDAG and the skeleton of $\H$ is chordal. 
By a classical result on chordal graphs \citep[Lemma 2.9]{lauritzen1996graphical}, there exists a simplicial vertex $b'$ in $\H$ with $\Sib_{\H}(b') \neq \emptyset$.
If $b'$ satisfies (ii), let $b := b'$. 
Otherwise, whenever $b' \rightarrow c$ and $b' - s - c$ holds in $\G$ (and hence in $\H$ by definition), replace $b'$ with $c$. 
We repeat this process until (ii) is satisfied and let $b$ be the final $b'$. 
Because $\H$ has no directed cycle, the process must terminate in a finite number of steps. 
Further, by \cref{lem:simplicial}, the process maintains the property that $b'$ is simplicial. Meanwhile, the fact that $\Sib_{\H}(b') \neq \emptyset$ is also maintained. 
Hence, the resulting $b$ is simplicial in $\H$ with $\Sib_{\H}(b) \neq \emptyset$.

Finally, we choose $a \in \Sib_{\H}(b)$ such that $a$ is topologically minimal in $\Sib_{\H}(b)$. 
It only remains to show that criterion (i) holds. 
Since $b$ is simplicial in $\H$, for each $u \in \Sib_{\H}(b)$, we have $u=a$ or $u \in \N_{\H}(a)$. 
Further, by the way that $a$ is chosen, we know $u \notin \Pa_{\H}(a)$. Hence, we have $u \in \{a\} \cup \Sib_{\H}(a) \cup \Ch_{\H}(a)$, i.e., 
\[ \Sib_{\H}(b) \subseteq \{a\} \cup \Sib_{\H}(a) \cup \Ch_{\H}(a). \]
By the fact that $\Sib_{\G}(b) = \Sib_{\H}(b)$, $\Sib_{\G}(a)=\Sib_{\H}(a)$ and $\Ch_{\H}(a) \subseteq \Ch_{\G}(a)$, we obtain 
\[ \Sib_{\G}(b) \subseteq \{a\} \cup \Sib_{\G}(a) \cup \Ch_{\G}(a), \]
i.e., criterion (i). 
\end{proof}

\begin{proof}[Proof of \cref{prop:mpdag_upper}]
Suppose $\G$ has $l$ undirected edges and $k$ directed edges, so we have $\pseudorank(\G) = 2l + k$. The maximal chain is given by 
\[ \G=:\G_{2l+k} \cover \G_{2(l-1)+k+1} \cover \dots \cover \G_{l+k} \cover \G_{l+k-1} \cover \dots \cover \G_{0} := \Glst, \]
which consists of two segments: (i) starting from $\G$, we repeatedly apply \cref{lem:one-by-one} to convert one undirected edge to a directed edge until we reach $\G_{l+k}$ that only contains directed edges; (ii) then, we remove directed edges iteratively in any order till we reach the empty graph. 
\end{proof}

\subsubsection{Proof of Proposition~\ref{prop:shd_pseudo}} \label{appendix:shd_pseudo}

\begin{proof}
The proof is based on the following observation: for any pair of neighboring graphs $\G_1, \G_2 \in \mathfrak{G}$ with $\pseudorank(\G_2) - \pseudorank(\G_1) = 1$, the graphs $\G_1$ and $\G_2$ are identical except for a single edge: either a missing edge versus a directed edge, or a directed edge versus an undirected edge.

Consider a pair of vertices such that the edge between them differs in $\G_s$ and $\G_t$. 
Then, based on the observation above, for any path $p \in \sP_{\Lp,\pseudo}(\G_s, \G_t)$, then $p$ must go through some neighboring graphs $\G_i, \G_{i+1}$ such that they differ exactly at this edge.  
If the edge is missing in one graph and undirected in the other, or if the edge is directed in opposite orientations in the two graphs, then at least two such moves exist in the path. This implies that $d_{\Lp,\pseudo}(\G_s, \G_t)$ is lower bounded by $\SHD_2(\G_s,\G_t)$. 
\end{proof}

\subsubsection{Polytree MPDAGs} \label{sec:apdx-polytree-mpdag} 
We first prove an auxiliary lemma. 

\begin{lemma} \label{lem:polytree-mpdag-rm}
Let $\G$ be a polytree MPDAG. By removing any edge from $\G$, we obtain an MPDAG $\G'$ that satisfies $\G' \prec \G$.
\end{lemma}
\begin{proof}
Using the characterization of valid MPDAGs in \cref{sec:valid-mpdag}, $\G'$ is still an MPDAG: it is a PDAG, and further (i) the chordality condition is satisfied, and (ii) no forbidden structure is introduced. 
It then follows from \cref{prop:order-mpdags} that $\G' \prec \G$. 
\end{proof}

\begin{proof}[Proof of \cref{thm:poset-polytree-mpdags}] \hfill
\begin{enumerate}
\item This directly follows from definition.

 \item Let $\G_1,\G_2$ be polytree MPDAGs with $\G_1 \coverby \G_2$. 
By \cref{prop:pseudo-covering}, we have $\pseudorank(\G_1) +1 \leq \pseudorank(\G_2)$. For contradiction, suppose $\pseudorank(\G_1) +  1 < \pseudorank(\G_2)$, or equivalently, $\pseudorank(\G_1) +  2 \leq \pseudorank(\G_2)$. 
There are three cases to consider and we show that there is a contradiction in each case. 

\begin{itemize}
\smallskip \item[Case I] Two or more pairs of vertices are connected in $\G_2$ but not in $\G_1$. 
Let $(u,v)$ be such a pair. 
Using \cref{lem:polytree-mpdag-rm}, removing the edge between $u,v$ from $\G_2$ leads to an MPDAG $\tilde{\G}$ with $\tilde{\G} \prec \G_2$. 
It also follows from \cref{lem:polytree-mpdag-rm} that $\G_1 \prec \tilde{\G}$. This contradicts $\G_1 \coverby \G_2$. 

\smallskip \item[Case II] There is at least one pair of vertices that is connected via an undirected edge in $\G_2$ but is connected via a directed edge in $\G_1$, and there is at least one pair of vertices, say $(u,v)$, that are connected in $\G_2$ but not in $\G_1$. By an argument similar to Case I, removing the edge between $u,v$ from $\G_2$ yields an MPDAG $\tilde{\G}$ satisfying $\G_1 \prec \tilde{\G} \prec \G_2$ and hence a contradiction. 

\smallskip \item[Case III]  There are at least two pairs of vertices, say $(i,j)$ and $(k,l)$, that are connected by an undirected edge in $\G_2$ but by a directed edge in $\G_1$. 
Suppose $i \to j$ and $k \to l$ appear in $\G_1$. 
Let $\mathcal{C}$ be the CPDAG that corresponds to $\G_2$. 
Because $\G_2$ is an MPDAG, it must result from imposing a set of edge orientations $\mathcal{E}$ (can be empty) on $\mathcal{C}$ and completing Meek's rules. 
Let $\tilde{\G}$ be the MPDAG that represents $\mathcal{C}$ with orientations $\mathcal{E} \cup \{i \rightarrow j\}$. 
Following the argument from Case II, we have that $\G_1 \preceq \tilde{\G} \preceq \G_2$. 
Due to $i \rightarrow j$ in $\tilde{\G}$ but $i - j $ in $\G_2$, we know $\tilde{\G} \neq \G_2$.
To prove the result, we will show that $(i,j)$ and $(k,l)$ can be chosen in a way that ensures $\tilde{\G} \neq \G_1$. 

More concretely, we show that $(i,j)$ and $(k,l)$ can be chosen such that $k-l$ instead of $k \rightarrow l$ appears in $\tilde{\G}$. 
To see this, note that in a polytree, the only rule in \citet{meek1995causal} that can be applied when directing an edge is the following: if there is a structure of the form $u \rightarrow v - w$, direct the edge $v-w$ as $v\rightarrow w$. 
By construction of $\tilde{\G}$, $k - l$ becomes $k \rightarrow l$ in $\tilde{\G}$ only if it is compelled by this rule. 
Therefore, if $k \rightarrow l$ appears in $\tilde{\G}$, then there must exist an undirected path $i - j - \dots - k - l$ ($j$ may be identical to $k$) in $\G_2$ such that one or more applications of the rule propagates $i \rightarrow j$ to $k \rightarrow l$ in $\tilde{\G}$.
However, if we were to swap the choices of $(i,j)$ and $(k,l)$, then we would have $i \rightarrow j$ and $k -l $ in $\tilde{\G}$, which finishes the proof. 
\end{itemize}

\item Let $\G_1,\G_2,\G_3 $ be polytree MPDAGs with $\G_2,\G_3 \coverby \G_1$ and $\G_2 \neq \G_3$. 
By the previous result, we know 
\[\pseudorank(\G_2)=\pseudorank(\G_3)=\pseudorank(\G_1)-1.\]
Therefore, $\G_2$ and $\G_3$ differ from $\G_1$ by having either a single directed edge from $\G_1$ missing or having a single directed edge that is undirected in $\G_1$. Up to relabeling we therefore have three cases so let i) be the case where $\G_2$ and $\G_3$ both have an additional directed edge that is undirected in $\G_1$, ii) only $\G_2$ has such an edge and iii) neither do.

For the first case, note that when directing an undirected edge in an MPDAG that is a polytree, the resulting graph is guaranteed to be a polytree and therefore chordal. It is also an MPDAG unless the directed edge is now part of a forbidden structure of the form $ \to v -$ since the other three forbidden structures cannot exist in a polytree. Since $\G_2$ and $\G_3$ are MPDAGs by assumption, it follows that the edge directed in $\G_2$ (or $\G_3$) when compared to $\G_1$ does not point into another undirected edge in $\G_1$ or $\G_2$, respectively $\G_3$. Therefore, we can direct the edge directed in $\G_3$ also in $\G_2$ to obtain $\G_4$. By construction, $\G_4 \coverby \G_2,\G_3$. For the second and third case, note that we can remove any directed edge from an MPDAG that is a polytree by Lemma \ref{lem:polytree-mpdag-rm}. In both Case ii) and iii), we can therefore construct $\G_4$ by deleting the directed edge missing in $\G_3$ from $\G_2$. By construction, $\G_4 \coverby \G_2,\G_3$. 

 \item From Result 2, we know $\sP_{\Lp_\poly,\pseudo} = \sP_{\Lp_\poly}$ and hence $d_{\Lp_{\poly},\pseudo}(\G_s,\G_t) = d_{\Lp_{\poly}}(\G_s,\G_t)$. 
Further, by Result 3 and \cref{thm:semimod}, we also have $d_{\Lp_{\poly}}(\G_s,\G_t)=d_{\Lp_{\poly},\du}(\G_s,\G_t)$.
\end{enumerate}
\end{proof}

\section{Idiosyncrasies of the structural interventional distance and s/c-metric}
\subsection{Structural intervention distance}
\label{app:SID}

The structural intervention distance (SID) measures the distance between a candidate graph $\G_{\mathrm{guess}}$ and a ground-truth graph $\G_{\mathrm{true}}$ by counting the number of interventional distributions of the form $f(y \mid do(X = x))$ that would be incorrectly inferred if parent adjustment were applied according to $\G_{\mathrm{guess}}$, while the true causal model conforms to $\G_{\mathrm{true}}$. SID was originally proposed as an alternative to SHD, aiming to account for the fact that the graphs being compared are \emph{causal} rather than purely structural. In this sense, it represents an attempt to define a model-oriented distance specifically for causal DAGs.

However, as noted by \citet{peters2015structural}, SID is not a metric. The SID between a pair of DAGs is zero if and only if $\G_{\mathrm{guess}}$ is a subgraph of $\G_{\mathrm{true}}$. Consequently, SID does not satisfy the standard metric properties. It fails to ensure that $d_{\mathrm{SID}}(\G_1, \G_2) = 0$ if and only if $\G_1 = \G_2$, is not symmetric and does not obey the triangle inequality. It is, however, possible to symmetrize SID by averaging the two asymmetric distances for each pair of graphs. The resulting symmetrized SID satisfies the first two axioms of a metric, but as we will show with an example, it still fails to satisfy the triangle inequality.

\begin{figure}
\subfloat[$\G_1$]{
{\includegraphics[width=.12\textwidth]{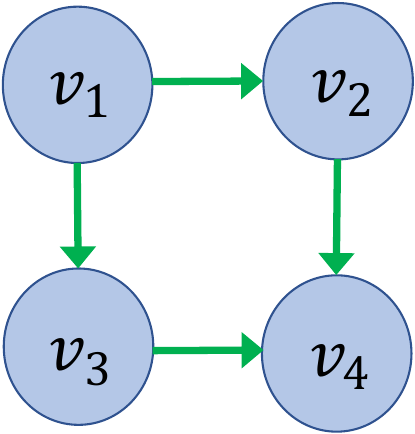}}
}
\hspace{0.3in}
\subfloat[$\G_2$]{
{\includegraphics[width=.12\textwidth]{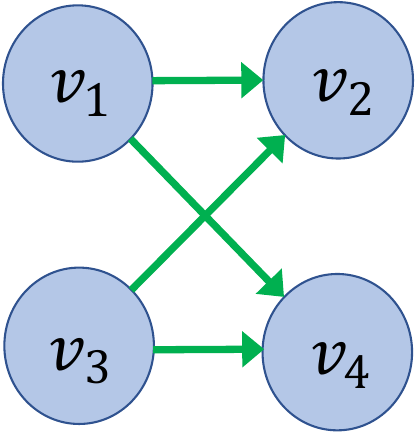}}
}
\hspace{0.3in}
\subfloat[$\G_3$]{
{\includegraphics[width=.12\textwidth]{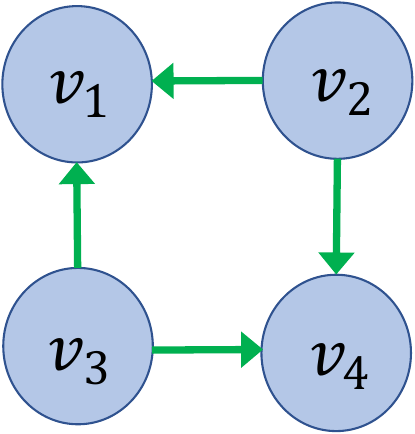}}
}
	\caption{Three DAGs for which symmetrized SID violates the triangle inequality  as shown in Example \ref{ex:sid counterexample}.}
	\label{fig:sid_counterexample}
\end{figure}

\begin{example}[Symmetrized SID violates triangle inequality]
\label{ex:sid counterexample}
Let $\G_1$, $\G_2$, and $\G_3$ be the three DAGs from \cref{fig:sid_counterexample}. 
Let $d_{\mathrm{SID}}$ denote the structural intervention distance for DAGs  \citep{peters2015structural}. Using the open source implementation for computing SID  \citep{henckel2024gadjid}, we obtain that $d_{\mathrm{SID}}(\G_1,\G_2)=3, d_{\mathrm{SID}}(\G_2,\G_1)=4,d_{\mathrm{SID}}(\G_2,\G_3)=5,d_{\mathrm{SID}}(\G_3,\G_2)=5,d_{\mathrm{SID}}(\G_1,\G_3)=9$ and $d_{\mathrm{SID}}(\G_3,\G_1)=9$. For the symmetrized version of SID, denoted $d_{\mathrm{sySID}}$, it therefore holds that 
\begin{align*}
    d_{\mathrm{sySID}}(G_1,G_2)+d_{\mathrm{sySID}}(G_2,G_3)= \frac{3+4}{2} + \frac{5+5}{2}=8.5 < 9 = \frac{9+9}{2}=d_{\mathrm{sySID}}(G_1,G_3),
\end{align*}
which shows that it does not satisfy triangle inequality and therefore fails to be a metric.
    
\end{example}

\subsection{s/c-metric}\label{app:sep-distance}

We interpret a probabilistic graphical model as the collection of all random variables satisfying the conditional independence constraints encoded by the graph. A closely related and popular alternative is to interpret a probabilistic graphical model directly as the collection of conditional independence statements encoded by the graph or more formally the set of all triples $(X,Y,Z)$ corresponding to a conditional independence statement of the form $X \indep Y \mid Z $ \citep{studeny1992conditional}. Based on this interpretation of a CPDAG as a collection of triples \citep{wahl2025metric} have proposed an alternative distance notion called the s/c-metric. It defines the distance between two CPDAGs as the weighted size of the symmetric difference of the two corresponding sets of triples. Concretely, each discordant CI statement of the form $X \indep Y \mid Z $ is weighted by the number of potential CI statements with conditioning sets of size equal to $|Z|$ times $(n-1)$, that is, if $|Z|=k$ then the weight is the inverse of $$\frac{n(n-1)^2}{2}\binom{n-2}{k}.$$
This ensures that the overall distance is a number between $0$ and $1$.
Similarly to our class of distances, the s/c-metric is a metric and model-oriented, albeit on a different model interpretation. It can also be generalized to any class of probabilistic graphical models that encode only conditional independence statements such as undirected graphs, albeit being expensive to compute.

However, the fact that the s/c-metric is implicitly defined on the set of all triples of subsets of the node set, while ignoring which of these correspond to a CPDAG or even valid conditional independence models for that matter leads to some idiosyncrasies. In particular, the s/c-metric between sparse graphs tends to be larger than between dense graphs, simply because the former encode more separation statements to begin with and as a result dropping an additional statement typically triggers a cascade of implications.

\begin{example}
Let $\G_1$ denote the empty CPDAG, $\G_2$ a CPDAG with a single (undirected) edge between nodes $i$ and $j$, $\G_3$ the fully connected CPDAG, and $\G_4$ the same graph with one (undirected) edge removed, all on $n$ vertices.  Since all d-separation statements hold in $\G_1$, it can be represented by the set of all triples of subsets on the set of $n$ vertices. In $\G_2$, on the other hand, all statements hold except those of the form 
\[\mathrm{V}_i \indep_{d} \mathrm{V}_j \mid Z, \quad Z \subseteq V \setminus \{\mathrm{V}_i, \mathrm{V}_j\}.\]
As a result, the s/c-metric between $\G_1$ and $\G_2$ is 
$
1/(n(n-1)).
$ 
Furthermore, for every disconnected node we add to both $\G_1$ and $\G_2$, their distance increases further. 
The CPDAG $\G_3$, on the other hand, encodes no d-separation statements and therefore corresponds to the empty set, while $\G_4$ encodes the single statement  
$\mathrm{V}_i \indep \mathrm{V}_j \mid \mathrm{V} \setminus \{\mathrm{V}_i, \mathrm{V}_j\}$. Hence, the s/c-metric between $\G_3$ and $\G_4$ equals $1/(n(n-1)^2)$, that is, it is $(n-1)$ times smaller than that between $\G_1$ and $\G_2$. Adding a disconnected node to both graphs increases this distance yet further.

In contrast, $d_{\Lp}(\G_1,\G_2) = d_{\Lp}(\G_3,\G_4) = 1$, as $\G_2$ is a covering element of $\G_1$ and $\G_4$ of $\G_3$ in $\Lp$. This reflects the fact that none of the triples intermediate between the triple representing $\G_1$ and that representing $\G_2$ correspond to a CPDAG, or even to a valid conditional independence model. Furthermore, these distances remain unchanged if we add a disconnected node to all four graphs.
\end{example}

\section{Algorithms and additional numerical results} \label{sec:adpx-alg}

\subsection{Proof of \cref{thm:alg}} \label{sec:proof-alg}
\begin{proof}
Because $\Lp$ is connected and $\g$ is finite, $d_{\Lp}$ is well-defined for every pair of graphs in $\g$. Suppose the algorithm does not terminate. Then, it must be due to graphs continuing to be pushed to \textsf{openSet} in line 22. However, this means line 19 of the algorithm is executed infinitely many times, where after each time $g[\G']$ is either created or strictly decreased. We have a contradiction because this implies the condition in line 18 must fail at some point. 

In what follows, we show that the algorithm returns the correct distance. 
Observe that the while loop cannot exit without returning either at line 8 or 12, because if so $\G_t$ is never visited and that contradicts the connectedness of $\Lp$. 
We first show that if the algorithm returns at line 8, $u^\ast$ upon return equals the desired distance.
To see this, we argue that there must exist a shortest path that is routed through a graph $\G'$ in $\textsf{openSet}$ for which it holds that $f(\G') \leq d_{\Lp}(\G_s,\G_t)$. 
Specifically, consider any shortest path $p^\ast$ from $\G_s$ to $\G_t$. 
Because $\G_t$ is not yet reached, we have $p^\ast \cap \textsf{openSet} \neq \emptyset$ and we choose $\G'$ to be the graph in this intersection that is pushed to \textsf{openSet} most recently. Indeed, for this graph on the frontier, we have $g[\G'] = d_{\Lp}(\G_s, \G')$ and hence $f(\G') = g[\G'] + h(\G') \leq d_{\Lp}(\G_s, \G') + d_{\Lp}(\G', \G_t) = d_{\Lp}(\G_s, \G_t)$  by the admissibility of $h$. 
Then, by the property of the priority queue, it follows that $\textsf{f} \leq f(\G') \leq  d_{\Lp}(\G_s,\G_t) \leq u^{\ast}$. 
Meanwhile, we also have $\textsf{f} \geq u^{\ast}$ by line 7 so we conclude $\textsf{f} = u^{\ast} = d_{\Lp}(\G_s,\G_t)$. 

Now we argue that \textsf{g} returned at line 12 equals $d_{\Lp}(\G_s,\G_t)$. Suppose this does not hold. 
Observe that $\textsf{g} = g[\G_t]$ when $\G_t$ was pushed into \textsf{openSet} and by construction $g[\G_t] \geq d_{\Lp}(\G_s,\G_t)$ holds, so we must have $\textsf{g} > d_{\Lp}(\G_s,\G_t)$. 
Let $p^{\ast}$ be any shortest path from $\G_s$ to $\G_t$. 
Similar to the argument in the preceding paragraph, let $\G^{\ast}$ be the graph on $p^\ast$ that was last pushed to \textsf{openSet} right before $\G_t$ was popped at line 6.
By construction of the algorithm and the fact that $p^{\ast}$ is a shortest path, we have $g[\G^{\ast}] = d_{\Lp}(\G_s,\G^{\ast})$. Further, using admissibility of $h$, we know 
\[ f(\G^{\ast}) = g[\G^{\ast}] + h(\G^{\ast}) \leq d_{\Lp}(\G_s,\G^{\ast}) + d_{\Lp}(\G^{\ast},\G_t) = d_{\Lp}(\G_s,\G_t). \]
However, because the priority queue pops the item with the smallest index, we also have $f(\G^{\ast}) \geq f(\G_t) = \textsf{g} + 0 > d_{\Lp}(\G_s, \G_t)$ and hence a contradiction. 
\end{proof}

\subsection{Specialization of Algorithm~\ref{alg:general_astar} to CPDAGs} 
\label{sec:additional_algs}

We specialize \cref{alg:general_astar} to $\g$, the class of CPDAGs over a vertex set $\sV$, by specifying the subroutines required by \cref{alg:general_astar}. 
By \cref{prop:cpdags}, an admissible heuristic can be constructed from 
\[ d_{\Lp}(\G,\G_t) \geq |\sk(\G) \triangle \sk(\G_t)| + 2m_{\mathrm{op}}(\sS(\G,\G_t)) \geq |\sk(\G) \triangle \sk(\G_t)|, \quad \G \in \g,\] 
where $|\sk(\G) \triangle \sk(\G_t)|$ and $m_{\mathrm{op}}$ can be computed using \cref{alg:lower_bound,alg:min_ops} listed below. By the same proposition, $u(\G)$ can be constructed from the upper bound $d_{\Lp}(\G,\G_t) \leq \min\left\{d_{\Lp,\du}(\G,\G_t), d_{\Lp,\ud}(\G,\G_t) \right\}$ for every $\G \in \g$. Next, we describe how to compute $d_{\Lp,\du}(\G,\G_t)$ and the same logic also applies to $d_{\Lp,\ud}(\G,\G_t)$. 

Because the poset for CPDAGs is graded, by \cref{prop:down_up_up_down_graded}, we have 
\[ d_{\Lp,\du}(\G,\G_t) = \rank(\G)+\rank(\G_t)-2 \rank(\G^\ast), \quad \G \in \g,\]
where $\rank(\cdot)$ is given by the total number of edges (i.e., $|\sk(\cdot)|$). 
Here, $\G^{\ast}$ is the inflection point, i.e., a graph with maximal rank among those that are below both $\G$ and $\G_t$. 
In fact, we can use \cref{alg:general_astar} to search for $d_{\Lp,\du}(\G,\G_t)$, only with the following adaptations: 
(i) $\textsc{EnumNeighbors}(\G')$ is restricted to the neighbors below, i.e., those graphs covered by $\G'$; 
(ii) the upper bound is set to 
\[ u(\G') := \begin{cases} \rank(\G_t) + \rank(\G'), & \quad \G' \not \preceq \G_t \\
\rank(\G_t) - \rank(\G'), & \quad \G' \preceq \G_t \end{cases},\]
where the first case corresponds to a path that first goes downward from $\G$ to $\Glst$ and then upward to $\G_t$; 
(iii) the heuristic is set to 
\[ h(\G'):= \rank(\G_t)+\rank(\G')-2|\mathrm{sk}(\G_t) \cap \mathrm{sk}(\G')|,\]
which is admissible as we show in \cref{sec:algorithm_theory}. 
Starting from $\G_s$, this algorithm searches downward until it finds the inflection point $\G^\ast \preceq \G_t$, upon which we have $u(\G^\ast) = h(\G^\ast) = \rank(\G_t) - \rank(\G^\ast)$ and the algorithm terminates immediately by line 7 of \cref{alg:general_astar}. 
For checking $\G' \preceq \G_t$, we use the following result due to \citet{Kocka2001OnCI}.

\begin{lemma}[Characterization of CPDAG containment]
For CPDAGs $\G,\G'$ with model map $\Mcpdag$, we have $\G' \preceq \G$ if and only if for any (and hence every) pair of DAGs $\D' \in [\G']$ and $\D \in [\G]$, it holds for every pair of vertices $(u,v)$ that
\[ \text{no edge between $u,v$ in $\D$} \implies \text{$u,v$ are d-separated by $\Pa_{\D}(u) \cup \Pa_{\D}(v)$ in $\D'$}.  \]
\end{lemma}

Finally, from \cref{prop:order-cpdags} it follows that for CPDAGs $\G, \G'$, we have $\G \coverby \G'$ if and only if there exist DAGs $\D \in [\G]$ and $\D' \in [\G']$ such that $\D$ can be transformed to $\D'$ through zero or more steps of covered edge reversal followed by a single edge addition. 
This can be turned into a subroutine that enumerates the neighbors of a CPDAG. 
In fact, the GES algorithm due to \citet{chickering2002optimal} has already implemented this subroutine: the first phase of GES uses a set of valid \textsc{Insert} operations to enumerate the neighbors above (i.e., those that cover the current graph), and the second phase use a set of valid \textsc{Delete} operations to enumerate the neighbors below (i.e., those that are covered by the current graph). The reader is referred to \citet[\S5]{chickering2002optimal} for a description of these operations.


\begin{algorithm}[tbp]
\DontPrintSemicolon
\SetKwFunction{Push}{Push}
\SetKwFunction{FindMinOps}{FindMinOps}
\SetKwFunction{DisjointTriplets}{DisjointTriplets}
\SetKwInOut{Require}{Require}
\SetKw{And}{and}
\SetKw{Or}{or}
\SetKw{In}{in}
\SetKw{NotIn}{not in}
\SetKw{Continue}{Continue}
\LinesNumbered

\KwIn{CPDAGs $\G_s$ and $\G_t$ over a vertex set $\sV$}
\Require{Subroutines $\FindMinOps(\cdot)$ and $\DisjointTriplets(\cdot)$}
\KwOut{The lower bound on $d_{\Lp}(\G_s,\G_t)$ given by \cref{prop:cpdags}}
$\G_{s,\sk} \gets \sk(\G_s)$\\
$\G_{t,\sk} \gets \sk(\G_t)$\\
$d \gets |\G_{s,\sk} \triangle \G_{t,\sk}|$ \tcp*{skeleton difference}
$\sS^{\ast} \gets \{\}$ \tcp*{$\sS(\G_s,\G_t)$}
\ForEach{ $a \in \sV$}{
    \ForEach{ $b \in \sV$}{
        \ForEach{ $c \in \sV$}{
        	\If{$a = b$ \Or $a=c$ \Or $b=c$}{
				\Continue
			}
			\If{$(a,b) \in \G_{s,\sk} \cap \G_{t,\sk}$ \And $(b,c) \in \G_{s,\sk} \cap \G_{t,\sk}$ \And $(a,c) \notin \G_{s,\sk} \cup \G_{t,\sk}$}{
				\uIf{$(a \rightarrow b \leftarrow c$ \In $\G_s)$ \And $(a \rightarrow b \leftarrow c$ \NotIn $\G_t)$}{
					$\Push(\sS^{\ast}, (a,b,c))$
				} 
				\uElseIf{$(a \rightarrow b \leftarrow c$ \In $\G_t)$ \And $(a \rightarrow b \leftarrow c$ \NotIn $\G_s)$}{
				$\Push(\sS^{\ast}, (a,b,c))$
				}
			}
		}
	}
		
}
$h \gets 0$\\
\ForEach{$\sS \in \DisjointTriplets(\sS^\ast)$}{
        $h \gets h + 2\cdot \FindMinOps(\sS)$
}
\Return{$h + d$}
\caption{Computing the lower bound in \cref{prop:cpdags} on $d_{\Lp}(\G_s,\G_t)$.}
\label{alg:lower_bound}
\end{algorithm}

\begin{algorithm}[tbp]
\DontPrintSemicolon
\SetKwFunction{Push}{Push}
\SetKwFunction{Pop}{Pop}
\SetKwFunction{Dict}{Dict}
\SetKwData{Q}{Q}
\SetKwData{pairs}{pairs}
\SetKw{And}{and}
\SetKw{Or}{or}
\SetKw{Continue}{Continue}
\LinesNumbered

\KwIn{Set $\sS$ consisting of a collection of vertex triplets}
\KwOut{$m_{\text{op}}(\sS)$}
$\Q \gets \Dict()$ \tcp*{tracks state and step count}
$\Q[S] \gets 0$ \\

\While{$\Q$ is not empty}{
    $(S', k) \gets \Pop(Q)$ \\
    \If{$S' = \emptyset$}{
        \Return{$k$}
    }
    $\pairs \gets \{\}$ \\
    \ForEach{$(a,b,c) \in S'$}{
        $\Push(\pairs, (a,b))$; $\Push(\pairs, (a,c))$; $\Push(\pairs, (b,c))$ \\
    }
    \ForEach{$(i,j) \in \pairs$}{
        \ForEach{$\mathrm{op} \in 
        \left\{\mathrm{OP}_{(i,j),\mathrm{CE}}, \mathrm{OP}_{(i,j),\mathrm{ER}} \right\}$}{
            $S'' \gets \mathrm{op}(S')$ \\
            $\Q[S''] \gets k+1$
        }
    }
}
\Return{$-1$} \tcp*{No Solution}
\caption{Subroutine \textsc{FindMinOps} for computing $m_{\text{op}}(\cdot)$ in \cref{eqs:mop}.}
\label{alg:min_ops}
\end{algorithm}

\subsection{Specialization of \cref{alg:general_astar} to MPDAGs and numerical results}
\label{app:MPDAG_numerical}
We consider using \cref{alg:general_astar} to compute $d_{\Lp_\poly}$ for a pair of polytree MPDAGs, where $\Lp_\poly$ is the subposet induced by polytree MPDAGs over a fixed vertex set. 
This is feasible because we show in \cref{thm:poset-polytree-mpdags} that $d_{\Lp_\poly} = d_{\Lp_{\poly},\du}$, so we can again adapt \cref{alg:general_astar} to compute the down-up distance.   
By the same theorem, the neighborhood of a graph $\G$ in $\Lp_\poly$ precisely consists of those graphs that differ from $\G$ by unit pseudo-rank, which leads to a straightforward \textsc{EnumNeighbors} subroutine. 
Finally, $h(\G):=\SHD_2(\G,\G_t)$ is an admissible heuristic. 

We compute the model-oriented distance for every pair of polytree MPDAGs over 5 vertices (6679 graphs in total), which takes 0.004 second per pair on average. \cref{fig:mpdag-numerical} compares the $d_{\Lp,\pseudo}$ and the structural Hamming distances for polytree MPDAGs over 5 vertices: about 98\% of the pairs have $d_{\Lp_{\poly}} = \SHD_2$. \cref{fig:cpdag_results_4nodes}(d) shows an example where $d_{\Lp_{\poly}} = \SHD_2 + 6$.

\begin{figure}[htbp]
\centering
\includegraphics[width=0.7\textwidth]{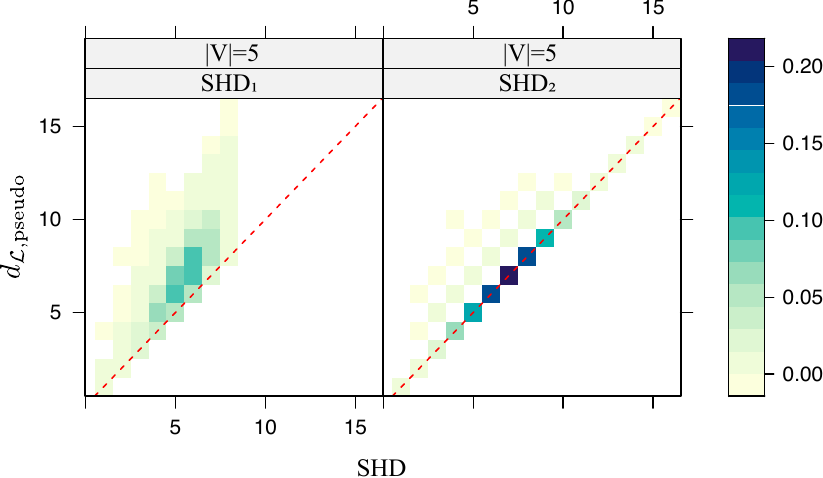}
\caption{The joint distribution of $d_{\Lp,\pseudo}$ \eqref{eqs:d-pseudo} and the structural Hamming distance computed for all pairs of polytree MPDAGs over 5 vertices, where the identity line is drawn as dashed.}
\label{fig:mpdag-numerical}
\end{figure}

\begin{figure}[htbp]
\centering
\begin{subfigure}[b]{0.28\columnwidth}
  \includegraphics[width=\linewidth]{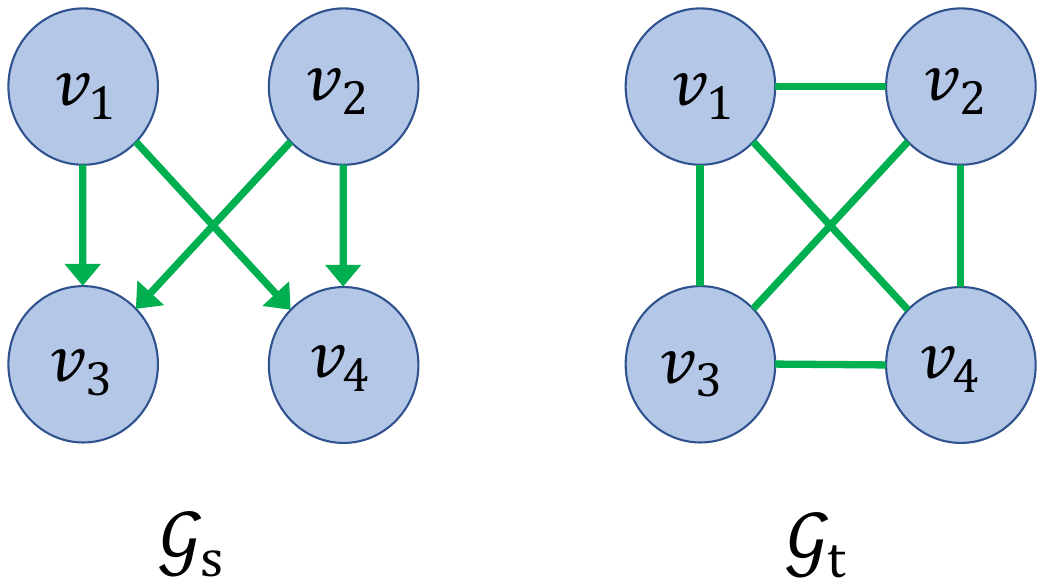}
  \caption{$d_{\Lp}=2$, $\SHD_1=6$, $\SHD_2=8$}
\end{subfigure}\quad\quad\quad
\begin{subfigure}[b]{0.28\columnwidth}
  \includegraphics[width=\linewidth]{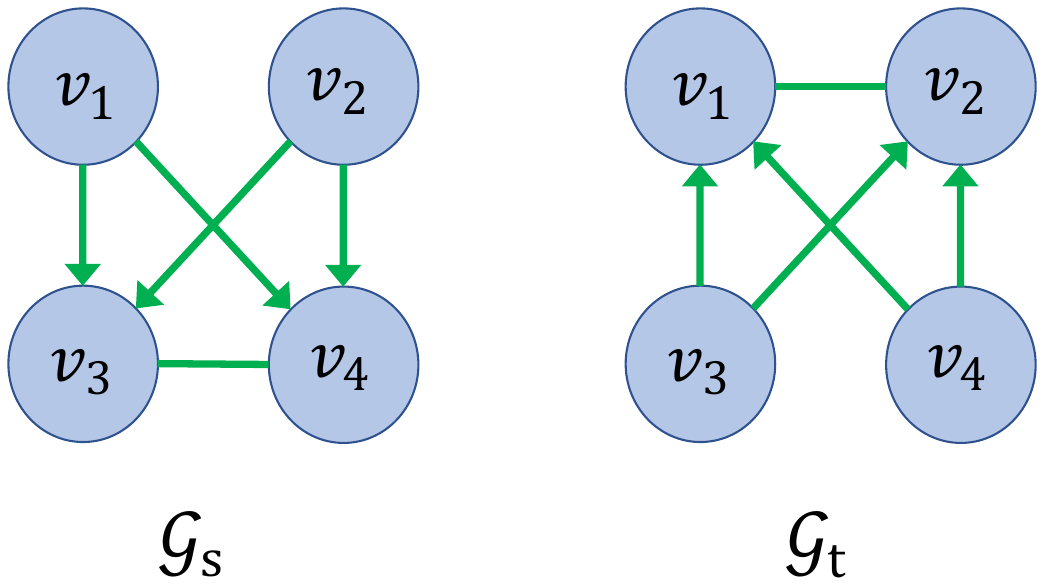}
  \caption{$d_{\Lp}=2$, $\SHD_1=6$, $\SHD_2=12$}
\end{subfigure}\quad\quad\quad
\begin{subfigure}[b]{0.28\columnwidth}
  \includegraphics[width=\linewidth]{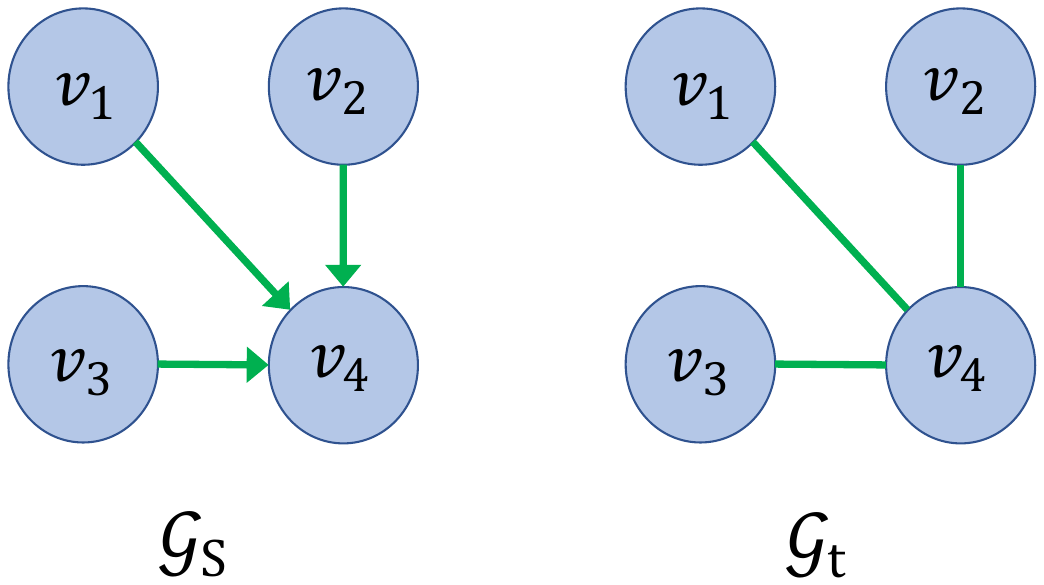}
  \caption{$d_{\Lp}=4$, $\SHD_1=3$, $\SHD_2=3$}
\end{subfigure}

\begin{subfigure}[b]{0.45\columnwidth}
  \includegraphics[width=\linewidth]{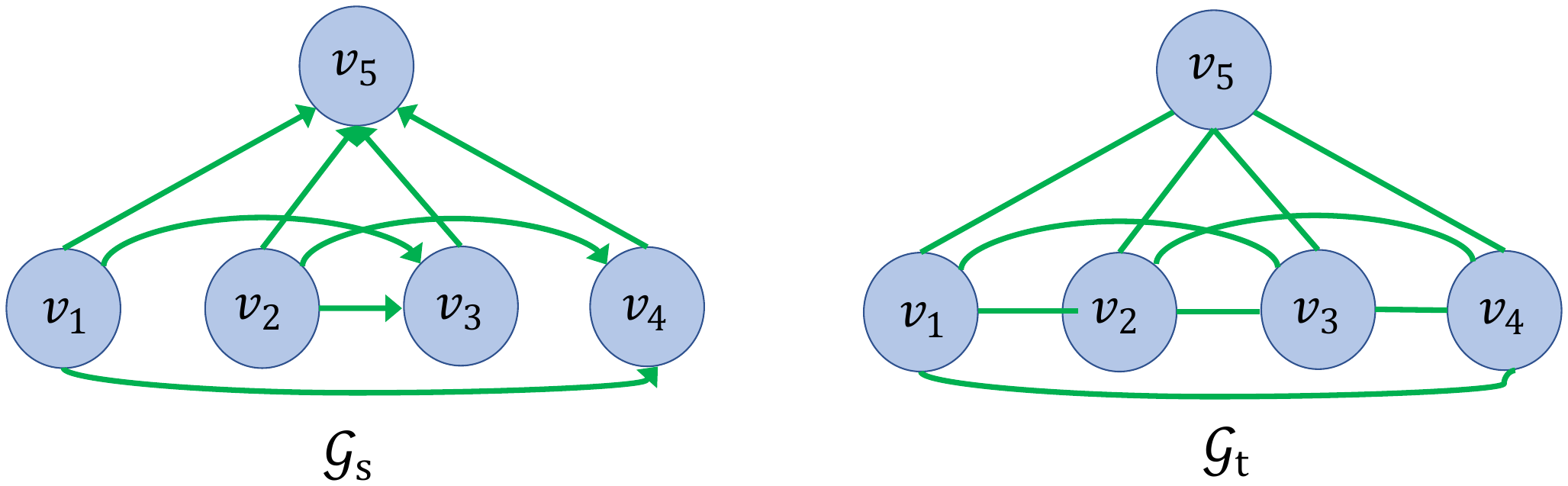}
  \caption{$d_{\Lp}=2$, $\SHD_1=10$, $\SHD_2=12$}
\end{subfigure}\quad\quad\quad
\begin{subfigure}[b]{0.45\columnwidth}
  \includegraphics[width=\linewidth]{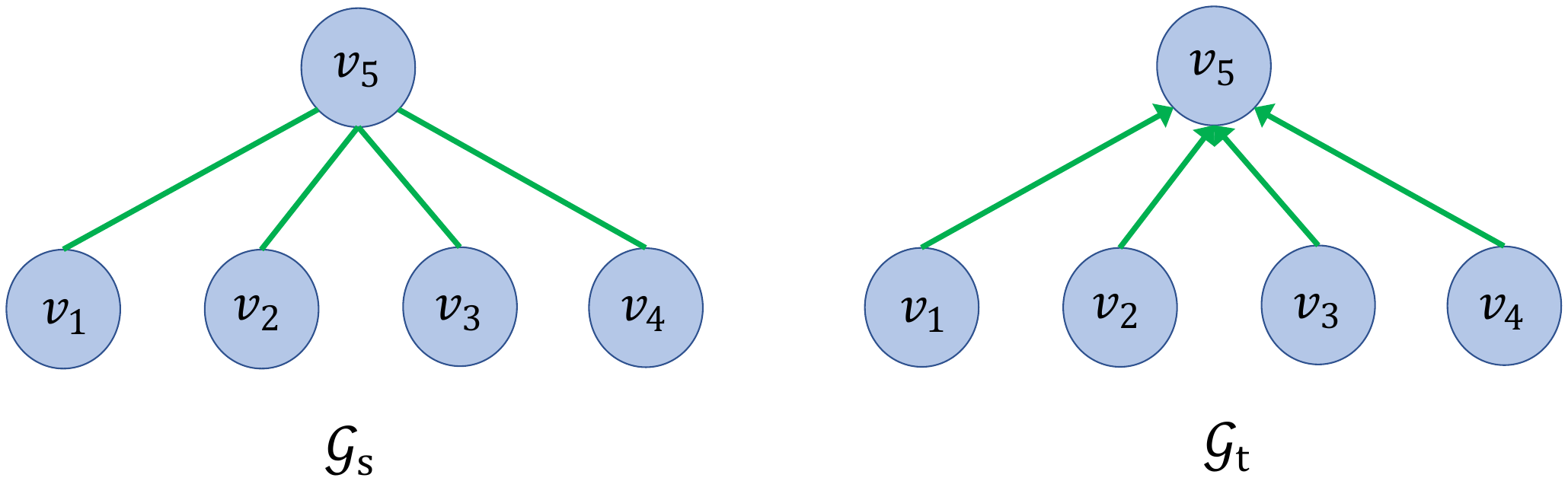}
  \caption{$d_{\Lp}=10$, $\SHD_1=4$, $\SHD_2=4$}
\end{subfigure}

\caption{Example graph pairs that exhibit large discrepancy between the model-oriented distance and the structural Hamming distance: (a-d) are CPDAGs, (e) are polytree MPDAGs.}
\label{fig:cpdag_results_4nodes}
\end{figure}

\end{appendix}

\end{document}